\documentclass[12pt,a4paper]{amsart}
\usepackage[francais]{babel}
\usepackage{ucs}
\usepackage[utf8x]{inputenc}
\usepackage[all]{xy}
\usepackage[T1]{fontenc}
\usepackage{lmodern}
\usepackage{amsmath, amssymb,amsthm}
\usepackage{fullpage}

\numberwithin{equation}{section}
\newcommand{\C}{\mathbb{C}}

\newcommand{\cind}{\mathrm{ind}}
\newcommand{\coker}{\mathrm{Coker}}
\renewcommand{\d}{\mathrm{d}}
\newcommand{\End}{\mathrm{End}}
\newcommand{\Fil}{\mathrm{Fil}}
\newcommand{\FS}{\mathrm{FS}}

\newcommand{\frob}{\mathrm{Frob}}
\newcommand{\GL}{\mathrm{GL}}
\newcommand{\Hom}{\mathrm{Hom}}
\newcommand{\Id}{\mathrm{Id}}
\newcommand{\Ind}{\mathrm{Ind}}

\newcommand{\Q}{\mathbb{Q}}
\newcommand{\Qbar}{\bar{\mathbb{Q}}}
\newcommand{\rbar}{\bar{r}}
\newcommand{\rec}{\mathrm{rec}}
\newcommand{\Res}{\mathrm{Res}}
\newcommand{\rhobar}{\bar{\rho}}
\newcommand{\Sp}{\mathrm{Sp}}
\newcommand{\Spec}{\mathrm{Spec}}
\newcommand{\Spf}{\mathrm{Spf}}
\newcommand{\Z}{\mathbb{Z}}

\newcommand{\dbl}{{\mathchoice{\mbox{\rm [\hspace{-0.15em}[}}
                              {\mbox{\rm [\hspace{-0.15em}[}}
                              {\mbox{\scriptsize\rm [\hspace{-0.15em}[}}
                              {\mbox{\tiny\rm [\hspace{-0.15em}[}}}}
\newcommand{\dbr}{{\mathchoice{\mbox{\rm ]\hspace{-0.15em}]}}
                              {\mbox{\rm ]\hspace{-0.15em}]}}
                              {\mbox{\scriptsize\rm ]\hspace{-0.15em}]}}
                              {\mbox{\tiny\rm ]\hspace{-0.15em}]}}}}

\newtheorem{theo}{Th\'eor\`eme}[section]
\newtheorem{conj}[theo]{Conjecture}
\newtheorem{coro}[theo]{Corollaire}
\newtheorem{defi}[theo]{D\'efinition}
\newtheorem{lemm}[theo]{Lemme}
\newtheorem{prop}[theo]{Proposition}
\newtheorem{rema}[theo]{Remarque}

\newtheorem{ex}{Exemple}[section]

\usepackage{color}
\newcounter{commentcounter}

\author{Christophe Breuil, Eugen Hellmann et Benjamin Schraen}
\address{Christophe Breuil \\
C.N.R.S.\\
Universit\'e Paris-Sud\\
Facult\'e d'Orsay\\
B\^atiment 425\\
F-91405 Orsay Cedex\\
France\\
Christophe.Breuil@math.u-psud.fr
}

\address{Eugen Hellmann\\
Mathematisches Institut\\
Universit\"at Bonn\\
Endenicher Allee 60\\
D-53115 Bonn \\
Germany\\
hellmann@math.uni-bonn.de}

\address{Benjamin Schraen\\
C.N.R.S.\\ 
Laboratoire de Math\'ematique\\
Universit\'e de Versailles St.~Quentin\\
45 Avenue des \'Etats-Unis\\
F-78035 Versailles\\
France\\
benjamin.schraen@uvsq.fr
}

\title{Une interpr\'etation modulaire de la vari\'et\'e trianguline}

\begin{document}
\begin{abstract}
En utilisant le syst\`eme de Taylor-Wiles-Kisin construit dans un travail r\'ecent de Caraiani, Emerton, Gee, Geraghty, Pa{\v{s}}k{\=u}nas et Shin, nous construisons un analogue de la vari\'et\'e de Hecke. Nous montrons que cette vari\'et\'e co\"incide avec une union de composantes irr\'eductibles de l'espace des repr\'esentations galoisiennes triangulines. Nous pr\'ecisons les relations de cette construction avec les conjectures de modularit\'e dans le cas cristallin ainsi qu'avec une conjecture de Breuil sur le socle des vecteurs localement analytiques de la cohomologie compl\'et\'ee. Nous donnons \'egalement une preuve d'une conjecture de Bella\"iche et Chenevier sur l'anneau local compl\'et\'e en certains points des vari\'et\'es de Hecke.

\bigskip\noindent
{\scshape Abstract.}
Using a patching module constructed in recent work of Caraiani, Emerton, Gee, Geraghty, Pa{\v{s}}k{\=u}nas and Shin we construct some kind of analogue of an eigenvariety. We can show that this \emph{patched eigenvariety} agrees with a union of irreducible components of a space of trianguline Galois representations. Building on this we discuss the relation with the modularity conjectures for the crystalline case, a conjecture of Breuil on the locally analytic socle of representations occurring in completed cohomology and with a conjecture of Bella\"iche and Chenevier on the complete local ring at certain points of eigenvarieties.
\end{abstract}

\maketitle
\tableofcontents

\vspace{1cm}

\section{Introduction}
Soit $p>2$ un nombre premier. Cet article a pour but d'utiliser la m\'ethode des syst\`emes de Taylor-Wiles telle que g\'en\'eralis\'ee par Kisin, puis par Caraiani-Emerton-Gee-Geraghty-Pa\v{s}k\=unas-Shin (\cite{CEGGPS}), pour \'etudier le probl\`eme des formes compagnons dans la th\'eorie des formes automorphes $p$-adiques sur les groupes unitaires compacts \`a l'infini.

Consid\'erons $F^+$ un corps de nombres totalement r\'eel, ainsi que $F$ une extension quadratique totalement imaginaire non ramifi\'ee de $F^+$, et $G$ un groupe unitaire sur $F^+$ d\'eploy\'e par $F$, compact \`a l'infini et isomorphe \`a $\GL_n$ au-dessus de toutes les places divisant $p$. Soit $L$ une extension finie de $\Q_p$.
Si $U^p$ d\'esigne un sous-groupe compact ouvert de $G(\mathbb{A}_{F^+}^{p\infty})$, on d\'esigne par $\widehat S(U^p,L)$ l'espace des fonctions continues $G(F^+)\backslash G(\mathbb{A}_{F^+}^\infty)/U^p\rightarrow L$. C'est un $L$-espace de Banach muni d'une action continue et unitaire du groupe $G(F^+\otimes_{\Q}\Q_p)$. On note $\widehat S(U^p,L)^{\rm an}\subset\widehat S(U^p,L)$ le sous-espace des vecteurs localement analytiques pour l'action du groupe de Lie $p$-adique $G(F^+\otimes_{\Q}\Q_p)$. Si $\mathfrak{p}$ est un id\'eal maximal de la \og grosse\fg\, alg\`ebre de Hecke $\widehat{\mathbb{T}}(U^p,\mathcal{O}_L)[\tfrac{1}{p}]$, et $\delta$ un caract\`ere continu du groupe $T(F^+\otimes_{\Q}\Q_p)$, on appelle forme surconvergente de poids $\delta$ associ\'ee \`a $\mathfrak{p}$ un vecteur propre, de caract\`ere propre $\delta$, du mono\"ide d'Atkin-Lehner $T(F^+\otimes_{\Q}\Q_p)^+$ agissant sur $\widehat S(U^p,L)[\mathfrak{p}]^{\rm an,N_0}$ o\`u $N_0$ est un sous-groupe compact ouvert stable par $T(F^+\otimes_{\Q}\Q_p)^+$ du groupe des matrices unipotentes sup\'erieures $N(F^+\otimes_{\Q}\Q_p)$. Soit $\rho$ la repr\'esentation $p$-adique, de dimension $n$, du groupe de Galois absolu de $F$ associ\'ee \`a l'id\'eal $\mathfrak{p}$. Dans \cite{BreuilAnalytiqueII}, l'un d'entre nous (C.~B.) donne une description conjecturale des caract\`eres $\delta$ pour lesquels il existe des formes surconvergentes associ\'ees \`a $\mathfrak{p}$ dans le cas o\`u toutes les repr\'esentations $\rho|_{\mathcal{G}_{F_{\tilde{v}}}}$ obtenues par restriction aux groupes de d\'ecomposition au-dessus de $p$ sont potentiellement cristallines et \og g\'en\'eriques\fg. Cette conjecture a \'et\'e par la suite \'etendue par Hansen (\cite{Hansen}) au cas o\`u certaines $\rho|_{\mathcal{G}_{F_{\tilde{v}}}}$ peuvent \^etre triangulines.

Un des r\'esultats de cet article est que l'existence des formes surconvergentes associ\'ees \`a tous les $\delta$ d\'ecrits dans \cite{BreuilAnalytiqueII} dans le cas cristallin est en fait une cons\'equence des conjectures de modularit\'e dans le cas cristallin.

D\'esignons par $\rhobar$ la r\'eduction de $\rho$ modulo l'id\'eal maximal de $\mathcal{O}_L$. La recherche des caract\`eres $\delta$ associ\'es \`a un id\'eal $\mathfrak{p}$ comme ci-dessus, ou de fa\c{c}on \'equivalente, \`a la repr\'esentation $\rho$ \'equivaut \`a d\'eterminer les points de la $\rhobar$-composante $\mathcal{E}(U^p)_{\rhobar}$ de la vari\'et\'e de Hecke de niveau mod\'er\'e $U^p$ (voir par exemple \cite{CheFougere}) sur lesquels la (pseudo-)repr\'esentation canonique est isomorphe \`a $\rho$. De tels points sont appel\'es \emph{points compagnons}. La vari\'et\'e de Hecke $\mathcal{E}(U^p)_{\rhobar}$ est une sous-vari\'et\'e rigide analytique ferm\'ee du produit $\mathfrak{X}_{\rhobar,S}\times\widehat T(F^+\otimes_{\Q}\Q_p)$ o\`u $\mathfrak{X}_{\rhobar,S}$ est un espace analytique rigide sur $L$ param\'etrant certains relev\'es de $\rhobar$ en caract\'eristique $0$ et $\widehat T(F^+\otimes_{\Q}\Q_p)$ est l'espace analytique rigide param\'etrant les caract\`eres continus du groupe $T(F^+\otimes_{\Q}\Q_p)$. La vari\'et\'e $\mathcal{E}(U^p)_{\rhobar}$ est ainsi munie d'une application \emph{poids} $\omega:\,\mathcal{E}(U^p)_{\rhobar}\rightarrow\widehat T(\mathcal{O}_{F^+}\otimes_{\Z}\Z_p)$. Le but de cet article est de construire, au moyen de syst\`emes de Taylor-Wiles-Kisin de niveau infini en $p$ construits dans \cite{CEGGPS}, un espace analytique rigide $X_p(\rhobar)$ sur $L$ muni d'une application \emph{poids} $\omega_\infty:\,X_p(\rhobar)\rightarrow \widehat T(\mathcal{O}_{F^+}\otimes_{\Z}\Z_p)\times\mathfrak{X}_{S_\infty}$, o\`u $\mathfrak{X}_{S_\infty}$ est isomorphe \`a un produit de boules unit\'e ouvertes, et un diagramme commutatif cart\'esien dans la cat\'egorie des espaces analytiques rigides r\'eduits
\begin{equation}\label{diagramme}
\begin{aligned}
\xymatrix{\mathcal{E}(U^p)_{\rhobar}\ar@{^{(}->}[r]\ar^{\omega}[d]&X_p(\rhobar)\ar^{\omega_\infty}[d]\\ \widehat T(\mathcal{O}_{F^+}\otimes_{\Z}\Z_p)\ar@{^{(}->}[r]& \widehat T(\mathcal{O}_{F^+}\otimes_{\Z}\Z_p)\times\mathfrak{X}_{S_\infty}}
\end{aligned}
\end{equation}
o\`u la fl\`eche horizontale du bas provient de l'inclusion d'un point ferm\'e dans $\mathfrak{X}_{S_\infty}$. Par construction, l'espace rigide analytique $X_p(\bar\rho)$ est un sous-espace analytique ferm\'e de l'espace produit
\[\prod_{v|p}(\mathfrak{X}^\square_{\rhobar_{\tilde{v}}}\times \widehat T(F_v^+))\times\mathbb{U}^g\]
o\`u $\mathfrak{X}^\square_{\rhobar_{\tilde{v}}}$ d\'esigne l'espace des rel\`evements (cadr\'es) de la repr\'esentation r\'esiduelle $\rhobar_{\tilde{v}}=\rhobar|_{\mathcal{G}_{F_{\tilde{v}}}}$ et $\mathbb{U}^g$ est un certain produit de boules unit\'es ouvertes. Il est en fait n\'ecessaire, dans tout ce qui pr\'ec\`ede de supposer $p>2n+1$ et $\rhobar|_{\mathcal{G}_{F(\zeta_p)}}$ absolument irr\'eductible dans l'utilisation des r\'esultats de \cite{CEGGPS}.

Si $v|p$ notons $X_{\rm tri}^\square(\rhobar_{\tilde{v}})$ le plus petit sous-espace analytique ferm\'e du produit $\mathfrak{X}^\square_{\rhobar_{\tilde{v}}}\times\widehat T(F_v^+)$ contenant tous les couples $(r,\delta)$ o\`u $r$ est un rel\`evement (cadr\'e) triangulin de $\rhobar_{\tilde{v}}$ et $\delta$ un param\`etre de $r$. Cet espace a \'et\'e introduit dans un cadre un peu diff\'erent par l'un d'entre nous (E.~H.) dans \cite{HellmannFS} suivant une id\'ee de Kisin (\cite{Kisinoverconvergent}). Nous prouvons alors le r\'esultat suivant (Th\'eor\`eme \ref{egalitecomposante} et Corollaire \ref{reduit}) :

\begin{theo}\label{principal}
L'espace $X_p(\rhobar)$ est r\'eduit et s'identifie \`a une union de composantes irr\'eductibles de $\prod_{v|p}X_{\rm tri}^\square(\rhobar_{\tilde{v}})\times\mathbb{U}^g$.
\end{theo}

Au passage nous prouvons \'egalement le r\'esultat suivant sur la structure de l'espace $X_{\rm tri}^\square(\rhobar_{\tilde{v}})$ (Th\'eor\`eme \ref{ouvertsature}).

\begin{theo}\label{equidimintro}
L'espace $X_{\rm tri}^\square(\rhobar_{\tilde{v}})$ est \'equidimensionnel de dimension $n^2+[F_v^+:\Q_p]\tfrac{n(n+1)}{2}$.
\end{theo}
Nous renvoyons le lecteur \`a l'\'enonc\'e du Th\'eor\`eme \ref{ouvertsature} pour des propri\'et\'es plus fines de l'espace $X_{\rm tri}^\square(\rhobar)$. La preuve de ce r\'esultat s'appuie essentiellement sur les travaux de Kedlaya, Pottharst et Xiao (\cite{KPX}).

Nous introduisons alors la notion de composante irr\'eductible \emph{automorphe} de l'espace $\prod_{v|p}X_{\rm tri}^\square(\rhobar_{\tilde{v}})$, il s'agit d'une composante irr\'eductible $X$ de ce produit telle que $X\times\mathbb{U}^g$ apparaisse dans $X_p(\rhobar)$. Il nous semble alors naturel de conjecturer (Conjecture \ref{conjcomp}) :

\begin{conj}\label{conjintro}
Une composante irr\'eductible de $\prod_{v|p}X_{\rm tri}^\square(\rhobar_{\tilde{v}})$ est une composante irr\'eductible automorphe si et seulement si son ouvert de lissit\'e (plus pr\'ecis\'ement son intersection avec l'ouvert $\prod_{v|p}U_{\rm tri}^\square(\rhobar_{\tilde{v}})^{\rm reg}$ de $\prod_{v|p}X_{\rm tri}^\square(\rhobar_{\tilde{v}})$ d\'efini dans la section 2) contient un point cristallin.
\end{conj}

Nous prouvons alors le r\'esultat suivant (Propositions \ref{implicationconj} et \ref{reciproqueconj}) :

\begin{theo}\label{conjectures}
La Conjecture \ref{conjintro} est \'equivalente aux conjectures de modularit\'e dans le cas cristallin (plus pr\'ecis\'ement \`a la Conjecture \ref{modularite}).
\end{theo}

Nous prouvons par ailleurs que cette conjecture implique une partie d'une conjecture de l'un d'entre nous (C.~B.) concernant le module de Jacquet-Emerton des vecteurs localement analytiques de la cohomologie compl\'et\'ee (cf.~Corollaire \ref{remBreuilAnalytiqueII}).

Dans la derni\`ere partie, nous utilisons cette construction pour prouver un grand nombre de cas d'une conjecture de Bella\"iche et Chenevier sur la structure infinit\'esimale en certains points de la vari\'et\'e de Hecke $Y$ d'un groupe unitaire $G$. On peut voir cette conjecture comme une \emph{conjecture $R=T$ infinit\'esimale}. Nous renvoyons le lecteur au corps du texte pour tout ce qui concerne les notations et conventions relatives \`a ces vari\'et\'es de Hecke.

\begin{theo}\label{theoR=T}
Supposons $p>2n+1$ et $\rhobar|_{\mathcal{G}_{F(\zeta_p)}}$ absolument irr\'eductible. Soit $z=(\pi,\chi)$ un point de la vari\'et\'e de Hecke $\mathcal{E}(U^p)_{\rhobar}$ correspondant \`a une repr\'esentation automorphe $\pi$ de $G(\mathbb{A})$. Supposons qu'en restriction \`a tout groupe de d\'ecomposition au-dessus de $p$ la repr\'esentation galoisienne $\rho_\pi$ associ\'ee \`a $\pi$ soit cristalline et que les valeurs propres de son Frobenius cristallin $\varphi$ soient deux-\`a-deux distinctes. Sous ces hypoth\`eses, le caract\`ere $\chi$ d\'efinit un drapeau de $D_{\rm cris}(\rho_\pi)$ stable par $\varphi$ que nous supposons \^etre en position g\'en\'erale relativement \`a la filtration de de Rham. Alors l'anneau local compl\'et\'e $\hat{\mathcal{O}}_{\mathcal{E}(U^p)_{\rhobar},z}$ pro-repr\'esente le foncteur des d\'eformations triangulines de $\rho$ munies d'une triangulation d\'eformant la triangulation d\'efinie par le drapeau $\varphi$-stable de $D_{\rm cris}(\rho_\pi)$.
\end{theo}

Le plan de l'article est le suivant. Nous commen\c{c}ons dans la section \ref{deformations} par d\'efinir les espaces locaux $X_{\rm tri}^\square(\rbar)$ pour $\rbar$ une repr\'esentation continue du groupe de Galois d'une extension finie de $\Q_p$ sur un espace vectoriel de dimension finie et de caract\'eristique $p$. Nous prouvons alors le Th\'eor\`eme \ref{equidimintro} et donnons un certain nombre de r\'esultats additionnels sur les points cristallins qui apparaissent dans cet espace. La section \ref{HTW} contient la construction de l'espace $X_p(\rhobar)$ ainsi que la preuve de ses propri\'et\'es g\'eom\'etriques. Les Th\'eor\`emes \ref{principal} et \ref{conjectures}, font l'objet de la section \ref{sectioncomposantes}. La section \ref{Global} fait le lien avec les vari\'et\'es de Hecke. Nous y prouvons le caract\`ere cart\'esien du diagramme \ref{diagramme} ainsi que le Th\'eor\`eme \ref{theoR=T}.

\emph{Remerciements :} Nous remercions Ahmed Abbes pour avoir r\'epondu \`a nos questions. Une partie de cet article a \'et\'e \'ecrite lors de diff\'erents s\'ejours dans les institutions suivantes : I.H.\'E.S., Math.~Institut de l'Universit\'e de Bonn et M.S.R.I. En particulier B.~S. souhaite remercier l'I.H.\'E.S. pour son hospitalit\'e prolong\'ee lors de la r\'edaction de cet article, ainsi que le B.I.C.M.R. de P\'ekin. Par ailleurs, C.~B. et B.~S. sont membres du projet ANR Th\'eHopaD (ANR-2011-BS01-005), et E.~H. est membre du SFB -TR 45 de la DFG.

\emph{Notations :} 
On fixe un nombre premier $p$, une extension alg\'ebrique $\bar{\Q}_p$ de $\Q_p$, ainsi qu'un isomorphisme entre $\bar{\Q}_p$ et $\mathbb{C}$.

Si $L$ est un corps local non archim\'edien localement compact, on note $|\cdot|_L$ sa valeur absolue normalis\'ee, c'est-\`a-dire telle que la norme d'une uniformisante de $L$ soit l'inverse du cardinal du corps r\'esiduel de $L$. On note $|\cdot|_p$ la valeur absolue sur $\bar{\Q}_p$ telle que $|p|_p=p^{-1}$ (donc $|\cdot|_L=|\cdot|_p^{[L:\Q_p]}$ sur $L$).

Si $K$ est un corps, on note $\mathcal{G}_K=\mathrm{Gal}(\bar{K}/K)$ le groupe de Galois absolu de $K$. Lorsque $K$ est de plus une extension finie de $\Q_p$, on note $\rec_K:\, K^\times\rightarrow \mathcal{G}_K^{\rm ab}$ le morphisme de r\'eciprocit\'e de la th\'eorie du corps de classes normalis\'e de telle sorte que l'image d'une uniformisante soit un \'el\'ement de Frobenius g\'eom\'etrique. On note $\varepsilon$ le caract\`ere $p$-adique cyclotomique de $\mathcal{G}_K$, de sorte que $\varepsilon\circ\rec_K$ est le caract\`ere $x\mapsto N_{K/\Q_p}(x)|x|_K$ de $K^\times$. On adopte comme convention que le poids de Hodge-Tate de $\varepsilon$ est $1$.

Si un groupe alg\'ebrique est d\'esign\'e par une lettre majuscule romaine $(G,H,N,T,\dots)$ on d\'esigne par la lettre minuscule gothique correspondante $(\mathfrak{g},\mathfrak{h},\mathfrak{n},\mathfrak{t},\dots)$ son alg\`ebre de Lie.

Fixons bri\`evement quelques conventions concernant le groupe $\GL_n$. On note $B$ le sous-groupe de Borel des matrices triangulaires sup\'erieures, et $T$ son sous-tore maximal constitu\'e des matrices diagonales. Ce choix d\'etermine un ensemble $\Delta=\{\alpha_1,\cdots,\alpha_{n-1}\}$ de racines simples, $\alpha_i$ d\'esignant le caract\`ere alg\'ebrique $(t_1,\cdots,t_n)\mapsto t_it_{i+1}^{-1}$. Pour $1\leq i\leq n$, on note $\beta_i$ le cocaract\`ere $t\mapsto (\underbrace{t,\cdots,t}_i,1,\cdots,1)$, de telle sorte que l'on ait $\alpha_i\circ \beta_j(t)=t^{\delta_{i,j}}$. Nous convenons que la notion de dominance pour un caract\`ere du tore $T$ ou de son alg\`ebre de Lie $\mathfrak{t}$ est toujours relative \`a ce choix de racines simples. Si $I\subset\{1,\cdots,n-1\}$, on note $P_I$ le sous-groupe parabolique contenant $B$ correspondant \`a l'ensemble $\{\alpha_i\ |\ i\in I\}$, et $\overline{P}_I$ le sous-groupe parabolique oppos\'e, on d\'esigne par $L_I$ leur sous-groupe de Levi commun et $W_I\subseteq W$ son groupe de Weyl. Rappelons que si $M$ d\'esigne un module de plus haut poids de $U(\mathfrak{gl}_n)$, relativement \`a l'alg\`ebre de Lie $\mathfrak{b}$, le sous-groupe $P_I$ est dit adapt\'e \`a $M$ si l'action de $\mathfrak{p}_I$ sur $M$ se prolonge en une repr\'esentation alg\'ebrique de $P_I$. Si $\rho$ d\'esigne la demi-somme des racines positives, $\lambda$ un caract\`ere de $\mathfrak{t}$ et $w$ un \'el\'ement du groupe de Weyl de $G$, on d\'efinit $w\cdot\lambda=w(\lambda+\rho)-\rho$. Si $\lambda$ est un poids dominant entier de $\mathfrak{t}$, le $U(\mathfrak{g})$-module simple de plus haut poids $w\cdot\lambda$ est adapt\'e \`a $\mathfrak{p}_I$, si et seulement si $w\cdot \lambda$ est un poids dominant de $L_I$, ce qui est encore \'equivalent \`a ce que $w\in{}^IW$ l'ensemble des repr\'esentants de $W_I\backslash W$ de longueur minimale (\cite[\S\ 9.4]{HumBGG}).

Soit $L$ et $K$ deux extensions finies de $\Q_p$, on note $K_0$ l'extension maximale non ramifi\'ee sur $\Q_p$ contenue dans $K$ et $K'_0$ l'extension maximale non ramifi\'ee sur $\Q_p$ contenue dans $K(\{\zeta_{p^n}|\,n\geq 1\})$ o\`u $\zeta_{p^n}$ est un syst\`eme compatible de racines $p^n$-i\`emes de l'unit\'e. On note de plus $\mathcal{R}_K$ l'anneau de Robba de $K$ et $\mathcal{R}_{L,K}=L\otimes_{\Q_p}\mathcal{R}_K$. On rappelle que $\mathcal{R}_K$ est l'anneau des s\'eries formelles $f=\sum_{n\in\Z}a_nz^n$ avec $a_n\in K'_0$ convergeant sur une couronne $r(f)\leq |z|_p <1$ (\cite[\S\ 2.2]{KPX}). 

Si $H$ est un groupe de Lie $p$-adique ab\'elien, on note $\widehat{H}$ l'espace analytique rigide sur $\Q_p$ repr\'esentant le foncteur associant \`a un espace affino\"ide $X=\Sp(A)$ le groupe $\Hom_{\rm cont}(H,A^\times)$ des morphismes de groupes continus $H\rightarrow A^\times$ (voir par exemple \cite[\S\ 6.4]{Emertonlocan}). Si $X$ est un espace analytique rigide sur $\Q_p$ et $L$ une extension finie de $\Q_p$, on note $X_L=\Sp(L)\times_{\Sp(\Q_p)}X$.

\section{Espaces de d\'eformations}\label{deformations}

Soit $L$ une extension finie de $\Q_p$ dans $\Qbar_p$. On note $\mathcal{O}_L$ son anneau d'entiers et $k_L$ son corps r\'esiduel. On fixe $\varpi_L$ une uniformisante de $\mathcal{O}_L$.

\subsection{Quelques d\'efinitions de g\'eom\'etrie analytique rigide}

Soit $X$ un espace analytique rigide sur $L$. On appelle ferm\'e analytique de $X$ un ``analytic subset'' au sens de \cite[\S\ 9.5.2]{BGR} (appel\'e aussi ``analytic set'' dans \cite{ConradIrred}).

\begin{defi}
Une partie $U\subset X$ est un ouvert de Zariski si son compl\'ementaire est un ferm\'e analytique de $X$. 
\end{defi}

Un ouvert de Zariski est en particulier un ouvert admissible que l'on munit de la structure d'espace analytique rigide sur $L$ induite par $X$.

\begin{defi}
(i) Une partie $Z$ de $X$ est dense au sens de Zariski dans $X$, ou Zariski-dense dans $X$, si le plus petit ferm\'e analytique de $X$ contenant $Z$ est l'espace sous-jacent \`a $X$.\\
(ii) Si $A$ et $B$ sont deux parties de $X$, on dit que $A$ s'accumule en $B$ si pour tout $b\in B$ et pour tout voisinage affino\"ide $U$ de $b$, il existe un voisinage affino\"ide $V\subset U$ de $b$ tel que $A\cap V$ est Zariski-dense dans $V$. Un ensemble de $X$ est dit d'accumulation s'il s'accumule en lui-m\^eme (cf.~\cite[p.~979]{CheFougere}).
\end{defi}

Les composantes irr\'eductibles d'un espace analytique rigide sont d\'efinies en \cite[Def. 2.2.2]{ConradIrred}. On v\'erifie facilement en utilisant \cite[Lem. 2.2.3]{ConradIrred} que, si $X$ est un espace analytique rigide irr\'eductible et $U\subset X$ un ouvert de Zariski non vide de $X$, l'espace analytique rigide $U$ est connexe.

Soit $\FS_{\mathcal{O}_L}$ la cat\'egorie des sch\'emas formels $\mathcal{X}$ sur $\mathcal{O}_L$, localement noeth\'eriens et dont la r\'eduction modulo un id\'eal de d\'efinition est un sch\'ema localement de type fini sur $k_L$. On note $\mathcal{X}\mapsto\mathcal{X}^{\rm rig}$ le foncteur construit par Raynaud et Berthelot de la cat\'egorie $\FS_{\mathcal{O}_L}$ vers la cat\'egorie des espaces analytiques rigides sur $L$ (cf. \cite[\S\ 7]{DeJong}, \cite[\S\ 0.2]{Berthelot}).

Si $R$ est une $\mathcal{O}_L$-alg\`ebre locale, plate, compl\`ete, noeth\'erienne et de corps r\'esiduel fini (sur $k_L$), de sorte que le sch\'ema formel $\Spf(R)$ est un objet de $\FS_{\mathcal{O}_L}$, il existe une bijection canonique entre l'ensemble des composantes irr\'eductibles du sch\'ema $\Spec(R[\tfrac{1}{p}])$ et l'ensemble des composantes irr\'eductibles de l'espace analytique rigide $\Spf(R)^{\rm rig}$ d\'efinie de la fa\c con suivante. Une composante irr\'eductible de $\Spec(R[\tfrac{1}{p}])$ correspond \`a un id\'eal premier minimal $\mathfrak{p}$ de $R[\tfrac{1}{p}]$. On obtient alors une composante irr\'eductible de $\Spf(R)^{\rm rig}$ en consid\'erant l'image du morphisme $\Spf(R/(\mathfrak{p}\cap R))^{\rm rig}\rightarrow\Spf(R)^{\rm rig}$. Le caract\`ere bijectif de cette application est une cons\'equence de \cite[Thm. 2.3.1]{ConradIrred} et du fait que les id\'eaux premiers minimaux de $R$ sont en bijection avec les composantes connexes de la normalisation de $\Spf(R)$.

Si $R$ est une $\mathcal{O}_L$-alg\`ebre comme ci-dessus, mais suppos\'ee de plus normale, par \cite[Prop. 7.3.6]{DeJong} le morphisme $R[\tfrac{1}{p}]\rightarrow \mathcal{O}(\Spf(R)^{\rm rig})$ identifie l'anneau $R[\tfrac{1}{p}]$ \`a l'anneau des fonctions {\it born\'ees} sur $\Spf(R)^{\rm rig}$.

\begin{prop}\label{ouvertlisse}
Si $X$ est un espace analytique rigide r\'eduit, l'ensemble des points lisses de $X$ (\cite[\S\ 7.3.2]{BGR}) est un ouvert de Zariski de $X$ qui est de plus Zariski-dense dans $X$.
\end{prop}

Un espace analytique rigide $X$ est dit \'equidimensionnel de dimension $d$ si pour tout point $x\in X$, l'anneau local compl\'et\'e $\widehat{\mathcal{O}}_{X,x}$ de $X$ en $x$ est un anneau local \'equidimensionnel de dimension $d$. De m\^eme $X$ est dit sans composante immerg\'ee si pour tout point $x\in X$, tous les id\'eaux premiers associ\'es de $\widehat{\mathcal{O}}_{X,x}$ sont des id\'eaux premiers minimaux.

Pour $f:\,X\rightarrow Y$ un morphisme entre espaces analytiques rigides et $x\in X$, on dit que $f$ est \'etale en $x$ si la condition de \cite[D\'ef.~1.7.10 i)]{Huber} est v\'erifi\'ee.

Si $X$ est un espace analytique rigide r\'eduit, on note $\mathcal{O}_X^+$ le sous-faisceau de $\mathcal{O}_X$ des fonctions de norme plus petite que $1$.

\subsection{L'espace des repr\'esentations triangulines}\label{triangulines}

Nous rappelons ici la construction de l'espace des d\'eformations triangulines d'une repr\'esentation galoisienne en caract\'eristique $p$. 

Fixons $K$ une extension finie de $\Q_p$. On suppose que l'extension $L$ de $\Q_p$ est suffisamment grande pour qu'il existe $[K:\Q_p]$ plongements de $K$ dans $L$, on note $\Hom(K,L)$ l'ensemble de ces plongements et $\mathcal{T}=\widehat{K^\times}$ (l'espace analytique rigide sur $\Q_p$ param\'etrant les caract\`eres continus de $K^{\times}$). Si ${\bf k}=(k_\tau)_{\tau\in\Hom(K,L)}\in\Z^{\Hom(K,L)}$, on note $x^{\bf k}$ l'\'el\'ement de $\mathcal{T}(L)$ d\'efini par $x\mapsto\prod_{\tau\in\Hom(K,L)}\tau(x)^{k_\tau}$. Soit $n$ un entier strictement positif. Pour ${\bf k}=(k_{\tau,i})_{1\leq i\leq n,\tau:\, K\hookrightarrow L}\in(\mathbb{Z}^n)^{\Hom(K,L)}$, on note $\delta_{\bf k}$ l'\'el\'ement de $\mathcal{T}^n(L)$ d\'efini par $(x_1,\cdots,x_n)\mapsto\prod_{i,\tau}\tau(x_i)^{k_{\tau,i}}$ (o\`u $(x_1,\cdots,x_n)\in (K^\times)^n$) et on s'autorise l'abus de langage consistant \`a d\'esigner par le m\^eme signe sa restriction \`a $(\mathcal{O}_K^\times)^n$. Un tel poids est dit dominant si, pour tout $\tau:\, K\hookrightarrow L$, on a $k_{\tau,1}\geq\cdots\geq k_{\tau,n}$. Si de plus, pour tout $\tau\in\Hom(K,L)$, on a $k_{\tau,1}>\cdots>k_{\tau,n}$, on dit que le poids est strictement dominant.

Rappelons que, si $L'$ d\'esigne une extension finie de $\Q_p$, les \'el\'ements de $\mathcal{T}(L')$ sont canoniquement en bijection avec l'ensemble des classes d'isomorphisme de $(\varphi,\Gamma_K)$-modules de rang $1$ sur $\mathcal{R}_{L',K}$ par la construction $\delta\mapsto\mathcal{R}_{L',K}(\delta)$ de \cite[Cons. 6.2.4]{KPX}.

Soit $\rbar$ une repr\'esentation continue de $\mathcal{G}_K$ sur un $k_L$-espace vectoriel $V_{k_L}$ de dimension $n$. Fixons une fois pour toute une base $\bar{b}$ de $V_{k_L}$, ce qui nous permet de d\'efinir la $\mathcal{O}_L$-alg\`ebre locale noeth\'erienne compl\`ete $R_{\rbar}^{\square}$ qui pro-repr\'esente le foncteur associant \`a une $\mathcal{O}_L$-alg\`ebre artinienne locale $A$ de corps r\'esiduel $k_L$ l'ensemble des classes d'\'equivalence de triplets $(V_A,\iota,b)$ o\`u $V_A$ est un $A$-module libre de rang $n$ muni d'une action continue $A$-lin\'eaire de $\mathcal{G}_K$, $\iota$ un isomorphisme $\mathcal{G}_K$-\'equivariant entre $V_A\otimes_A k_L$ et $V_{k_L}$, et $b$ une base de $V_A$ se r\'eduisant sur $\bar{b}$ \`a travers $\iota$. On note $\mathfrak{X}_{\rbar}^\square$ l'espace analytique rigide sur $L$ associ\'e au sch\'ema formel $\Spf (R_{\rbar}^\square)$. On note $\mathcal{T}_{\rm reg}$ le compl\'ementaire dans $\mathcal{T}_L=\mathcal{T}\times_{\Q_p}L$ de l'ensemble des $L$-points $x\mapsto x^{-\bf i}$ et $x\mapsto x^{\bf 1+i}|x|_K$ pour ${\bf i}\in\mathbb{Z}_{\geq 0}^{\Hom(K,L)}$ et $\mathcal{T}_{\rm reg}^n\subset\mathcal{T}_L^n$ l'ouvert de Zariski des points $(\delta_1,\cdots,\delta_n)$ tels que $\delta_i\delta_j^{-1}\in\mathcal{T}_{\rm reg}$ pour $i\neq j$. On d\'esigne par $U_{\rm tri}^\square(\rbar)^{\rm reg}$ l'ensemble des points $(x,\delta)\in \mathfrak{X}_{\rbar}^\square\times\mathcal{T}_{\rm reg}^n$ tels que, si $r_x$ d\'esigne la repr\'esentation de $\mathcal{G}_K$ associ\'ee \`a $x$, son $(\varphi,\Gamma_K)$-module $D_{\rm rig}^\dagger(r_x)$ est triangulin (voir \cite{Co}, \cite[\S\ 2]{BelChe}, \cite[\S\ 6]{KPX}) et $\delta=(\delta_1,\cdots,\delta_n)$ est un syst\`eme de param\`etres de $r_x$ (ou plus simplement un param\`etre de $r_x$), ce qui signifie qu'il existe une filtration croissante de $D_{\rm rig}^\dagger(r_x)$ par des sous-$\mathcal{R}_{k(x),K}$-modules stable par $\varphi$ et $\Gamma_K$ (n\'ecessairement facteurs directs comme $\mathcal{R}_{k(x),K}$-modules) dont les gradu\'es successifs sont isomorphes \`a $\mathcal{R}_{k(x),K}(\delta_1)$, $\cdots ,\mathcal{R}_{k(x),K}(\delta_n)$ comme $(\varphi,\Gamma_K)$-modules, o\`u $k(x)$ est le corps r\'esiduel de $\mathfrak{X}_{\rbar}^\square$ au point $x$. 

\begin{defi}\label{deftriangulines}
On appelle \emph{espace des d\'eformations triangulines de $\rbar$}, et on le note $X_{\rm tri}^\square(\rbar)$, l'espace analytique rigide sur $L$ obtenu comme adh\'erence de Zariski de $U_{\rm tri}^\square(\rbar)^{\rm reg}$ dans $\mathfrak{X}_{\rbar}^\square\times\mathcal{T}_L^n$, i.e. le ferm\'e analytique intersection des ferm\'es analytiques de $\mathfrak{X}_{\rbar}^\square\times\mathcal{T}_L^n$ contenant $U_{\rm tri}^\square(\rbar)^{\rm reg}$ muni de sa structure de vari\'et\'e analytique rigide r\'eduite (\cite[\S\ 9.5.3 Prop. 4]{BGR}).
\end{defi}

Remarquons que la vari\'et\'e analytique rigide $X_{\rm tri}^\square(\rbar)$ est s\'epar\'ee car ferm\'ee dans la vari\'et\'e analytique rigide s\'epar\'ee $\mathfrak{X}_{\rbar}^\square\times\mathcal{T}_L^n$ (cf. \cite[p. 393]{BGR}). 

Notons $\mathcal{W}=\widehat{\mathcal{O}_K^\times}$ (l'espace analytique rigide sur $\Q_p$ param\'etrant les caract\`eres continus de $\mathcal{O}_K^{\times}$) et $\mathcal{W}_L=\mathcal{W}\times_{\Q_p}L$. La restriction d'un caract\`ere de $K^\times$ \`a $\mathcal{O}_K^\times$ induit un morphisme d'espaces analytiques rigides $\mathcal{T}\rightarrow \mathcal{W}$ sur $\Q_p$. On note $\omega'$ le morphisme $X_{\rm tri}^\square(\rbar)\rightarrow\mathcal{T}_L^n$ obtenu en composant l'inclusion de $X_{\rm tri}^\square(\rbar)$ dans $\mathfrak{X}_{\rbar}^\square\times\mathcal{T}_L^n$ avec la seconde projection, on obtient une application \emph{poids} $\omega:\,X_{\rm tri}^\square(\rbar)\rightarrow\mathcal{W}_L^n$ en composant $\omega'$ avec le morphisme de restriction $\mathcal{T}_L^n\rightarrow\mathcal{W}_L^n$.

\begin{rema}
Si on note $X_{\rm tri}^\square(\rbar)^{\rm reg}$ l'image r\'eciproque de $\mathcal{T}_{\rm reg}^n$ par $\omega$, l'espace $X_{\rm tri}^\square(\rbar)^{\rm reg}$ est un ouvert de Zariski de $X_{\rm tri}^\square(\rbar)$. On a $U_{\rm tri}^\square(\rbar)^{\rm reg}\subset X_{\rm tri}^\square(\rbar)^{\rm reg}$ mais cette inclusion est toujours stricte. Il existe en effet des points $(r_x,\delta_x)$ de $X_{\rm tri}^\square(\rbar)$ tels que $\delta_x$ ne soit pas un param\`etre de la repr\'esentation $r_x$, m\^eme si $r_x$ est trianguline et $\delta_x$ est r\'egulier. Le Th\'eor\`eme 6.3.13 de \cite{KPX} montre cependant que $r_x$ est toujours trianguline.
\end{rema}

Le reste de cette section est presque enti\`erement consacr\'e \`a la preuve du r\'esultat suivant concernant la g\'eom\'etrie de l'espace $X_{\rm tri}^\square(\rbar)$.

\begin{theo}\label{ouvertsature}
(i) L'espace $X_{\rm tri}^\square(\rbar)$ est \'equidimensionnel de dimension $n^2+[K:\Q_p]\frac{n(n+1)}{2}$.\\
(ii) L'ensemble $U_{\rm tri}^\square(\rbar)^{\rm reg}$ est un ouvert de Zariski de $X_{\rm tri}^\square(\rbar)$ qui est Zariski-dense dans $X_{\rm tri}^\square(\rbar)$.\\
(iii) L'ouvert $U_{\rm tri}^\square(\rbar)^{\rm reg}$ est lisse sur $\Q_p$ et la restriction de l'application $\omega'$ \`a $U_{\rm tri}^\square(\rbar)^{\rm reg}$ est un morphisme lisse.
\end{theo}

\begin{proof}
La strat\'egie de la preuve est la suivante. Nous allons construire un espace analytique rigide lisse $\mathcal{S}^\square(\rbar)$ s'ins\'erant dans un diagramme commutatif comme ci-dessous
\begin{equation}\label{spaceofphigamma}
\begin{aligned}
\begin{xy}
\xymatrix{
&\mathcal{S}^\square(\rbar)\ar[ld]_{\pi_{\rbar}}\ar[rd]^\kappa&\\
X_{\rm tri}^\square(\rbar)\ar[rr]^{\omega'} && \mathcal{T}_L^n
}
\end{xy}
\end{aligned}
\end{equation}
o\`u le morphisme $\pi_{\rbar}$ est lisse de dimension relative $n$ et d'image $U_{\rm tri}^\square(\rbar)^{\rm reg}$ et le morphisme $\kappa$ est lisse de dimension relative $n^2+[K:\Q_p]\frac{n(n-1)}{2}$.

L'espace analytique rigide $\mathcal{S}^\square(\rbar)$ repr\'esente le foncteur
\begin{equation}\label{foncteurtriangulin}
X\mapsto \{(r,{\rm F}_\bullet,\delta,\nu)\}/\sim\end{equation}
o\`u $X$ est un espace analytique r\'eduit sur $L$ et $r$ est un morphisme continu $\mathcal{G}_K\rightarrow\GL_n(\mathcal{O}_X^+)$ tel qu'en tout point $x\in X$, la r\'eduction de $r\otimes \mathcal{O}_{k(x)}$ co\"incide avec $\rbar$. De plus ${\rm F}_\bullet$ est une filtration croissante de ${\bf D}^\dagger_{\rm rig}(r)$ par des sous-$(\varphi,\Gamma_K)$-modules localement sur $X$ facteurs directs comme $\mathcal{R}_X$-modules tels que $F_0=0$ et $F_n={\bf D}^\dagger_{\rm rig}(r)$ et $\delta=(\delta_1,\dots,\delta_n)\in \mathcal{T}_{\rm reg}^n(X)$. Finalement $\nu=(\nu_1,\dots,\nu_n)$ est une collection de trivialisations $\nu_i:{\rm gr}_i(D)\cong \mathcal{R}_X(\delta_i)$ (avec les notations de \cite[\S2]{HellmSchrDensity}). La fl\`eche $\kappa$ correspond alors au morphisme de foncteurs $(r,{\rm F}_\bullet,\delta,\nu)\mapsto \delta$.

Pour prouver la repr\'esentabilit\'e de ce foncteur, nous suivons la strat\'egie de \cite{Cheneviertriangulines}, \cite{HellmannFS} et \cite{HellmSchrDensity}.

Soit $\mathcal{S}_n$ l'espace analytique rigide sur $L$ repr\'esentant le foncteur
\begin{equation*} X\longmapsto\{(D,{\rm Fil}_\bullet(D),\delta,\nu)\}/\sim\end{equation*}
o\`u $X$ est comme avant, $D$ est un  $(\varphi,\Gamma_K)$-module sur $\mathcal{R}_X$ de rang $n$, et ${\rm Fil}_\bullet$ est une filtration croissante de $D$ par des sous-$(\varphi,\Gamma_K)$-modules localement sur $X$ facteurs directs comme $\mathcal{R}_X$-modules tels que $F_0=0$ et $F_n=D$. De plus $\delta=(\delta_1,\dots,\delta_n)\in \mathcal{T}_{\rm reg}^n(X)$ et $\nu=(\nu_1,\dots,\nu_n)$ est une collection de trivialisations $\nu_i:{\rm gr}_i(D)\cong \mathcal{R}_X(\delta_i)$. Ce foncteur est repr\'esentable par un espace analytique rigide s\'epar\'e d'apr\`es \cite[Theorem 2.4]{HellmSchrDensity}.
Soit $\mathcal{S}_n^{\rm adm}\subset \mathcal{S}_n$ le sous-ensemble admissible au sens de \cite[Theorem 1.2]{Hellm}, i.e.~le sous-espace ouvert adique\footnote{Nous utilisons ici le fait qu'un ouvert de l'espace adique associ\'e \`a un espace analytique rigide quasi-s\'epar\'e est quasi-s\'epar\'e et donc provient lui-m\^eme d'un espace analytique rigide d'apr\`es \cite[\S1.1.11]{Huber}.} maximal de $\mathcal{S}_n$ tel qu'il existe un fibr\'e vectoriel $\mathcal{V}$ sur $\mathcal{S}_n^{\rm adm}$ et un morphisme continu $\mathcal{G}_K\rightarrow{\rm Aut}_{\mathcal{O}_{\mathcal{S}_n^{\rm adm}}}(\mathcal{V})$ tels que ${\bf D}^\dagger_{\rm rig}(\mathcal{V})$ est l'image inverse du $(\varphi,\Gamma_K)$-module universel de $\mathcal{S}_n$ sur $\mathcal{S}_n^{\rm adm}$. Cet ouvert est non vide car toute repr\'esentation cristalline \`a poids de Hodge-Tate deux-\`a-deux distincts et Frobenius suffisamment g\'en\'erique permet d'en construire un point.
 
Soit $\pi:\mathcal{S}_n^{\square,{\rm adm}}\rightarrow \mathcal{S}_n^{\rm adm}$ le $\GL_n$-torseur des trivialisations du fibr\'e vectoriel $\mathcal{V}$. Il existe donc un isomorphisme $\beta:\,\pi^\ast\mathcal{V}\simeq\mathcal{O}_{\mathcal{S}_n^{\square,{\rm adm}}}^n$ sur $\mathcal{S}_n^{\square,{\rm adm}}$ et l'action de $\mathcal{G}_K$ sur $\mathcal{V}$ est donn\'ee par un morphisme continu \[\tilde r:\,\mathcal{G}_K\rightarrow\GL_n(\Gamma(\mathcal{S}_n^{\square,{\rm adm}},\mathcal{O}_{\mathcal{S}_n^{\square,{\rm adm}}})).\] Comme $\mathcal{G}_K$ est topologiquement engendr\'e par un nombre fini d'\'el\'ements (voir par exemple \cite[Prop.~II.14 et I.25 ]{SeCohGa}), l'ensemble des points $x$ de $\mathcal{S}_n^{\square,{\rm adm}}$ o\`u $\tilde r$ se factorise \`a travers \[\GL_n(\Gamma(\mathcal{S}_n^{\square,{\rm adm}},\mathcal{O}_{\mathcal{S}_n^{\square,{\rm adm}}}^+))\subset\GL_n(\Gamma(\mathcal{S}_n^{\square,{\rm adm}},\mathcal{O}_{\mathcal{S}_n^{\square,{\rm adm}}}))\] est un ouvert admissible de $\mathcal{S}_n^{\square,{\rm adm}}$.
Comme de plus $\mathcal{G}_K$ est compact, cet ouvert est non vide. Rappelons, que l'on a fix\'e une repr\'esentation r\'esiduelle $\rbar:\mathcal{G}_K\rightarrow \GL_n(k_L)$. On d\'efinit alors $\mathcal{S}^\square(\rbar)\subset \mathcal{S}^{\square,{\rm adm}}$ comme l'ouvert admissible non vide (en consid\'erant des relev\'es cristallins de $\rbar$) des points $x$ o\`u le morphisme $\tilde r_x$ se r\'eduit sur $\rbar$ modulo l'id\'eal maximal de $k(x)^+$. Il est alors imm\'ediat de v\'erifier que $\mathcal{S}^\square(\rbar)$ repr\'esente bien le foncteur \eqref{foncteurtriangulin}. Notons $r_X:\,\mathcal{G}_K\rightarrow\GL_n(\mathcal{O}_{\mathcal{S}^\square(\rbar)}^+)$ le morphisme universel sur $\mathcal{S}^\square(\rbar)$.

On d\'efinit alors le morphisme $\kappa:\mathcal{S}^\square(\rbar)\rightarrow \mathcal{T}_L^n$ comme la compos\'ee du morphisme $\mathcal{S}^\square(\rbar)\hookrightarrow\mathcal{S}_n^{\square,{\rm adm}}\rightarrow\mathcal{S}_n^{\rm adm}\hookrightarrow \mathcal{S}_n$, qui est lisse de dimension relative $n^2$, et de la projection canonique $\mathcal{S}_n\rightarrow \mathcal{T}_L^n$, qui est lisse de dimension relative $[K:\Q_p]\tfrac{n(n-1)}{2}$ par \cite[Theorem 2.4(i)]{HellmSchrDensity}. Comme $\mathcal{T}_L$ est lisse de dimension $n[K:\mathbb{Q}_p]+n$, on en conclut que $\mathcal{S}^\square(\rbar)$ est lisse et \'equidimensionnel de dimension $n^2+n+\frac{n(n+1)}{2}[K:\mathbb{Q}_p]$. On a \'egalement un morphisme $\mathcal{S}^\square(\rbar)\rightarrow \mathfrak{X}_{\rbar}^\square\times\mathcal{T}_L^n$ donn\'e par $(r,{\rm F}_\bullet,\delta,\nu)\mapsto (r,\delta)$. La description des points de $\mathcal{S}^\square(\rbar)$ montre que l'image de ce morphisme est exactement $U_{\rm tri}^\square(\rbar)^{\rm reg}$. On obtient ainsi un morphisme 
\[\pi_{\rbar}:\mathcal{S}^\square(\rbar)\longrightarrow X_{\rm tri}^\square(\rbar)\]
d'image $U_{\rm tri}^\square(\rbar)^{\rm reg}$.

Montrons tout d'abord que l'espace $X_{\rm tri}^\square(\rbar)$ est \'equidimensionnel, de dimension \[\dim X_{\rm tri}^\square(\rbar)=n^2+[K:\Q_p]\frac{n(n+1)}{2}.\]
D'apr\`es le Lemme 2.12 de \cite{HellmSchrDensity}, les points de $U_{\rm tri}^\square(\rbar)^{\rm reg}$ sont strictement triangulins au sens de \cite[Def.~6.3.1]{KPX}. Le Corollaire 6.3.10 de \emph{loc.~cit.} permet alors de construire un morphisme projectif et birationnel, donc surjectif, $f:\tilde X_{\rm tri}^\square(\rbar)\rightarrow X_{\rm tri}^\square(\rbar)$ et, si $D:={\bf D}^\dagger_{\rm rig}(f^\ast r_X)$, une filtration $\Fil_\bullet(D)\subset D$ de $D$ par des sous-$(\varphi,\Gamma)$-modules de $D$ ainsi que des suites exactes courtes pour $1\leq i\leq n$
\[0\longrightarrow {\rm gr}_i D\longrightarrow \mathcal{R}_{\tilde X_{\rm tri}^\square(\rbar)}(\delta_i)\otimes\mathcal{L}_i\longrightarrow \mathcal{M}_i\longrightarrow 0\]
o\`u $\mathcal{L}_i$ est un fibr\'e en droites sur $\tilde X_{\rm tri}^\square(\rbar)$ et $\mathcal{M}_i$ un $\mathcal{R}_{\tilde X_{\rm tri}^\square(\rbar)}$-module qui est (localement sur $\tilde X_{\rm tri}^\square(\rbar)$) tu\'e par une puissance de $t$ et de support un ferm\'e analytique $Z_i\subset \tilde X_{\rm tri}^\square(\rbar)$ ne contenant aucune composante irr\'eductible de $\tilde X_{\rm tri}^\square(\rbar)$. On a de plus $f^{-1}(U_{\rm tri}^\square(\rbar)^{\rm reg})\subset\tilde X_{\rm tri}^\square(\rbar)\backslash\bigcup_i Z_i$.

Soit $U\subset\tilde X_{\rm tri}^\square(\rbar)$ l'intersection du compl\'ementaire de $\bigcup_i Z_i$ dans $\tilde X_{\rm tri}^\square(\rbar)$ avec l'image r\'eciproque de $\mathcal{T}^n_{\rm reg}\subset\mathcal{T}^n_L$ via $\tilde X_{\rm tri}^\square(\rbar)\rightarrow X_{\rm tri}^\square(\rbar)\rightarrow\mathcal{T}^n_L$. Il s'agit d'un ouvert de Zariski de $\tilde X_{\rm tri}^\square(\rbar)$ qui contient $f^{-1}(U_{\rm tri}^\square(\rbar)^{\rm reg})$, ce qui prouve qu'il est aussi Zariski dense dans $\tilde X_{\rm tri}^\square(\rbar)$. Notons $U^\square\rightarrow U$ le $\mathbb{G}_m^n$-torseur des trivialisations des fibr\'es en droites $\mathcal{L}_i$. L'espace $U^\square$ poss\`ede par construction la propri\'et\'e suivante. Il existe des trivialisations canoniques des images inverses des faisceaux $\mathcal{L}_i$ sur $U^\square$ telles que, si $g:\,Y\rightarrow U$ est un morphisme, et $s_i$ des trivialisations $\mathcal{O}_Y\simeq g^\ast(\mathcal{L}_i)$, il existe un unique morphisme $Y\rightarrow U^\square$ relevant $g$ de sorte que les $s_i$ sont les images inverses des trivialisations canoniques sur $U^\square$. La propri\'et\'e universelle de $\mathcal{S}^\square(\rbar)$ nous permet par ailleurs de construire un morphisme $s:\,U^\square\rightarrow\mathcal{S}^\square(\rbar)$ tel que $\pi_{\rbar}\circ s$ est la compos\'ee \[U^\square\longrightarrow U\overset{f|_U}{\longrightarrow} X_{\rm tri}^\square(\rbar).\]
 Comme $f$ est la compos\'ee d'\'eclatements et de normalisations, on peut trouver un ouvert de Zariski et Zariski dense $V$ de $X_{\rm tri}^\square(\rbar)$ tel que $f^{-1}(V)\subset U$ et $f$ induise un isomorphisme de $f^{-1}(V)$ sur $V$. Notons de m\^eme $V^\square\rightarrow f^{-1}(V)$ le $\mathbb{G}_m^n$-torseur des trivialisations des faisceaux $\mathcal{L}_i|_{f^{-1}(V)}$, il est facile de v\'erifier que $V^\square$ n'est autre que l'image r\'eciproque de $f^{-1}(V)$ dans $U^\square$. La propri\'et\'e universelle de $V^\square\rightarrow V$ permet de construire un morphisme $\pi_{\rbar}^\square$ de $\pi_{\rbar}^{-1}(V)$ dans $V^\square$ dont la compos\'ee avec $V^\square\rightarrow f^{-1}(V)$ est $f^{-1}|_V\circ \pi_{\rbar}$. On a alors $\pi^\square_{\rbar}\circ s|_{V^\square}=\Id_{V^\square}$. Or, la description des points de $\mathcal{S}^\square(\rbar)$ montre que le morphisme $\pi^\square_{\rbar}$ de $\pi_{\rbar}^{-1}(V)$ dans $V^\square$ est injectif, on en conclut que l'on a donc \'egalement $s\circ \pi_{\rbar}^\square=\Id_{\pi_{\rbar}^{-1}(V)}$, ce qui implique, puisque tous les espaces sont r\'eduits, que $\pi_{\rbar}^\square$ induit un isomorphisme de $\pi_{\rbar}^{-1}(V)$ sur $V^\square$. Ainsi, $V^\square$ est \'equidimensionnel de dimension $n^2+n+\frac{n(n+1)}{2}[K:\Q_p]$, $V^\square\rightarrow V$ \'etant un $\mathbb{G}_m^n$-torseur, on en conclut que l'ouvert $V$ est \'equidimensionnel de dimension $n^2+[K:\Q_p]\frac{n(n+1)}{2}$. Comme $V$ est un ouvert dense de $X_{\rm tri}^\square(\rbar)$, il en est de m\^eme de $X_{\rm tri}^\square(\rbar)$, ceci prouve le point (i).

Prouvons \`a pr\'esent que l'application $\pi_{\rbar}$ de $\mathcal{S}^\square(\rbar)$ vers $X_{\rm tri}^\square(\rbar)$ est lisse de dimension relative $n$. Pour ce faire, il suffit de prouver que si $x\in \mathcal{S}^\square(\rbar)$ et $y=(r_y,\delta_y)=\pi_{\rbar}(x)\in U_{\rm tri}^\square(\rbar)^{\rm reg}$, il existe un isomorphisme d'anneaux locaux complets \[\widehat{\mathcal{O}}_{\mathcal{S}^\square(\rbar),x}\simeq\widehat{\mathcal{O}}_{X_{\rm tri}^\square(\rbar),y}\dbl x_1,\dots,x_n\dbr.\] Notons $A=\widehat{\mathcal{O}}_{X_{\rm tri}^\square(\rbar),y}$ et $B=\widehat{\mathcal{O}}_{\mathcal{S}^\square(\rbar),x}$ dans ce qui suit, de sorte que $\pi_{\rbar}$ induise un morphisme local d'anneaux locaux complets $A\rightarrow B$.

Le morphisme de projection $X_{\rm tri}^\square(\rbar)\rightarrow \mathfrak{X}^\square_{\rbar}$ donne lieu \`a un morphisme d'anneaux locaux complets \[\widehat{\mathcal{O}}_{\mathfrak{X}^\square_{\rbar},r_y}\longrightarrow\widehat{\mathcal{O}}_{X_{\rm tri}^\square(\rbar),y}.\] En utilisant \cite[Lem.~2.3.3 et Prop.~2.3.5]{KisinModularity}, on voit qu'il existe de plus un isomorphisme topologique entre $\widehat{\mathcal{O}}_{\mathfrak{X}^\square_{\rbar},r_y}$ et $R^\square_{r_y}$, l'anneau de d\'eformation universel cadr\'e de $r_y$, la repr\'esentation galoisienne correspondant au point $y$. Soit $\mathcal{F}$ l'unique triangulation du $(\varphi,\Gamma_K)$-module $D$ de $r_y$ correspondant au point $y$ et $R^\square_{r_y,\mathcal{F}}$ l'anneau de d\'eformation universel cadr\'e de la paire $(r_y,\mathcal{F})$ en un sens \'evident. Comme pr\'ec\'edemment, l'anneau $B=\widehat{\mathcal{O}}_{\mathcal{S}^\square(\rbar),x}$ est naturellement isomorphe \`a une $R^\square_{r_y,\mathcal{F}}$-alg\`ebre locale compl\`ete lisse de dimension relative $n$. Fixons donc $(x_1,\dots,x_n)$ une famille de g\'en\'erateurs topologiques de la $R^\square_{r_y,\mathcal{F}}$-alg\`ebre $B$. On d\'efinit alors une application continue $A$-lin\'eaire de $A\dbl X_1,\dots,X_n\dbr$ dans $B$ envoyant $X_i$ sur $x_i$. Comme l'application $R^\square_{r_y}\rightarrow B$ se factorise \`a travers l'anneau $A$, on obtient au final une suite de morphismes
\[R^\square_{r_y}\dbl X_1,\dots,X_n\dbr\rightarrow A\dbl X_1,\dots,X_n\dbr\rightarrow B.\]
D'apr\`es \cite[Prop.~2.3.6 et 2.3.9]{BelChe}, l'application naturelle $R^\square_{r_y}\rightarrow R^\square_{r_y,\mathcal{F}}$ est surjective. Comme l'anneau $B$ est engendr\'e par l'image de $R^\square_{r_y}$ et par les $x_1,\dots, x_n$, on en conclut que l'application $A\dbl X_1,\dots,X_n\dbr\rightarrow B$ est surjective. Ces deux anneaux locaux \'etant noeth\'eriens, complets et de m\^eme dimension, il suffit de prouver que $A\dbl X_1,\dots,X_n\dbr$ est int\`egre pour conclure que $A\dbl X_1,\dots,X_n\dbr\rightarrow B$ est un isomorphisme. Il suffit donc de prouver que l'anneau $A$ est int\`egre. Comme $X_{\rm tri}^\square(\rbar)$ est r\'eduit, on d\'eduit de \cite[\S7.2,Prop.~8]{BGR} que $A$ est r\'eduit. Il suffit donc de prouver que cet anneau n'a qu'un seul id\'eal premier minimal. Pour cela, il suffit de prouver que la fibre du normalis\'e $X_{\rm tri}^\square(\rbar)'$ de $X_{\rm tri}^\square(\rbar)$ au dessus de $y$ est un singleton. Remarquons que par la propri\'et\'e universelle du normalis\'e, le morphisme $\pi_{\rbar}$ se factorise de fa\c{c}on unique \`a travers $X_{\rm tri}^\square(\rbar)'$ puisque $\mathcal{S}^\square(\rbar)$ est lisse donc normal. Comme $\tilde X_{\rm tri}^\square(\rbar)$ est normal, il en est de m\^eme de $U$ et donc de $U^\square$. Ainsi le morphisme $U^\square\rightarrow X_{\rm tri}^\square(\rbar)$ se factorise de fa\c{c}on unique \`a travers $X_{\rm tri}^\square(\rbar)'$, de sorte que le diagramme ci-dessous est commutatif.
\[\begin{xy}
\xymatrix{
&&\mathcal{S}^\square(\rbar)\ar[d]^{\pi_{\rbar}}\ar[dl]\\
U^\square\ar[r] \ar@/^/[rru]^s & X_{\rm tri}^\square(\rbar)'\ar[r] & X_{\rm tri}^\square(\rbar)
}
\end{xy}
\]
Puisque $f^{-1}(U_{\rm tri}^\square(\rbar)^{\rm reg})\subset U$, l'existence de $s$ montre que tous les points de $X_{\rm tri}^\square(\rbar)'$ au dessus de $y\in U_{\rm tri}^\square(\rbar)^{\rm reg}$ sont dans l'image de $\mathcal{S}^\square(\rbar)$ dans $X_{\rm tri}^\square(\rbar)'$. Comme la fibre de $\pi_{\rbar}$ au-dessus de $y$ est connexe (car isomorphe \`a $\mathbb{G}_m^n$ puisqu'il y a seulement une triangulation de $r_y$ de param\`etre $\delta_y$), son image dans la fibre de $X_{\rm tri}^\square(\rbar)'$ au dessus de $y$ doit \'egalement \^etre connexe. Comme cette fibre est finie, il s'agit d'un singleton.

Il suffit \`a pr\'esent de prouver (ii). Pour ce faire, on raisonne comme dans \cite[Thm.~2.11]{HellmSchrDensity}. Comme le morphisme $\pi_{\rbar}$ est lisse, son image $U_{\rm tri}^\square(\rbar)^{\rm reg}$ dans $X_{\rm tri}^\square(\rbar)$ est un ouvert adique par \cite[Prop.~1.7.8]{Huber}. L'existence de $s:\,U^\square\rightarrow\mathcal{S}^\square(\rbar)$ telle que $\pi_{\rbar}\circ s$ \'egale $f|_U$ compos\'ee avec $U^\square\rightarrow U$ montre que $U_{\rm tri}^\square(\rbar)^{\rm reg}$ contient $f(U)$. Par ailleurs, $f$ \'etant surjective, l'inclusion $f^{-1}(U_{\rm tri}^\square(\rbar)^{\rm reg})\subset U$ montre que $f(U)=U_{\rm tri}^\square(\rbar)^{\rm reg}$. Comme $f$ est un morphisme projectif, l'ensemble $f(U)=U_{\rm tri}^\square(\rbar)^{\rm reg}$ est une partie constructible au sens de Zariski de $X_{\rm tri}^\square(\rbar)$ d'apr\`es \cite[Lem.~2.14]{HellmSchrDensity}. On d\'eduit alors de \cite[Lem.~2.13]{HellmSchrDensity} que $U_{\rm tri}^\square(\rbar)^{\rm reg}$ est un ouvert de Zariski de $X_{\rm tri}^\square(\rbar)$, en particulier admissible.

Le diagramme \eqref{spaceofphigamma}, la lissit\'e des applications $\pi_{\rbar}$ et $\kappa$, ainsi que le Lemme \ref{lissite2} montrent alors que $U_{\rm tri}^\square(\rbar)^{\rm reg}$ est lisse et que l'application $\omega'$ est lisse, il en est donc de m\^eme de $\omega$.
\end{proof}

\begin{rema}
Des r\'esultats analogues ont \'egalement \'et\'e obtenus par Bergdall (\cite{Bergdall}) et Liu (\cite{Liufiltration}).
\end{rema}

\subsection{Espaces de d\'eformations cristallines et poids de Sen}

Soit ${\bf k}$ un poids alg\'ebrique dominant. On note $R_{\rbar}^{\square,{\bf k}-{\rm cr}}=R_{\rbar}^\square/\cap {\mathfrak p}_x$ o\`u l'intersection est prise sur l'ensemble des id\'eaux premiers ${\mathfrak p}_x={\rm Ker}(x:\,R_{\rbar}^\square\rightarrow\bar{\Q}_p)$ tels que la repr\'esentation $r_x$ de $\mathcal{G}_K$ associ\'ee \`a $x$ est cristalline de poids de Hodge-Tate ${\bf k}$. C'est une $\mathcal{O}_L$-alg\`ebre locale, plate, compl\`ete, noeth\'erienne, r\'eduite de m\^eme corps r\'esiduel que $R_{\rbar}^\square$ et par \cite{Kisindef}, pour $A$ une $\bar \Q_p$-alg\`ebre locale artinienne, un morphisme $x:\,R_{\rbar}^\square\rightarrow A$ se factorise \`a travers $R_{\rbar}^{\square,{\bf k}-{\rm cr}}$ si et seulement si $r_x$ est cristalline de poids de Hodge-Tate ${\bf k}$. On note $\mathfrak{X}_{\rbar}^{\square,{\bf k}-{\rm cr}}$ l'espace $\Spf(R_{\rbar}^{\square,{\bf k}-{\rm cr}})^{\rm rig}$ et on rappelle que ses composantes irr\'eductibles sont en bijection canonique avec les composantes irr\'eductibles du sch\'ema $\Spec(R_{\rbar}^{\square,{\bf k}-{\rm cr}}[\frac{1}{p}])$. Par \cite[Lem. 7.1.9]{DeJong}, on note que ses points sont aussi les m\^emes que les points ferm\'es de $\Spec(R_{\rbar}^{\square,{\bf k}-{\rm cr}}[\frac{1}{p}])$.

Si $r$ est une repr\'esentation cristalline du groupe $\mathcal{G}_K$ sur un $L$-espace vectoriel, on appelle Frobenius cristallin lin\'earis\'e de $r$ l'endomorphisme $L$-lin\'eaire $\varphi^{[K_0:\Q_p]}$ de \[D_{\rm cris}(r)\otimes_{K_0\otimes_{\Q_p}L,\tau}L=(B_{\rm cris}\otimes_{\Q_p}r)^{\mathcal{G}_K}\otimes_{K_0\otimes_{\Q_p}L,\tau}L\] pour un plongement $\tau$ de $K_0$ dans $L$ (la classe de conjugaison de l'endomorphisme obtenu est alors ind\'ependante du plongement).

Supposons d\'esormais que ${\bf k}$ est strictement dominant. Soit $r$ une repr\'esentation de $\mathcal{G}_K$ de dimension $n$ sur $L$, cristalline de poids de Hodge-Tate ${\bf k}$. Supposons que les valeurs propres du Frobenius cristallin lin\'earis\'e de $r$ soient deux-\`a-deux distinctes (en particulier $\varphi^{[K_0:\Q_p]}$ est diagonalisable quitte \`a agrandir $L$). Rappelons que l'on appelle raffinement de $r$ un ordre de ces valeurs propres $\varphi_1,\dots,\varphi_n$. Le m\^eme raisonnement que dans \cite[\S\ 2.4.2]{BelChe}, mais en utilisant \cite[Cons. 6.2.4]{KPX}, montre qu'il y a une bijection entre raffinements de $r$ et (syst\`emes de) param\`etres de $r$. Soit $\varphi_1,\dots,\varphi_n$ un raffinement de $r$, on peut alors d\'efinir une filtration croissante de $D_{\rm cris}(r)$ par des sous-$K_0\otimes_{\Q_p}L$-modules libres, en posant \[\mathcal{F}_i(D_{\rm cris}(r))=\bigoplus_{j\leq i}D_{\rm cris}(r)^{\varphi^{[K_0:\Q_p]}=\varphi_j}.\] Une telle filtration d\'efinit, comme dans \cite[\S\ 2.4.1]{BelChe}, un ordre sur les poids de Hodge-Tate du module filtr\'e $D_{\rm dR}(r)_{K_0\otimes_{\Q_p}L,\tau}L$ pour tout $\tau\in\Hom(K,L)$. Le raffinement est alors dit \emph{non critique} si, pour tout $\tau$, les poids de Hodge-Tate de $D_{\rm dR}(r)_{K_0\otimes_{\Q_p}L,\tau}L$ sont ordonn\'es dans l'ordre {\it d\'ecroissant} (attention que, contrairement \`a \cite{BelChe}, nous supposons que le poids de $\varepsilon$ est $1$).

\begin{defi}\label{generique}
Un point $x$ de l'espace $\mathfrak{X}_{\rbar}^{\square,{\bf k}-{\rm cr}}$ est dit g\'en\'erique\footnote{Cette d\'efinition est une version faible de la d\'efinition de \cite[\S3.18]{CheFougere}, diff\'erente de la notion de g\'en\'ericit\'e de \cite[D\'ef.~5.2]{BreuilAnalytiqueII}.} si les valeurs propres du Frobenius cristallin lin\'earis\'e de la repr\'esentation $r_x$ sont deux-\`a-deux distinctes, et si tous les raffinements de $r_x$ sont non critiques.
\end{defi}

Un caract\`ere continu $\delta\in\mathcal{T}(L)$ de $K^\times$ dans $L^\times$ est en particulier une fonction localement $\mathbb{Q}_p$-analytique sur $K^\times$. Sa diff\'erentielle $\d\delta$ en $1\in K^\times$ est alors une application $\mathbb{Q}_p$-lin\'eaire de l'espace tangent de $K^\times$ en $1$ dans $L$. Une $\mathbb{Q}_p$-base de cet espace tangent est donn\'ee par les diff\'erentielles en $1$ des diff\'erents plongements $\tau:\, K\hookrightarrow L$. On appelle alors \emph{poids} du caract\`ere $\delta$ les coordonn\'ees de $\d\delta$ dans cette base, on note ces poids $\omega_{\tau}(\delta)$. Plus pr\'ecis\'ement on a $\d\delta=\sum_{\tau}\omega_\tau(\delta)\d\tau$. D'apr\`es le Lemme 2.1 de \cite{STFour}, l'application $(\omega_{\tau})$ ainsi construire de $\mathcal{T}(L)$ vers $L^{[K:\mathbb{Q}_p]}$ provient d'un morphisme d'espaces analytiques rigides $\mathcal{T}_L\rightarrow\mathbb{A}^{[K:\mathbb{Q}_p]}_L$.

Si $V$ est une repr\'esentation continue de $\mathcal{G}_K$ sur un $L$-espace vectoriel de dimension finie. Si $\tau$ est un plongement de $K$ dans $L$, on appelle $\tau$-poids de Sen, les oppos\'es des valeurs propres de l'endomorphisme $\Theta_V$ sur ${\bf D}_{\rm Sen}(V)\otimes_{K\otimes_{\mathbb{Q}_p}L,\tau}L$ (cf.~par exemple \cite[\S5.3]{Bergereqdiff}).

\begin{prop}\label{Sen}
Si $x=(r_x,\delta_x)\in X_{\rm tri}^\square(\rbar)$. Alors pour tout plongement $\tau$ de $K$ dans $L$, l'ensemble $\{\omega_{\tau}(\delta_{i,x})|\,1\leq i\leq n\}$ est l'ensemble des $\tau$-poids de Sen de $r_x$.
\end{prop}

\begin{proof}
Comme $U_{\rm tri}^\square(\rbar)^{\rm reg}$ est un ouvert de Zariski dense dans $X_{\rm tri}^\square(\rbar)$, il suffit de prouver l'assertion pour un point $x\in U_{\rm tri}^\square(\rbar)^{\rm reg}$, puisque les applications $\omega_\tau$ sont analytiques d'apr\`es \cite[\S7.2]{BerColfamilles}. Nous pouvons donc, en utilisant le diagramme \eqref{spaceofphigamma}, nous ramener \`a prouver le r\'esultat pour les points de l'espace $\mathcal{S}^\square(\rbar)$ et raisonner par r\'ecurrence pour se ramener au cas o\`u $n=1$. Consid\'erons donc la famille universelle $\mathcal{R}(\delta)$ sur $\mathcal{T}_L$. Par analyticit\'e et Zariski-densit\'e des caract\`ere localement alg\'ebriques, il suffit de prouver l'assertion en un point $\delta_x\in\mathcal{T}$ tel que $\delta_x|_{\mathcal{O}_K^\times}$ est un caract\`ere alg\'ebrique. Nous sommes donc ramen\'es \`a prouver l'assertion pour $D=\mathcal{R}(\delta_x)$ avec $\delta_x(z)=\prod_{\tau}\tau(x)^{k_\tau}$ lorsque $z\in\mathcal{O}_K^\times$. Un calcul facile montre que $\omega_\tau(\delta_x)=k_\tau$. Il reste donc \`a prouver que $\Theta_D$ agit sur ${\bf D}_{\rm Sen}(\mathcal{R}(\delta_x))$ par multiplication par $k_\tau$. Quitte \`a multiplier $\delta_x$ par un caract\`ere non ramifi\'e, on peut supposer que $\delta_x$ est \`a valeurs dans $\mathcal{O}_{k(x)}^\times$, et donc que $D$ est le $(\varphi,\Gamma_K)$-module associ\'e au caract\`ere $\delta_x\circ{\rm rec}^{-1}$. Or on sait qu'un tel caract\`ere a pour $\tau$-poids de Hodge-Tate $k_\tau$.
\end{proof}

\begin{lemm}\label{isosature}
Soit $X$ un espace rigide analytique r\'eduit et $D$ un $(\varphi,\Gamma_K)$-module sur $X$. Soit $D_1$ un sous-$\mathcal{R}_X$-module de $D$ stable par $\varphi$ et $\Gamma_K$. Supposons qu'il existe un morphisme injectif de $(\varphi,\Gamma)$-modules \[\alpha: D_1\longrightarrow G_1\] o\`u $G_1$ est un $(\varphi,\Gamma_K)$-module localement sur $X$ libre de rang $1$. Si $i_x^\ast(\alpha)$ est un isomorphisme pour $x$ dans une partie Zariski-dense de $X$, alors le morphisme $\alpha$ est un isomorphisme.
\end{lemm}

\begin{proof}
C'est exactement le contenu de la preuve de \cite[Cor.~6.3.10]{KPX}.
\end{proof}

\begin{lemm}\label{parametregenerique}
Soit $x=(r_x,\delta_x)$ un point de $X_{\rm tri}^\square(\rbar)$ tel que $r_x$ est une repr\'esentation cristalline g\'en\'erique. Alors le caract\`ere $\delta_x$ est localement alg\'ebrique, strictement dominant et $x\in U_{\rm tri}^\square(\rbar)^{\rm reg}$.
\end{lemm}

\begin{proof}
La preuve qui suit est dans le m\^eme esprit que \cite[Prop.~9.2]{BreuilAnalytiqueII}. Comme dans la preuve du Th\'eor\`eme \ref{ouvertsature}, il existe d'apr\`es \cite[Cor.~6.3.10]{KPX} un morphisme propre et birationnel $\tilde X=\tilde{X}_{\rm tri}^\square(\rbar)\rightarrow X_{\rm tri}^\square(\rbar)$, ainsi qu'une filtration $F_\bullet$ de $\tilde{D}_X$, l'image inverse de $D_X={\bf D}^\dagger_{\rm rig}(r_X)$ par des sous-$\mathcal{R}_X$-modules stables par $\varphi$ et $\Gamma_K$ telle que $F_0=0$ et $F_n=\tilde{D}_X$, une famille $(\mathcal{L}_i)_{1\leq i\leq n}$ de $\mathcal{O}_{\tilde{X}}$-modules inversibles ainsi que des morphismes injectifs \[\alpha_i:\, F_i/F_{i-1}\longrightarrow\mathcal{R}_{\tilde{X}}(\delta_{i,\tilde{X}})\otimes_{\mathcal{O}_{\tilde{X}}}\mathcal{L}_i\] qui sont des isomorphismes sur une partie Zariski-dense de $\tilde{X}$. Par \cite[Cor.~6.3.10.(2')]{KPX}, $\alpha_1$ est alors automatiquement un isomorphisme. Fixons $\tilde{x}$ un ant\'ec\'edent de $x$ dans $\tilde{X}$. En sp\'ecialisant le $(\varphi,\Gamma_K)$-module $D_{\tilde{X}}$ en $\tilde{x}$, on obtient en particulier une injection \[\alpha_{1,x}:\,\mathcal{R}_{k(x),K}(\delta_{1,x})\hookrightarrow i_{\tilde{x}}^\ast(D_{\tilde{X}})\simeq i_x^\ast(D_X)\]
de $(\varphi,\Gamma_K)$-modules sur $k(x)$. Consid\'erons le satur\'e de l'image de $\alpha_{1,x}$, c'est-\`a-dire l'intersection ${\rm Im}(\alpha_{1,x})[\frac{1}{t}]\cap i_x^\ast(D_x)$. D'apr\`es \cite[Prop.~2.2.2]{BelChe} et \cite{Co}, il s'agit d'un $(\varphi,\Gamma_K)$-module de rang $1$, donc de la forme $\mathcal{R}_{k(x),K}(\delta_1')$ avec $\delta_1'\in\mathcal{T}(k(x))$. Comme $\mathcal{R}_{k(x),K}(\delta_{1,x})$ est isomorphe \`a un sous-$(\varphi,\Gamma_K)$-module de $\mathcal{R}_{k(x),K}(\delta_1')$, il existe ${\bf n}\in\mathbb{N}^{[K:\mathbb{Q}_p]}$ tel que $\delta_{1,x}=x^{\bf n}\delta_1'$. Comme $i_x^\ast(D_X)$ est le $(\varphi,\Gamma)$-module de la repr\'esentation cristalline $r_x$, $\delta_1'$ est le premier terme d'un param\`etre $\delta'=(\delta_1',\dots,\delta_n')$ de $i_x^\ast(D_X)$. Les caract\`eres $\delta_x$ et $\delta'$ sont localement alg\'ebriques de poids respectifs ${\bf k}=(k_{\tau,i})$ et ${\bf k}'=(k'_{\tau,i})$ et comme la repr\'esentation galoisienne associ\'ee au $(\varphi,\Gamma)$-module $i_x^\ast(D_X)$ est g\'en\'erique, $\delta'$ est un caract\`ere dominant. Par ailleurs, on d\'eduit de la Proposition \ref{Sen} que les ensembles $\{k_{\tau,i}|\,1\leq i\leq n\}$ et $\{k'_{\tau,i}|\,1\leq i\leq n\}$ sont \'egaux pour tout plongement $\tau$ de $K$ dans $L$, en particulier l'ensemble $\{k_{\tau,i}'|\,1\leq i\leq n\}\cup\{k_{\tau,1}\}$ est de cardinal $n$. Dans la suite d'in\'egalit\'es
\[ k_{\tau,1}\geq k'_{\tau,1}>k'_{\tau,2}>\dots>k'_{\tau,n},\]
la premi\`ere est donc n\'ecessairement une \'egalit\'e. Ceci \'etant vrai pour tout $\tau$, on a $\delta_{1,x}=\delta_{1,x}'$, donc $\mathcal{R}_{k(x),K}(\delta_{1,x})$ est un sous-module satur\'e de $i_x^\ast(D_X)$. Par \cite[Lem.~2.2.3]{BelChe}, on en conclut que le $\mathcal{R}_{k(x),K}$-module $i_x^\ast(D_X)/\mathcal{R}_{k(x),K}(\delta_{1,x})$ est libre. Ceci implique que, sur un voisinage affino\"ide $U_1$ de $\tilde{x}$ dans $\tilde{X}_{\rm tri}^\square(\rbar)$, le quotient $\coker(\alpha_1)$ est un $(\varphi,\Gamma_K)$-module. En particulier, le lemme \ref{isosature} montre que l'application $\alpha_2$ est un isomorphisme sur $U_1$. Par r\'ecurrence, on construit une suite d\'ecroissante de voisinages $(U_i)_{1\leq i\leq n}$ de $\tilde{x}$ telle que, sur $U_i$, les application $\alpha_j$ pour $j\leq i$ soient des isomorphismes. On conclut au final que $\delta_x$ est un param\`etre du $(\varphi,\Gamma_K)$-module $i_x^\ast(D_X)$, c'est-\`a-dire que le point $x$ appartient \`a $U_{\rm tri}^\square(\rbar)^{\rm reg}$.
\end{proof}

Rappelons que $T\cong \mathbb{G}_m^n$ est le tore diagonal de $\GL_n$. On introduit un automorphisme $\iota_K$ de $\widehat{T(K)}$ en posant
\begin{equation*}
\iota_K(\delta_1,\cdots,\delta_n)=\delta_B\cdot(\delta_1,\cdots,\delta_i\cdot(\varepsilon\circ\rec_K)^{i-1},\cdots,\delta_n\cdot(\varepsilon\circ\rec_K)^{n-1})
\end{equation*}
o\`u $\delta_B=|\cdot|_K^{n-1}\otimes |\cdot|_K^{n-3}\otimes \cdots \otimes |\cdot|_K^{1-n}$ est le caract\`ere module de $B(K)$. La raison d'\^etre de cette param\'etrisation est la suivante. Normalisons la correspondance locale de Langlands pour $\GL_n(K)$ de telle sorte que la repr\'esentation de Weil $\chi_1\oplus \cdots \oplus \chi_n$ corresponde \`a l'induite parabolique lisse non normalis\'ee $\Ind_{B(K)}^{\GL_n(K)}\big(\delta_B^{1/2}\cdot (\chi_1\circ \rec_K\cdots \otimes \chi_n\circ\rec_K)\big)$. Si $r$ est une repr\'esentation cristalline de dimension $n$ de $\mathcal{G}_K$ sur $L$ de poids de Hodge-Tate ${\bf k}=(k_{\tau,i})_{1\leq i\leq n, \tau :K\hookrightarrow L}$ strictement dominant, et $\pi$ une repr\'esentation non ramifi\'ee irr\'eductible de $\GL_n(K)$ sur $\Qbar_p$ telle que la repr\'esentation de Weil associ\'ee \`a $r$ (\cite{Fo}) corresponde \`a $\pi|\det|_K^{\frac{1-n}{2}}$ par la correspondance de Langlands normalis\'ee comme ci-dessus, alors $(\delta_1,\cdots,\delta_n)$ est un param\`etre de $r$ si et seulement si $\iota_K(\delta_1,\cdots,\delta_n)$ est un syst\`eme de valeurs propres de $T(K)$ agissant sur $J_{B}(\pi\otimes_{\Qbar_p} W_{\bf n})$ o\`u $W_{\bf n}$ est la repr\'esentation alg\'ebrique irr\'eductible de $\mathrm{Res}_{K/\mathbb{Q}_p}(\GL_n)$ sur $\Qbar_p$ de plus haut poids ${\bf n}=(k_{\tau,i}+i-1)_{1\leq i\leq n, \tau :K\hookrightarrow\Qbar_p}$ relativement \`a $B$ et $J_B$ est le foncteur de Jacquet comme \'etendu par Emerton aux repr\'esentations localement alg\'ebriques, cf. \cite[Prop. 4.3.6]{EmertonJacquetI}.

\subsection{Espaces de formes automorphes}\label{automorphe}

Fixons $F$ un corps CM, i.e. une extension quadratique totalement imaginaire d'un corps $F^+\subset F$ totalement r\'eel et notons $c$ l'unique automorphisme non trivial de $F$ fixant $F^+$. Nous supposons de plus que toutes les places de $F^+$ divisant $p$ sont compl\`etement d\'ecompos\'ees dans $F$ et que l'extension $F/F^+$ est non ramifi\'ee en toute place finie. Pour toute place finie $w$ de $F$, on fixe une uniformisante $\varpi_w$ de $F_w$. On choisit d\'esormais le corps $L$ tel que pour toute place $v$ de $F^+$ divisant $p$, on a $|\Hom(F^+_v,L)|=[F^+_v:\Q_p]$.

Supposons d\'esormais $n$ impair ou $[F^+:\mathbb{Q}]$ pair. Soit $G$ un groupe unitaire \`a $n$ variables d\'efini sur $F^+$, d\'eploy\'e par $F$, quasi-d\'eploy\'e en toute place finie et tel que $G(F^+\otimes_{\Q}\mathbb{R})$ est compact. On fixe $U^p\subset G(\mathbb{A}_{F^+}^{p\infty})$ un sous-groupe compact ouvert de la forme $\prod_{v}U_v$ o\`u $U_v$ est un sous-groupe compact ouvert de $G(F_v)$. Ici $\mathbb{A}_{F^+}^{p\infty}$ d\'esigne les ad\`eles finis de $F^+$ hors les places divisant $p$. On fixe aussi un ensemble fini $S$ de places finies de $F^+$ contenant l'ensemble $S_p$ des places divisant $p$ ainsi que l'ensemble des places $v$ telles que $U_v$ n'est pas hypersp\'ecial, et on suppose que si $v\in S$, alors $v$ est totalement d\'ecompos\'ee dans $F$ (de mani\`ere \'equivalente, on suppose $U_v$ hypersp\'ecial en toutes les places de $F^+$ inertes dans $F$).

Si $v\in S$, fixons $\tilde{v}$ une place de $F$ divisant $v$ et notons $\tilde{v}^c$ son image par l'automorphisme $c$. L'inclusion de $F^+$ dans $F$ induit donc deux isomorphismes $F^+_v\xrightarrow{\sim}F_{\tilde{v}}$ et $F^+_v\xrightarrow{\sim}F_{\tilde{v}^c}$. Fixons de plus un isomorphisme de groupes alg\'ebriques $G\times_{F^+}F\simeq \GL_{n,F}$. Si $v$ est d\'ecompos\'ee dans $F$ et $w$ est une place de $F$ divisant $v$, on obtient un isomorphisme $i_w:\,G(F_v^+)\xrightarrow{\sim}\GL_n(F_w)$. On suppose de plus que, pour $v\notin S$, on a $U_v=i_w^{-1}(\GL_n(\mathcal{O}_{F_w}))$, condition qui ne d\'epend pas du choix de $w$ divisant $v$. Pour $v\in S$, notons $K_v$ le sous-groupe compact maximal de $G(F^+_v)$ image r\'eciproque par $i_{\tilde{v}}$ du groupe $\GL_n(\mathcal{O}_{F_{\tilde{v}}})$. On note aussi $G_p=\prod_{v\in S_p}G(F^+_v)$ et $K_p=\prod_{v\in S_p}K_v$. On d\'esigne par $B_v$ l'image inverse par $i_{\tilde{v}}$ du sous-groupe des matrices triangulaires sup\'erieures de $\GL_n(F_{\tilde{v}})$ et $T_v$ et $N_v$ les images inverses des sous-groupes de matrices diagonales et unipotentes sup\'erieures. On pose alors 
\begin{align*}
B_p&=\prod_{v\in S_p}B_v,  & B_p^0&=B_p\cap K_p,\\ N_p&=\prod_{v\in S_p} N_v, & N_p^0&=N_p\cap K_p, \\ T_p&=\prod_{v\in S_p}T_v, & T_p^0&=T_p\cap K_p.
\end{align*} 
Enfin, on note $\overline{B}_p$ et $\overline{N}_p$ les sous-groupes de Borel et unipotent oppos\'es \`a $B_p$ et $N_p$. Si $I\subset\{1,\cdots,n\}$ et $v\in S_p$, on note $P_{v,I}$, $\overline{P}_{v,I}$ et $L_{v,I}$ les images r\'eciproques de $P_I(F_{\tilde{v}})$, $\overline{P}_I(F_{\tilde{v}})$ et $L_I(F_{\tilde{v}})$.

Soit $\lambda=(\lambda_{\tau,i})\in(\Z^n)^{\Hom(F^+,L)}$. Un \'el\'ement $\delta\in\widehat{T}_p(L)$ est dit alg\'ebrique de poids $\lambda$ s'il est de la forme $\bigotimes_{v\in S_p}\delta_{\lambda_v}$ o\`u $\lambda_v=(\lambda_{\tau,i})
_{1\leq i\leq n,\tau:\, F^+_v\hookrightarrow L}$. Il est dit localement alg\'ebrique de poids $\lambda$ s'il co\"incide avec $\bigotimes_{v\in S_p}\delta_{\lambda_v}$ sur un sous-groupe ouvert de $T_p$.

Si $V=\prod_v V_v$ est un sous-groupe compact ouvert de $G(\mathbb{A}_{F^+}^\infty)$ et $\sigma$ une repr\'esentation continue de $\prod_v V_v$ sur un $\mathcal{O}_L$-module topologique, on note $S(V,\sigma)$ le $\mathcal{O}_L$-module des fonctions continues $f:\, G(F^+)\backslash G(\mathbb{A}_{F^+}^\infty)\rightarrow\sigma$ telles que $f(g\cdot u)=u^{-1}(f(g))$ pour tout $g\in \mathbb{A}_{F^+}^\infty$ et $u\in V$.

Si $v$ est une place de $F^+$ n'appartenant pas \`a $S$ et d\'ecompos\'ee dans $F$, on note $\mathbb{T}_v$ l'alg\`ebre de Hecke sph\'erique du groupe $G(F^+_v)$ relativement \`a $U_v$, c'est-\`a-dire \[\mathbb{T}_v=\mathcal{O}_L[U_v\backslash G(F^+_v)/U_v].\] On pose alors $\mathbb{T}^S=\varinjlim_I(\bigotimes_{v\in I}\mathbb{T}_v)$, la limite inductive \'etant prise sur l'ensemble des parties finies de places finies de $F^+$ d\'ecompos\'ees dans $F$ et d'intersection nulle avec $S$. L'alg\`ebre $\mathbb{T}^S$ agit alors naturellement sur $S(U^pU_p,\sigma)$ pour tout sous-groupe compact ouvert $U_p$ de $G_p$. De plus cette action est fonctorielle en $U_p$ et en $\sigma$ (en un sens \'evident). Si $w$ est une place de $F$ divisant $v$, on note $T_{w,i}$ l'\'el\'ement de $\mathbb{T}_v$ de support $U_v\gamma_{w,i} U_v$ o\`u $\gamma_{w,i}$ est l'\'el\'ement de $G(F_v^+)$ image r\'eciproque par $i_w$ de la matrice diagonale \[\beta_i(\varpi_w)=(\underbrace{\varpi_w,\cdots,\varpi_w}_{i},1,\cdots,1).\]

Consid\'erons un corps quelconque $C$ et $\theta:\,\mathbb{T}^S\rightarrow C$ un morphisme d'anneaux. Si $\rho$ est une repr\'esentation $C$-lin\'eaire de $\mathcal{G}_F=\mathrm{Gal}(\bar{F}/F)$ sur un $C$-espace vectoriel de dimension finie, on dit que $\rho$ est associ\'ee \`a $\theta$ si pour toute place $w$ de $F$ au-dessus d'une place de $F^+$ n'appartenant pas \`a $S$ et compl\`etement d\'ecompos\'ee dans $F$, le polyn\^ome caract\'eristique de $\rho(\frob_w)$ est donn\'e par
\begin{equation}\label{polcar}
X^n+\cdots+(-1)^j(Nw)^{j(j-1)/2}T_{w,j}X^{n-j}+\cdots+(-1)^n(Nw)^{n(n-1)/2}T_{w,n}
\end{equation}
o\`u $\frob_w$ est un Frobenius g\'eom\'etrique en $w$ et $Nw$ le cardinal du corps r\'esiduel de $F_w$. Si $\pi$ est une repr\'esentation automorphe de $G(\mathbb{A}_{F^+})$ non ramifi\'ee hors de $S$, il existe une unique $\Qbar_p$-repr\'esentation semi-simple $\rho_\pi$ de $G_F$ associ\'ee au caract\`ere de $\mathbb{T}^S$ donn\'e par l'action de $\mathbb{T}^S$ sur la $\Qbar_p$-droite $(\pi^S)^{U^S}\otimes_{\C}\Qbar_p$ o\`u $U^S=\prod_{v\notin S}U_v$ (voir \cite[Thm. 3.2.3]{ChHaII}).

On dit qu'un id\'eal maximal $\mathfrak{m}$ de l'alg\`ebre $\mathbb{T}^S$ est automorphe de niveau mod\'er\'e $U^p$ s'il existe une repr\'esentation alg\'ebrique $\sigma$ de $K_p$ sur un $\mathcal{O}_L$-module (avec topologie discr\`ete) qui est irr\'eductible apr\`es extension des scalaires de $\mathcal{O}_L$ \`a $\Qbar_p$ (donc de dimension finie sur $\Qbar_p$) et un sous-groupe compact ouvert $U_p$ de $G_p$ tels que $S(U^pU_p,\sigma)_{\mathfrak{m}}\neq0$. En particulier par ce qui pr\'ec\`ede (en proc\'edant par exemple comme dans \cite[\S\ 3]{Emint} ou \cite[\S\ 6.2.3]{BelChe}) il existe alors aussi une repr\'esentation continue $\rhobar_{\mathfrak{m}}$ de $G_F$ sur un $\mathbb{T}^S/\mathfrak{m}$-espace vectoriel de dimension $n$ associ\'ee au caract\`ere $\mathbb{T}^S\rightarrow\mathbb{T}^S/\mathfrak{m}$. On dit que l'id\'eal maximal $\mathfrak{m}$ est de plus non Eisenstein si $\rhobar_{\mathfrak{m}}$ est absolument irr\'eductible. Fixons d\'esormais un tel id\'eal maximal $\mathfrak{m}$ automorphe de niveau mod\'er\'e $U^p$ non Eisenstein et une repr\'esentation alg\'ebrique $\sigma$ comme ci-dessus. On v\'erifie facilement \`a partir de (\ref{polcar}) et des d\'efinitions que les repr\'esentations irr\'eductibles $\rhobar_{\mathfrak{m}}^\vee\circ c$ et $\rhobar_{\mathfrak{m}}\otimes\bar{\varepsilon}^{n-1}$ sont isomorphes (o\`u $\bar{\varepsilon}$ est la r\'eduction de $\varepsilon$ modulo $p$). Quitte \`a agrandir $L$, on suppose $\mathbb{T}^S/\mathfrak{m}=k_L$. On suppose \'egalement $p>2$.

Soit $\rhobar$ une repr\'esentation continue absolument irr\'eductible de $\mathcal{G}_F$ sur un $k_L$-espace vectoriel de dimension $n$ telle que $\rhobar^\vee\circ c\cong\rhobar\otimes\bar{\varepsilon}^{n-1}$ et $\rhobar$ est non ramifi\'ee hors de $S$. Alors le foncteur associant \`a toute $\mathcal{O}_L$-alg\`ebre locale artinienne $A$ de corps r\'esiduel $k_L$ l'ensemble des classes d'isomorphisme de d\'eformations $\rho_A$ de $\rhobar$ \`a $A$ non ramifi\'ees en dehors de $S$ et v\'erifiant $\rho_A^{\vee}\circ c\simeq \rho_A\otimes\varepsilon^{n-1}$ est pro-repr\'esentable par une $\mathcal{O}_L$-alg\`ebre locale compl\`ete noeth\'erienne de corps r\'esiduel $k_L$ not\'ee $R_{\rhobar,S}$.

Pour tout sous-groupe compact ouvert $U_p$ de $G_p$ on note $\mathbb{T}^S(U^pU_p,\sigma)_{\mathfrak{m}}$ l'image de $\mathbb{T}^S$ dans l'alg\`ebre des endomorphismes de $S(U^pU_p,\sigma)_{\mathfrak{m}}$. De \cite[Prop. 6.7]{Thorne} (notons que l'on utilise ici $p\ne 2$) on d\'eduit qu'il existe un morphisme canonique de $\mathcal{O}_L$-alg\`ebres locales noeth\'eriennes compl\`etes
$$\psi:\,R_{\rhobar_{\mathfrak{m}},S}\longrightarrow\mathbb{T}^S(U^pU_p,\sigma)_{\mathfrak{m}}$$
tel que, pour tout morphisme de $\mathcal{O}_L$-alg\`ebres $\theta:\,\mathbb{T}^S(U^pU_p,\sigma)_{\mathfrak{m}}\rightarrow L'$ o\`u $L'$ est une extension finie de $L$, la compos\'ee $\theta\circ\psi$ d\'efinit une d\'eformation de $\rhobar_{\mathfrak{m}}$ associ\'ee au caract\`ere $\theta$.

On d\'efinit l'espace de cohomologie compl\'et\'ee du groupe $G$ de niveau mod\'er\'e $U^p$, ou espace des formes automorphes $p$-adiques de niveau mod\'er\'e $U^p$, comme la compl\'etion de l'espace $\varinjlim_{U_p} S(U^pU_p,L)$ pour la norme d\'efinie par le r\'eseau $\varinjlim_{U_p} S(U^pU_p,\mathcal{O}_L)$, la limite inductive \'etant prise sur les sous-groupes compacts ouverts $U_p$ de $G_p$. On note $\widehat{S}(U^p,L)$ cet espace de Banach $p$-adique, qui s'identifie \'egalement \`a l'espace des fonctions continues $G(F^+)\backslash G(\mathbb{A}_{F^+}^\infty)/U^p\rightarrow L$ muni de la norme sup. Il est naturellement muni d'une action continue de $G_p$ par translation \`a droite sur ces fonctions et d'une action continue de $\mathbb{T}^S$ qui commute \`a celle de $G_p$, ces actions pr\'eservant la boule unit\'e $\widehat{S}(U^p,{\mathcal O}_L)$ (i.e.~l'adh\'erence de $\varinjlim_{U_p} S(U^pU_p,\mathcal{O}_L)$ dans $\widehat{S}(U^p,L)$). 

Notons que ces actions s'\'etendent au localis\'e $\widehat{S}(U^p,L)_{\mathfrak m}$, un espace de Banach $p$-adique sur $L$ qui s'identifie au compl\'et\'e de $\varinjlim_{U_p} S(U^pU_p,L)_{\mathfrak m}$ par rapport \`a $\varinjlim_{U_p} S(U^pU_p,{\mathcal O}_L)_{\mathfrak m}$, dont elles pr\'eservent la boule unit\'e $\widehat{S}(U^p,{\mathcal O}_L)_{\mathfrak m}$. On en d\'eduit que $\widehat{S}(U^p,L)_{\mathfrak m}$ et $\widehat{S}(U^p,{\mathcal O}_L)_{\mathfrak m}$ sont munis via $\psi$ d'une action continue de la $\mathcal{O}_L$-alg\`ebre $R_{\rhobar_{\mathfrak{m}},S}$ commutant \`a l'action de $G_p$ (et qui co\"\i ncide sur ${\mathcal O}_L$ avec la multiplication scalaire).

Si $v$ est une place de $F^+$ divisant $p$, on note $\rhobar_{\tilde{v}}$ la restriction de $\rhobar_{\mathfrak{m}}$ au groupe de d\'ecomposition en $\tilde{v}$. Rappelons que l'on a d\'efini un sous-espace analytique rigide ferm\'e et r\'eduit $X_{\rm tri}^\square(\rhobar_{\tilde{v}})$ de $\Spf(R^\square_{\rhobar_{\tilde{v}}})^{\rm rig}\times\widehat{T}_{v,L}$ correspondant aux d\'eformations triangulines de $\rhobar_{\tilde{v}}$ (D\'efinition \ref{deftriangulines}) ainsi qu'un ouvert admissible Zariski-dense $U_{\rm tri}^\square(\rhobar_{\tilde{v}})^{\rm reg}$ de $X_{\rm tri}^\square(\rhobar_{\tilde{v}})$ (Th\'eor\`eme \ref{ouvertsature}). On note $X_{\rm tri}^\square(\rhobar_p)$ l'espace analytique rigide produit $\prod_{v\in S_p}X_{\rm tri}^\square(\rhobar_{\tilde{v}})$, qui est donc un sous-espace analytique rigide ferm\'e et r\'eduit de $\Spf(R^\square_{\rhobar_p})^{\rm rig}\times\widehat{T}_{p,L}$ o\`u $R^\square_{\rhobar_p}$ d\'esigne $\widehat{\bigotimes}_{v\in S_p}R^\square_{\rhobar_{\tilde{v}}}$. De m\^eme, on note $U_{\rm tri}^\square(\rhobar_p)^{\rm reg}=\prod_{v\in S_p}U_{\rm tri}^\square(\rhobar_{\tilde{v}})^{\rm reg}$, un ouvert admissible et Zariski-dense de $X_{\rm tri}^\square(\rhobar_p)$.

\section{La vari\'et\'e de Hecke-Taylor-Wiles}

Nous conservons les notations et hypoth\`eses de la partie pr\'ec\'edente.

\subsection{Vecteurs localement $R$-analytiques}\label{analytique}

Soit $R$ une $\mathcal{O}_L$-alg\`ebre locale compl\`ete noeth\'erienne commutative de corps r\'esiduel $k_L$. On munit $R$ de sa topologie ${\mathfrak m}_R$-adique o\`u ${\mathfrak m}_R$ est l'id\'eal maximal de $R$. On d\'esigne par la lettre $G$ un groupe de Lie $p$-adique, par $G_0$ un sous-groupe compact ouvert de $G$, par $D(G_0,L)$ les distributions localement ($\Q_p$-)analytiques sur $G_0$ \`a valeurs dans $L$ (i.e. le dual continu des fonctions localement $\Q_p$-analytiques $G_0\rightarrow L$, cf. \cite{STdist}, \cite{STdual}, \cite{Emertonlocan}) et par $R\dbl G_0\dbr$ l'alg\`ebre d'Iwasawa de $G_0$ \`a coefficients dans $R$. Rappelons que ${\mathcal O}_L\dbl G_0\dbr[1/p]$ est le dual continu de l'espace de Banach des fonctions continues $G_0\rightarrow L$.

\begin{defi}\label{Radm}
Soit $\Pi$ un espace de Banach sur $L$ muni d'une action continue de $G$ et d'une action de l'alg\`ebre $R$ commutant \`a $G$ dont la restriction \`a ${\mathcal O}_L$ est la multiplication scalaire et telle que l'application canonique $R\times \Pi\rightarrow \Pi$ est continue. On dit que $\Pi$ est une repr\'esentation continue $R$-admissible si son dual continu $\Pi'$ est un $R\dbl G_0\dbr[1/p]$-module de type fini.
\end{defi}

Pr\'ecisons que l'action de $G$ sur $\Pi'$ est par translation \`a droite sur les fonctions et l'action de $R$ donn\'ee par $(r\cdot f)(v)=f(r\cdot v)$ ($r\in R$, $f\in \Pi'$, $v\in \Pi$).

\begin{ex}\label{ex}
{\rm Les espaces de Banach $\widehat{S}(U^p,L)$ et $\widehat{S}(U^p,L)_{\mathfrak m}$ du \S\ \ref{automorphe} sont des repr\'esentations continues $R_{\rhobar_{\mathfrak{m}},S}$-admissibles de $G_p$ sur $L$ (car leurs duaux sont en fait des ${\mathcal O}_L\dbl K_p\dbr[1/p]$-modules de type fini).}
\end{ex}

Fixons dans la suite une pr\'esentation $\mathcal{O}_L\dbl X_1,\cdots,X_s\dbr\twoheadrightarrow R$ de $R$. On a un isomorphisme de $\mathcal{O}_L$-alg\`ebres topologiques $\mathcal{O}_L\dbl X_1,\cdots,X_s\dbr\simeq\mathcal{O}_L\dbl \Z_p^s\dbr$ d\'efini par $X_i\mapsto\delta_i-1$ o\`u $(\delta_1,\cdots,\delta_s)$ d\'esigne la base canonique de $\Z_p^s$. Donc $\Pi'$ (comme dans la D\'efinition \ref{Radm}) est en fait un $\mathcal{O}_L\dbl \Z_p^s\times G_0\dbr$-module de type fini. Par \cite[Thm. 3.5]{STBan}, on voit que $\Pi$ est une repr\'esentation continue admissible de $\Z_p^s\times G$ sur $L$ et que la D\'efinition \ref{Radm} est ind\'ependante du choix du sous-groupe ouvert compact $G_0$ de $G$. Inversement, toute repr\'esentation continue admissible de $\Z_p^s\times G$ sur $L$ dont l'action induite de $\mathcal{O}_L\dbl \Z_p^s\dbr$ se factorise par $R$ est une repr\'esentation continue $R$-admissible de $G$. 

\begin{defi}\label{vecanal}
Soit $\Pi$ une repr\'esentation continue $R$-admissible de $G$ sur $L$. Un \'el\'ement de $\Pi$ est dit localement $R$-analytique s'il est localement analytique pour l'action de $\Z_p^s\times G$ au sens de \cite[\S\ 7]{STdist} (on montre ci-dessous que cela ne d\'epend pas du choix de la pr\'esentation choisie $\mathcal{O}_L\dbl \Z_p^s\dbr\twoheadrightarrow R$ de $R$). 
\end{defi}

On note $\Pi^{R-{\rm an}}$ le sous-espace de ces vecteurs localement analytiques. Rappelons que c'est un $L$-espace vectoriel localement convexe de type compact stable sous l'action de $\Z_p^s\times G$ (\cite[\S\ 7]{STdist}).

Notons $R^{\rm rig}={\mathcal O}(\Spf(R)^{\rm rig})$ l'alg\`ebre des fonctions globales sur l'espace analytique rigide $\Spf(R)^{\rm rig}$. Concr\`etement, on peut en donner la description suivante (cf. \cite[Def.\ 7.1.3]{DeJong}). Pour tout $n\geq1$, on a dans $R[1/p]$ le sous-anneau $R[\mathfrak{m}_R^n/\varpi]$ (avec des notations \'evidentes) dont on note $R_n^\circ$ la compl\'etion $\varpi$-adique. On a alors 
$$R^{\rm rig}=\varprojlim_n (R_n^\circ\otimes_{\mathcal{O}_L}L).$$
De plus \cite[Lem. 7.2.2]{DeJong} montre que $R^{\rm rig}$ est une $L$-alg\`ebre de Fr\'echet-Stein au sens de \cite[\S\ 3]{STdist}. Enfin, si $S$ est une autre $\mathcal{O}_L$-alg\`ebre locale compl\`ete noeth\'erienne, on d\'eduit de \cite[Prop. 7.2.4.g)]{DeJong} un isomorphisme 
\begin{equation}\label{tenseur}
R^{\rm rig}\widehat{\otimes}_{L}S^{\rm rig}\simeq (R\widehat{\otimes}_{\mathcal{O}_L}S)^{\rm rig}
\end{equation}
(o\`u le produit tensoriel \`a gauche est dans la cat\'egorie des $L$-alg\`ebres de Fr\'echet-Stein et \`a droite dans la cat\'egorie des $\mathcal{O}_L$-alg\`ebres locales compl\`etes noeth\'eriennes). Rappelons par ailleurs que $\mathcal{O}_L\dbl \Z_p^s\dbr^{\rm rig}$ s'identifie \`a l'alg\`ebre $D(\Z_p^s,L)$ des distributions localement analytiques sur $\Z_p^s$ \`a valeurs dans $L$ (voir par exemple \cite[Thm. 2.2]{STFour}) et que l'on a un isomorphisme de $L$-alg\`ebres de Fr\'echet-Stein $D(\Z_p^s\times G_0,L)\cong D(\Z_p^s,L)\widehat\otimes_LD(G_0,L)$ (\cite[Prop. A.3]{STdual}) o\`u $D(\Z_p^s\times G_0,L)$ est l'alg\`ebre des distributions localement analytiques sur $\Z_p^s\times G_0$ \`a valeurs dans $L$. La surjection $\mathcal{O}_L\dbl \Z_p^s\dbr\twoheadrightarrow R$ induit une application continue d'alg\`ebres de Fr\'echet-Stein $D(\Z_p^s,L)\cong \mathcal{O}_L\dbl \Z_p^s\dbr^{\rm rig}\rightarrow R^{\rm rig}$ (qui en fait est encore surjective par le Lemme \ref{quotient} en appendice). Par \cite[Thm. 7.1]{STdist}, le dual continu de l'espace des vecteurs localement analytiques de $\Pi$ pour l'action de $\Z_p^s\times G$ s'identifie \`a $D(\Z_p^s\times G_0,L)\otimes_{\mathcal{O}_L\dbl \Z_p^s\times G_0\dbr}\Pi'$.

\begin{lemm}\label{isom}
Soit $\Pi$ une repr\'esentation continue $R$-admissible de $G$ sur $L$. On a un isomorphisme de $D(\Z_p^s,L)\widehat\otimes_LD(G_0,L)$-modules de type fini 
\begin{equation*}
D(\Z_p^s\times G_0,L)\otimes_{\mathcal{O}_L\dbl \Z_p^s\times G_0\dbr}\Pi'\simeq(R^{\rm rig}\widehat{\otimes}_L D(G_0,L))\otimes_{R\dbl G_0\dbr} \Pi'.
\end{equation*}
\end{lemm}
\begin{proof}
Soit $I$ le noyau du morphisme $\mathcal{O}_L\dbl \Z_p^s\dbr\twoheadrightarrow R$. D'apr\`es le Lemme \ref{quotient} en appendice, on a une suite exacte courte (topologique) d'espaces de Fr\'echet sur $L$ 
$$0\rightarrow D(\Z_p^s,L)I\rightarrow D(\Z_p^s,L)\rightarrow R^{\rm rig}\rightarrow0.$$
De plus par \cite[Thm. 4.11]{STdist} $D(\Z_p^s,L)$ est (fid\`element) plat sur $L\otimes_{\mathcal{O}_L}\mathcal{O}_L\dbl \Z_p^s\dbr$, d'o\`u un isomorphisme d'espaces de Fr\'echet $D(\Z_p^s,L)\otimes_{\mathcal{O}_L\dbl \Z_p^s\dbr} I\simeq D(\Z_p^s,L)I$. Comme $D(G_0,L)$ est un espace de Fr\'echet sur $L$, le foncteur $D(G_0,L)\widehat{\otimes}_L-$ pr\'eserve les suites exactes courtes d'espaces de Fr\'echet sur $L$ (cf. par exemple \cite[Lem. 4.13]{SchraenGL3}), d'o\`u une suite exacte d'espaces de Fr\'echet
\begin{equation*}
0\rightarrow D(G_0,L)\widehat{\otimes}_L (D(\Z_p^s,L)I)\rightarrow D(G_0,L)\widehat{\otimes}_L D(\Z_p^s,L)\rightarrow D(G_0,L)\widehat{\otimes}_L R^{\rm rig}\rightarrow0
\end{equation*}
et un isomorphisme \[D(G_0,L)\widehat{\otimes}_L(D(\Z_p^s,L)I)\simeq (D(G_0,L)\widehat{\otimes}_LD(\Z_p^s,L))\otimes_{\mathcal{O}_L\dbl \Z_p^s\dbr} I\] (rappelons que $I$ est un $\mathcal{O}_L\dbl \Z_p^s\dbr$-module de type fini). Comme on a un isomorphisme $D(G_0\times \Z_p^s,L)\simeq D(G_0,L)\widehat{\otimes}_L D(\Z_p^s,L)$, on obtient une suite exacte courte
\begin{equation*}
0\rightarrow D(G_0\times\Z_p^s,L)\otimes_{\mathcal{O}_L\dbl \Z_p^s\dbr} I\rightarrow D(G_0\times\Z_p^s,L)\rightarrow D(G_0,L)\widehat{\otimes}_L R^{\rm rig}\rightarrow0
\end{equation*}
et donc un isomorphisme $D(G_0\times\Z_p^s,L)\otimes_{\mathcal{O}_L\dbl \Z_p^s\dbr}R\simeq D(G_0,L)\widehat{\otimes}_L R^{\rm rig}$. Avec l'isomorphisme (\'evident) $R\otimes_{\mathcal{O}_L\dbl \Z_p^s\dbr}\mathcal{O}_L\dbl \Z_p^s\times G_0\dbr\cong R\dbl G_0\dbr$, on en d\'eduit donc des isomorphismes de $D(\Z_p^s\times G_0,L)$-modules (de type fini)
\begin{align*}
\begin{aligned}
&D(\Z_p^s\times G_0,L)\otimes_{\mathcal{O}_L\dbl \Z_p^s\times G_0\dbr} \Pi' \simeq (D(\Z_p^s\times G_0,L)\otimes_{\mathcal{O}_L\dbl \Z_p^s\times G_0\dbr} R\dbl G_0\dbr)\otimes_{R\dbl G_0\dbr}\Pi'\\
\simeq & (\!D(\Z_p^s\times G_0,L)\!\otimes_{\mathcal{O}_L\dbl \Z_p^s\times G_0\dbr}\!(\!R\!\otimes_{\mathcal{O}_L\dbl \Z_p^s\dbr}\!\mathcal{O}_L\dbl \Z_p^s\!\times \!G_0\dbr))\!\otimes_{R\dbl G_0\dbr}\!\Pi'\\
\simeq &(D(\Z_p^s\times G_0,L)\otimes_{\mathcal{O}_L\dbl \Z_p^s\dbr} R)\otimes_{R\dbl G_0\dbr}\Pi'\\
\simeq &(R^{\rm rig}\widehat{\otimes}_L D(G_0,L))\otimes_{R\dbl G_0\dbr}\Pi'.
\end{aligned}
\end{align*}
\end{proof}

On a donc un isomorphisme d'espaces de Fr\'echet nucl\'eaires sur $L$
\begin{equation}\label{isodual}
(\Pi^{R-{\rm an}})'\simeq(R^{\rm rig}\widehat{\otimes}_LD(G_0,L))\otimes_{R\dbl G_0\dbr}\Pi'.
\end{equation}
qui montre bien par bidualit\'e (\cite[Cor. 1.4]{STdistr}) que le sous-espace des vecteurs localement $R$-analytiques de $\Pi$ ne d\'epend pas du choix de la pr\'esentation de $R$.

On termine cette section avec la d\'efinition suivante. Si $J\subset R$ est un id\'eal et $\Pi$ une repr\'esentation de Banach $R$-admissible de $G$, on note $\Pi[J]$ le sous-$L$-espace vectoriel des $v\in \Pi$ tels que $J\cdot v=0$ muni de la topologie induite. Si $J$ est engendr\'e par un seul \'el\'ement $r\in R$, en dualisant la suite exacte $0\rightarrow \Pi[J]\rightarrow \Pi \buildrel r \over \rightarrow \Pi$ (qui reste exacte car $\Pi$ est une repr\'esentation admissible de $\Z_p^s\times G$) on obtient un isomorphisme
\begin{equation}\label{invariantsdual}
\Pi[J]'\simeq \Pi'/J\Pi'
\end{equation}
compatible \`a $G$ et $R$. Par une r\'ecurrence imm\'ediate (car $J$ est de type fini), on en d\'eduit ce m\^eme isomorphisme pour tout id\'eal $J$. En particulier on voit que $\Pi[J]$ est une repr\'esentation $R/J$-admissible de $G$.

\subsection{La vari\'et\'e de Hecke-Taylor-Wiles}\label{HTW}

On fixe $\mathfrak{m}$ un id\'eal maximal automorphe non Eisenstein de niveau mod\'er\'e $U^p$ de $\mathbb{T}^S$ tel que $\mathbb{T}^S/\mathfrak{m}=k_L$ comme au \S\ \ref{automorphe}, et on note simplement $\rhobar$ au lieu de $\rhobar_{\mathfrak{m}}$.

On fixe $g\geq1$ un entier (pour l'instant quelconque) et on note $R_{\rhobar_{\tilde{v}}}^{\bar\square}$ le quotient maximal r\'eduit et sans $p$-torsion de $R_{\rhobar_{\tilde{v}}}^{\square}$  (cf.~\cite[\S\ 3]{Thorne}), $R^{\rm loc}=\widehat{\bigotimes}_{v\in S} R_{\rhobar_{\tilde{v}}}^{\bar\square}$ et $R_\infty=R^{\rm loc}\dbl x_1,\cdots, x_g\dbr$. On pose aussi
\[S_{\infty}=\mathcal{O}_L\dbl y_1,\cdots,y_q\dbr\]
o\`u $q=g+[F^+:\Q]\frac{n(n-1)}{2}+|S|n^2$ et on note $\mathfrak{a}=(y_1,\cdots,y_q)\subset S_\infty$. Quitte \`a rapetisser le niveau mod\'er\'e $U^p$ en des places totalement d\'ecompos\'ees (et \`a augmenter $S$ en cons\'equence), on suppose que pour tout $h\in G(\mathbb{A}_{F^+}^\infty)$ il v\'erifie
\begin{equation*}
G(F^+)\cap (hU^pK_ph^{-1})=\{1\}.
\end{equation*}

Notons que la ${\mathcal O}_L$-alg\`ebre $\mathbb{T}^S(U^pU_p,L)_{\mathfrak{m}}$ du \S\ \ref{automorphe} est r\'eduite (et sans $p$-torsion) pour tout sous-groupe ouvert compact $U_p$ de $G_p$ : cela se d\'eduit de la semi-simplicit\'e de la $G(\mathbb{A}_{F^+}^\infty)$-repr\'esentation $\varinjlim_{U} S(U,L)\otimes_L\bar\Q_p$ o\`u la limite inductive est prise sur les sous-groupes compacts ouverts $U$ de $G(\mathbb{A}_{F^+}^\infty)$ (cf. par exemple \cite[Prop.\ 3.3.2]{CHT}). En particulier l'action de $R_{\rhobar,S}$ sur $\widehat{S}(U^p,L)_{\mathfrak m}$ se factorise par $R_{\rhobar,S}\twoheadrightarrow R_{\rhobar,\mathcal S}$ o\`u $R_{\rhobar,\mathcal S}$ est l'anneau de d\'eformation associ\'e au probl\`eme de d\'eformation (avec les notations de \cite[\S\ 2.3]{CHT})
$${\mathcal S}=\big(F/F^+, S, \widetilde S,{\mathcal O}_L, \rhobar, \varepsilon^{1-n}\delta^n_{F/F^+}, \{R^{\bar\square}_{\rhobar_{\tilde v}}\}_{v\in S}\big)$$
avec $\widetilde S=\{\widetilde v\ \vert\ v\in S\}$ et $\delta_{F/F^+}$ le caract\`ere quadratique de $\mathrm{Gal}(\bar{F}/F^+)$ associ\'e \`a l'extension $F/F^+$. Notre probl\`eme global de d\'eformations n'est pas cadr\'e, l'anneau $R_{\rhobar,\mathcal S}$ est donc ici l'anneau not\'e $R_{\mathcal S}^{\rm univ}$ dans \cite{Thorne}.

Rappelons qu'une repr\'esentation continue $R_\infty$-admissible de $G_p$ sur $L$ est dite unitaire si pour un choix de norme la boule unit\'e est stable par $G_p$ et par $R_\infty$. Le th\'eor\`eme suivant est essentiellement contenu dans les r\'esultats de \cite[\S\ 2]{CEGGPS}.

\begin{theo}\label{TaylorWiles}
Supposons $p\geq 2n+2$ et que $\bar\rho$ reste absolument irr\'eductible apr\`es restriction \`a $G_{F(\zeta_p)}$ o\`u $\zeta_p$ est une racine primitive $p$-i\`eme de l'unit\'e. Il existe un entier $g\geq 1$ tel que l'on ait une repr\'esentation continue $R_\infty$-admissible unitaire $\Pi_\infty$ de $G_p$ sur $L$ et des morphismes de $\mathcal{O}_L$-alg\`ebres locales $S_\infty\rightarrow R_\infty$ et $R_\infty\rightarrow R_{\rhobar,\mathcal S}$ v\'erifiant les propri\'et\'es suivantes
\begin{itemize}
\item[(i)]il existe une boule unit\'e $\Pi_\infty^0\subset\Pi_\infty$ stable par $G_p$ et $R_\infty$ telle que \[M_\infty=(\Pi_\infty^0)'=\Hom_{{\mathcal O}_L}(\Pi_\infty^0,{\mathcal O}_L)\] est un $S_\infty\dbl K_p\dbr$-module projectif de type fini (via $S_\infty\rightarrow R_\infty$);
\item[(ii)]il existe un isomorphisme $R_\infty/\mathfrak{a}R_\infty\simeq R_{\rhobar,\mathcal S}$ de $\mathcal{O}_L$-alg\`ebres locales noeth\'eriennes compl\`etes et un isomorphisme de repr\'esentations continues $R_\infty/\mathfrak{a}R_\infty$-admissibles unitaires de $G_p$ sur $L$
\begin{equation}\label{isopatching}
\Pi_\infty[\mathfrak{a}]\simeq\widehat{S}(U^p,L)_{\mathfrak{m}}.
\end{equation}
\end{itemize}
\end{theo}
\begin{proof}
Une telle repr\'esentation continue est essentiellement construite dans \cite[Cor.\ 2.9]{CEGGPS} en prenant pour $\Pi_\infty$ le dual continu du $R_\infty\dbl K_p\dbr$-module de type fini $M_\infty$ de \emph{loc.~cit.} (cf. \cite[Thm. 3.5]{STBan}, on utilise comme au \S\ \ref{analytique} que $R_\infty\dbl K_p\dbr$ est un quotient de l'alg\`ebre d'Iwasawa de $\Z_p^s\times K_p$ pour $s$ convenable). La seule diff\'erence avec \cite[\S\ 2]{CEGGPS} est qu'il faut remplacer la place $\mathfrak p$ de $F^+$ fix\'ee au-dessus de $p$ par l'ensemble $S_p$ (i.e. toutes les places au-dessus de $p$) et l'ensemble $T=S_p\cup \{v_1\}$ (o\`u $v_1$ est une place hors $p$ d\'ecompos\'ee particuli\`ere dont le but est d'obtenir des propri\'et\'es de multiplicit\'e 1 dont nous n'avons pas besoin, cf. \cite[\S\ 2.3]{CEGGPS}) par l'ensemble $T=S$. En rempla\c cant donc le probl\`eme de d\'eformation ${\mathcal S}_{Q_N}$ de \cite[\S\ 2.5]{CEGGPS} par le probl\`eme de d\'eformation (en adaptant les notations de \emph{loc.~cit.})
$$\big(F/F^+, S\cup Q_N, \widetilde S\cup \widetilde Q_N,{\mathcal O}_L, \rhobar, \varepsilon^{1-n}\delta^n_{F/F^+}, \{R^{\bar\square}_{\rhobar_{\tilde v}}\}_{v\in S}\cup \{R^{\bar\psi_{\tilde v}}_{\rhobar_{\tilde v}}\}_{v\in Q_N}\big),$$
une preuve analogue \`a celle de \cite[Cor.\ 2.9]{CEGGPS} donne le $M_\infty$ recherch\'e.
\end{proof}

La repr\'esentation continue $\Pi_\infty$ du Th\'eor\`eme \ref{TaylorWiles} est donc \`a la fois une repr\'esentation $R_\infty$-admissible de $G_p$ et une repr\'esentation $S_\infty$-admissible.

Rappelons que nous avons fix\'e au \S\ \ref{automorphe} un sous-groupe de Borel $B_p\subset G_p$ de tore diagonal $T_p$. En appliquant le foncteur $J_{B_p}$ d\'efini par Emerton dans \cite[Def. 3.4.5]{EmertonJacquetI} \`a la repr\'esentation $\Pi^{R_\infty-{\rm an}}$ de $G_p$ (cf. Remarque \ref{addenda} ci-dessous) et en prenant le dual continu, on obtient un $R_\infty^{\rm rig}\widehat{\otimes}_L\mathcal{O}(\widehat{T}_{p,L})$-module coadmissible $J_{B_p}(\Pi_\infty^{R_\infty-{\rm an}})'$ au sens de \cite[\S\ 3]{STdist} (voir aussi \cite[\S\ 1.2]{Emertonlocan}), qui lui-m\^eme correspond \`a un unique faisceau coh\'erent $\mathcal{M}_\infty$ sur $\Spf(R_\infty)^{\rm rig}\times_L\widehat{T}_{p,L}$ (au sens de \cite[\S\ 9.4]{BGR}) tel que \[\Gamma\left(\Spf(R_\infty)^{\rm rig}\times_L\widehat{T}_{p,L},\mathcal{M}_\infty\right)\simeq J_{B_p}(\Pi_\infty^{R_\infty-{\rm an}})'\] (voir encore par exemple \cite[\S\ 3]{STdist} et la discussion qui pr\'ec\`ede le Lemme \ref{isom}).

\begin{rema}\label{addenda}
{\rm Pr\'ecisons un peu le point pr\'ec\'edent. \`A proprement parler, Emerton ne construit le foncteur $J_P$ que pour une repr\'esentation localement analytique du groupe des $\mathbb{Q}_p$-points d'un sous-groupe parabolique $P$ d'un groupe r\'eductif d\'efini sur $\Q_p$. Mais si l'on fixe une pr\'esentation $\mathcal{O}_L\dbl \Z_p^s\dbr\twoheadrightarrow R_\infty$ de l'alg\`ebre $R_\infty$ comme au \S\ \ref{analytique}, on peut voir $\Pi_\infty^{R_\infty-{\rm an}}$ comme une repr\'esentation localement analytique du groupe $\Z_p^s\times G_p$ (qui ne d\'epend pas de la pr\'esentation, cf. \S\ \ref{analytique}). Il ne s'agit pas du groupe des $\Q_p$-points d'un groupe r\'eductif, mais il n'est pas difficile de v\'erifier que la construction de $J_{B_p}$ s'\'etend au groupe $\Z_p^s\times G_p$.}
\end{rema}

On suppose que l'extension $L$ de $\Q_p$ est suffisamment grande de telle sorte que, pour tout $v\in S$, les composantes irr\'eductibles du sch\'ema $\Spec(R_{\rhobar_{\tilde{v}}}^{\bar\square}[1/p])$ soient g\'eom\'etriquement irr\'eductibles (i.e. $C\times_{L}\bar{\Q}_p$ reste un sch\'ema irr\'eductible pour toute composante irr\'eductible $C$ de $\Spec(R_{\rhobar_{\tilde{v}}}^{\bar\square}[1/p])$). 

Si $X$ est un espace analytique rigide et $\mathcal{M}$ un faisceau coh\'erent sur $X$, on appelle {\it support sch\'ematique} de $\mathcal{M}$ la vari\'et\'e analytique rigide dont l'espace sous-jacent est le ferm\'e analytique $\{x\in X|\,\mathcal{M}_x\neq0\}$ de $X$ (cf. (\cite[\S\ 9.5.2 Prop.4]{BGR})) et le faisceau structural est donn\'e par $\mathcal{O}_X/\mathcal{I}$ o\`u $\mathcal{I}\subset\mathcal{O}_X$ est le faisceau coh\'erent d'id\'eaux annulateurs de $\mathcal{M}$.

\begin{defi}
On note $X_p(\rhobar)\subset\Spf(R_\infty)^{\rm rig}\times_L\widehat{T}_{p,L}$ le support sch\'ematique du faisceau coh\'erent $\mathcal{M}_\infty$. On note \'egalement $X_p(\rhobar)^{\rm red}$ la nilr\'eduction de $X_p(\rhobar)$, au sens de \cite[p. 389]{BGR}, que l'on appelle \emph{vari\'et\'e de Hecke-Taylor-Wiles}.
\end{defi}
Nous prouverons en fait plus loin que $X_p(\rhobar)$ est un espace analytique r\'eduit, de sorte que $X_p(\rhobar)=X_p(\rhobar)^{\rm red}$.

Il sera utile dans la suite de d\'esigner par $\mathfrak{X}^\square_{\rhobar^p}$ l'espace analytique rigide $\Spf(\widehat{\bigotimes}_{v\in S\backslash S_p}R_{\rhobar_{\tilde{v}}}^{\bar\square})^{\rm rig}$ et par $\mathfrak{X}^\square_{\rhobar_p}$ l'espace $\Spf(\widehat{\bigotimes}_{v\in S_p}R_{\rhobar_{\tilde{v}}}^{\bar\square})^{\rm rig}$. On note $\mathbb{U}$ la boule ouverte de rayon $1$, autrement dit l'espace analytique rigide $\Spf(\mathcal{O}_L\dbl T\dbr)^{\rm rig}$. On a $\Spf(R_\infty)^{\rm rig}\simeq\mathfrak{X}^\square_{\rhobar_p}\times\mathfrak{X}^\square_{\rhobar^p}\times\mathbb{U}^g$ par \cite[Prop. 7.2.4.g)]{DeJong}.

Si $V$ est une repr\'esentation localement analytique de $T_{p}$ et $\delta$ est un caract\`ere localement analytique de $T_{p}$, on note $V^{T_p=\delta}$ le sous-espace de $V$ o\`u l'action de $T_p$ se factorise par $\delta$.

\begin{prop}\label{points}
Soit $x=(y,\delta)\in (\Spf(R_\infty)^{\rm rig}\times\widehat{T}_{p,L})(L)$. Notons $\mathfrak{p}_x$ l'id\'eal maximal de $R_\infty[1/p]$ correspondant au point $y$. Alors $x\in X_p(\rhobar)$ si et seulement si $(J_{B_p}(\Pi_\infty[\mathfrak{p}_x]^{\rm an}))^{T_p=\delta}$ est non nul.
\end{prop}
\begin{proof}
De m\^eme qu'en \cite[Prop. $2.3.3(iii)$]{Emint}, la fibre de $\mathcal{M}_\infty$ en $x$ est duale de
\begin{equation*}
(J_{B_p}(\Pi_\infty^{R_\infty-{\rm an}})[\mathfrak{p}_x])^{T_p=\delta}.
\end{equation*}
Pour conclure on utilise l'isomorphisme $J_{B_p}(\Pi_\infty[\mathfrak{p}_x]^{\rm an})\simeq J_{B_p}(\Pi_\infty^{R_\infty-{\rm an}})[\mathfrak{p}_x]$ qui r\'esulte de l'exactitude du foncteur $\Pi\mapsto\Pi^{\rm an}$ (qui est une cons\'equence de \cite[Thm.~7.1]{STdist}), de l'exactitude \`a gauche du foncteur $J_{B_p}$ (\cite[Lem.~3.4.7]{EmertonJacquetI}) et de l'isomorphisme \eqref{invariantsdual}.
\end{proof}

\subsection{Premi\`eres propri\'et\'es}

\begin{prop}\label{egalitevectan}
L'inclusion $\Pi_\infty^{R_\infty-{\rm an}}\subset\Pi_\infty^{S_\infty-{\rm an}}$ provenant du morphisme $S_\infty\rightarrow R_\infty$ est une \'egalit\'e.
\end{prop}
\begin{proof}
Notons $A$ l'image de $R_\infty\dbl K_p\dbr$ dans $\mathrm{End}(M_\infty)$. Fixons $e_1,\cdots,e_m$ une famille de g\'en\'erateurs du $S_\infty\dbl K_p\dbr$-module de type fini $M_\infty$ (voir (i) du Th\'eor\`eme \ref{TaylorWiles}). L'application $f\mapsto (f(e_1),\cdots, f(e_m))$, o\`u $f\in A\subset\mathrm{End}(M_\infty)$, permet d'identifier $A$ \`a un sous-$S_\infty\dbl K_p\dbr$-module de $M_\infty^m$. L'isomorphisme $S_\infty\dbl K_p\dbr\simeq\mathcal{O}_L\dbl\Z_p^q\times K_p\dbr$ ainsi que \cite[Prop. V.2.2.4]{Lazard} (combin\'e au fait que $K_p$ contient un pro-$p$-sous-groupe analytique d'indice fini) montrent que $S_\infty\dbl K_p\dbr$ est un anneau noeth\'erien. On en d\'eduit que $A$ est un $S_\infty\dbl K_p\dbr$-module de type fini. D'apr\`es le Lemme \ref{annexe} en appendice appliqu\'e \`a $S_\infty\dbl K_p\dbr\rightarrow A$ et $R_\infty\dbl K_p\dbr\rightarrow A$, on a $(S_\infty\dbl K_p\dbr)^{\rm rig}\otimes_{S_\infty\dbl K_p\dbr} A\simeq A^{\rm rig}\simeq R_\infty\dbl K_p\dbr^{\rm rig}\otimes_{R_\infty\dbl K_p\dbr}A$, ce qui prouve, par \eqref{isodual} et \cite[Prop.~7.2.4.g]{DeJong}, que
\begin{align*}
\begin{aligned}
(\Pi_\infty^{S_\infty-{\rm an}})'&\simeq (S_\infty^{\rm rig}\widehat{\otimes}_L D(K_p,L))\otimes_{S_\infty\dbl K_p\dbr} \Pi_\infty'\\
&\simeq (S_\infty\dbl K_p\dbr^{\rm rig}\otimes_{S_\infty\dbl K_p\dbr}A)\otimes_A\Pi_\infty'\\
&\simeq A^{\rm rig}\otimes_A \Pi_\infty'\\
&\simeq R_\infty\dbl K_p\dbr^{\rm rig}\otimes_{R_\infty\dbl K_p\dbr}\Pi_\infty'\\
&\simeq (R_\infty^{\rm rig}\widehat{\otimes}_L D(K_p,L))\otimes_{R_\infty\dbl K_p\dbr}\Pi_\infty'\\
&\simeq (\Pi_\infty^{R_\infty-{\rm an}})'.
\end{aligned}
\end{align*}
\end{proof}

Notons $\mathcal{W}_\infty$ l'espace $\Spf(S_\infty)^{\rm rig}\times\widehat{T}_{p,L}^0$ qui joue ici le m\^eme r\^ole que l'espace des poids pour les vari\'et\'es de Hecke. On d\'efinit en effet une application poids $\omega_X$ de $X_p(\rhobar)$ dans $\mathcal{W}_\infty$ comme la compos\'ee de l'inclusion de $X_p(\rhobar)$ dans $\Spf(R_\infty)^{\rm rig}\times\widehat{T}_{p,L}$ avec l'application de $\Spf(R_\infty)^{\rm rig}\times\widehat{T}_{p,L}$ vers $\Spf(S_\infty)^{\rm rig}\times\widehat{T}_{p,L}^0=\mathcal{W}_\infty$ d\'eduite de la structure de $S_\infty$-alg\`ebre de $R_\infty$ et de la restriction $\widehat{T}_{p,L}\rightarrow\widehat{T}_{p,L}^0$. Soit $z\in T_p$ un \'el\'ement tel que $|\alpha(z)|_p<1$ pour toute racine positive $\alpha$, relativement \`a $B_p$, du groupe $G\times_{F^+}(F^+\otimes_{\Q}\Q_p)$. On peut par exemple choisir $z=(z_v)_{v\in S_p}\in\prod_{v\in S_p}T_v$, o\`u $z_v$ d\'esigne l'\'el\'ement diagonal $(\varpi_v^{n-1},\dots,\varpi_v,1)$. Il d\'efinit une application de restriction $\Spf(S_\infty)^{\rm rig}\times\widehat{T}_{p,L}\rightarrow\mathcal{W}_\infty\times\mathbb{G}_{m,L}$, d\'eduite de l'inclusion $z^{\mathbb{Z}}\subset T_p$ et de l'isomorphisme d'espaces analytiques rigides $\widehat{\Z}\simeq\mathbb{G}_m$. On note alors 
\[f:\Spf(R_\infty)^{\rm rig}\times\widehat{T}_{p,L}\longrightarrow \mathcal{W}_\infty\times\mathbb{G}_{m,L}\] 
sa compos\'ee avec le morphisme $\Spf(R_\infty)^{\rm rig}\times\widehat{T}_{p,L}\rightarrow\Spf(S_\infty)^{\rm rig}\times\widehat{T}_{p,L}$ provenant de la structure de $S_\infty$-alg\`ebre de $R_\infty$. Finalement on note $g$ la projection $\mathcal{W}_\infty\times\mathbb{G}_{m,L}\rightarrow\mathcal{W}_\infty$.

Notons $Y$ le sous-groupe ferm\'e de $T_p$ engendr\'e par $T_p^0$ et $z$. Ce sous-groupe v\'erifie les conditions de la discussion pr\'ec\'edant \cite[Prop. 3.2.27]{EmertonJacquetI}. On d\'eduit alors de \emph{loc.~cit.}, de \cite[Prop.~3.2.23]{EmertonJacquetI} et de la Proposition \ref{existencefactorisation} de l'appendice que $J_{B_p}(\Pi_\infty^{R_\infty-{\rm an}})'$ est \'egalement un $\mathcal{O}(\mathcal{W}_\infty\times\mathbb{G}_{m,L})$-module coadmissible en identifiant $\mathcal{W}_\infty\times\mathbb{G}_{m,L}$ \`a $\Spf(S_\infty)^{\rm rig}\times\widehat{Y}$. Comme pr\'ec\'edemment, il s'agit de l'espace des sections globales d'un faisceau coh\'erent $\mathcal{N}_\infty$ de $\mathcal{O}_{\mathcal{W}_\infty\times\mathbb{G}_{m,L}}$-modules. Notons $Z_z(\rhobar)$ le support sch\'ematique de ce faisceau coh\'erent dans $\mathcal{W}_\infty\times\mathbb{G}_{m,L}$. L'application $f$ de $X_p(\rhobar)$ dans $\mathcal{W}_\infty\times\mathbb{G}_{m,L}$ se factorise par $Z_z(\rhobar)$, on note encore $f$ l'application $f:\, X_p(\rhobar)\rightarrow Z_z(\rhobar)$ obtenue ainsi. De m\^eme, on note encore $g$ la restriction de $g$ \`a $Z_z(\rhobar)$. On obtient ainsi un diagramme commutatif
\begin{equation*}
\xymatrix{X_p(\rhobar)\ar^{f}[r]\ar_{\omega_X}[dr]& Z_z(\rhobar)\ar^{g}[d]\\ &\mathcal{W}_\infty.}
\end{equation*}
Les espaces $X_p(\rhobar)$ et $Z_z(\rhobar)$ sont quasi-Stein puisque ce sont des ferm\'es analytiques de $\Spf(R_\infty)\times_L\widehat{T}_{p,L}$ et $\mathcal{W}_\infty\times_L\mathbb{G}_{m,L}$ qui sont eux-m\^emes quasi-Stein.

Rappelons qu'une hypersurface de Fredholm de $\mathcal{W}_\infty\times\mathbb{G}_{m,L}$ est un ferm\'e analytique de la forme $Z(F)$ pour $F\in1+T\mathcal{O}(\mathcal{W}_\infty)\{\{T\}\}$, o\`u
\begin{equation*}
Z(F)=\{(x,t)\in\mathcal{W}_\infty\times\mathbb{G}_{m,L}|\, F(x,t^{-1})=0\},
\end{equation*}
et $\mathcal{O}(\mathcal{W}_\infty)\{\{T\}\}$ d\'esigne l'anneau des s\'eries $\sum_{n\geq0} a_nT^n$ convergeant sur $\mathcal{W}_\infty\times\mathbb{A}_L^1$ (dans \cite[\S4]{ConradIrred}, la notion d'hypersurface de Fredholm est d\'efinie en rempla\c{c}ant $\mathbb{G}_{m,L}$ par $\mathbb{A}_L^1$, mais cela ne change rien puisque l'ensemble des z\'eros d'une telle s\'erie $F$ est contenu dans l'ouvert $\mathcal{W}_\infty\times\mathbb{G}_{m,L}$ de $\mathcal{W}_\infty\times\mathbb{A}_L^1$).

\begin{lemm}\label{Fredholm}
Le ferm\'e analytique sous-jacent \`a $Z_z(\rhobar)$ est une hypersurface de Fredholm de $\mathcal{W}_\infty\times\mathbb{G}_{m,L}$. De plus, il existe un recouvrement admissible $(U'_i)_{i\in I}$ de $Z_z(\rhobar)$ par des affino\"ides $U'_i$ tels que $g$ induise une application finie surjective de $U'_i$ sur un ouvert affino\"ide $W_i$ de $\mathcal{W}_\infty$ et que $U'_i$ soit une composante connexe de $g^{-1}(W_i)$. Enfin, pour un tel recouvrement, pour tout $i\in I$, le module $\Gamma(U'_i,\mathcal{N}_\infty)$ est un $\mathcal{O}_{\mathcal{W}_\infty}(W_i)$-module projectif de type fini.
\end{lemm}

\begin{proof}
D'apr\`es la Proposition \ref{existencefactorisation} de l'appendice, il existe un recouvrement admissible de $\mathcal{W}_\infty$ par des ouverts affino\"{\i}des $U_1\subset U_2\subset\cdots\subset U_h\subset\cdots$, des $A_h=\mathcal{O}_{\mathcal{W}_\infty}(U_h)$-modules de Banach de type $(Pr)$, $V_h$, munis d'endomorphismes $A_h$-compacts $z_h$, un sous-groupe compact ouvert $N_0\subset N_p$ et un diagramme commutatif
\begin{equation}\label{compactdiagramme}
\begin{aligned}
\xymatrix{((\Pi_\infty^{R_\infty-{\rm an}})^{N_0})'\ar[r]\ar_{z}[d]&\cdots\ar[r]&V_{h+1}\ar_{z_{h+1}}[d]\ar[r]&V_{h+1}\widehat{\otimes}_{A_{h+1}}A_h\ar_{z_{h+1}\otimes1_{A_h}}[d]\ar^(.7){\beta_h}[r]&V_h\ar[r]\ar^{z_h}[d]\ar_{\alpha_h}[ld]&\cdots\\
((\Pi_\infty^{R_\infty-{\rm an}})^{N_0})'\ar[r]&\cdots\ar[r]&V_{h+1}\ar[r]&V_{h+1}\widehat{\otimes}_{A_{h+1}}A_h\ar[r]^(.7){\beta_h}&V_h\ar[r]&\cdots}
\end{aligned}
\end{equation}
de sorte que les lignes horizontales induisent un isomorphisme topologique \[((\Pi_\infty^{R_\infty-{\rm an}})^{N_0})'\simeq\varprojlim_hV_h\] (en ce qui concerne les notions de $A_h$-module de Banach de type $(Pr)$ et d'endomorphisme $A_h$-compact, nous renvoyons \`a \cite[\S2]{Buzzard}).

Notons $F_h$ la s\'erie caract\'eristique de $z_h$ (cf. par exemple la fin de \cite[\S2]{Buzzard}). Il s'agit d'un \'el\'ement de $A_h\{\{T\}\}$. Les factorisations $z_h=\beta_h\circ\alpha_h$ et $z_{h+1}\otimes1_{A_h}=\alpha_h\circ\beta_h$, ainsi que \cite[Lemma 2.12]{Buzzard} montrent que $F_h$ co\"{\i}ncide avec la s\'erie caract\'eristique de $z_{h+1}\otimes1_{A_h}$. On d\'eduit alors de \cite[Lemma 2.13]{Buzzard} que l'image de $F_{h+1}$ dans $A_h\{\{T\}\}$ co\"{\i}ncide avec $F_h$. On note $F$ l'\'el\'ement $(F_h)_{h\geq1}$ de $\mathcal{O}(\mathcal{W}_\infty)\{\{T\}\}$.

Montrons que le ferm\'e analytique sous-jacent \`a $Z_z(\rhobar)$ est exactement l'ensemble des points $(y,\lambda)\in \mathcal{W}_\infty\times\mathbb{G}_{m,L}$ tels que $F(y,\lambda^{-1})=0$. Fixons un tel $x=(y,\lambda)\in \mathcal{W}_\infty\times\mathbb{G}_{m,L}$ et $h$ tel que $y\in U_h$. Comme l'espace $\mathcal{W}_\infty\times\mathbb{G}_{m,L}$ est quasi-Stein, \cite[Satz 2.4.3]{KiehlAB} montre que la fibre de $\mathcal{N}_\infty$ en $x$ est non nulle si et seulement si $\Gamma(U_h\times\mathbb{G}_{m,L},\mathcal{N}_\infty)\otimes_{\mathcal{O}(U_h\times\mathbb{G}_{m,L})} k(x)\neq0$, o\`u $k(x)$ d\'esigne le corps r\'esiduel de $W_\infty\times_L\mathbb{G}_{m,L}$ en $x$. Les lemmes \ref{qs1} et \ref{qs2} en appendice montrent que l'on a un isomorphisme \[\Gamma(U_h\times_L\mathbb{G}_{m,L},\mathcal{N}_\infty)\simeq\varprojlim_{h'\geq h}(V_{h'}\widehat{\otimes}_{A_{h'}}A_h).\] On d\'eduit alors de l'existence des fl\`eches $\alpha_{h'}$ et $\beta_{h'}$, ainsi que du diagramme \eqref{compactdiagramme}, que la fibre de $\mathcal{N}_\infty$ en $x$ est non nulle si et seulement si $V_h/((z_h-\lambda)V_h+\mathfrak{p}_yV_h)\neq0$, o\`u $\mathfrak{p}_y$ est l'id\'eal maximal de $A_h$ associ\'e au point $y$. Or d'apr\`es \cite[Cor. 22.9]{nfa}, l'application $1-\lambda^{-1} z_h$ est un endomorphisme de Fredholm d'indice $0$ de $V_h/\mathfrak{p}_yV_h$ ($z_h$ \'etant un endomorphisme $A_h$-compact de $V_h$, il induit un endomorphisme compact de $V_h/\mathfrak{p}_yV_h$), ce qui signifie que son noyau et son conoyau sont de m\^eme dimension finie (cf. \cite[\S22]{nfa} pour les notions d'endomorphisme de Fredholm et d'indice). Ainsi $(y,\lambda)\in X_p(\rhobar)$ si et seulement si le noyau de $1-\lambda^{-1}z_h$ agissant sur $V_h/\mathfrak{p}_yV_h$ est non nul, ce qui \'equivaut \`a demander que $\lambda^{-1}$ soit un z\'ero de la s\'erie caract\'eristique de $z_h$ sur $V_h/\mathfrak{p}_yV_h$ (voir \cite[Prop. 3.2]{Buzzard} par exemple). On conclut alors en remarquant que cette s\'erie caract\'eristique est l'image de $F_h\in A_h\{\{T\}\}$ dans $k(y)\{\{T\}\}$ par r\'eduction modulo $\mathfrak{p}_y$ (\cite[Lemma 2.13]{Buzzard}).

D'apr\`es \cite[Thm. 4.6]{Buzzard}, il existe un recouvrement admissible du ferm\'e analytique de $Z_z(\rhobar)$ par des ouverts affino\"ides $U'_i$ tels que l'application $g:\, Z_z(\rhobar)\rightarrow \mathcal{W}_\infty$ induise une application finie surjective $U'_i\rightarrow W_i$ avec $W_i$ un ouvert affino\"ide de $\mathcal{W}_\infty$ et $U'_i$ soit une composante connexe de $g^{-1}(W_i)$. Fixons $i\in I$. La discussion des cinq premiers paragraphes de \cite[\S\ 5]{Buzzard} montre alors que $U_i'$ est un ferm\'e analytique de $W_i\times_L\mathbb{G}_{m,L}$ d\'efini par un id\'eal engendr\'e par un polyn\^ome $Q(T)\in1+T\mathcal{O}_{\mathcal{W}_\infty}(W_i)[T]$ et tel que $F(T)=Q(T)S(T)$ o\`u $S(T)$ est une s\'erie de Fredholm dans $1+T\mathcal{O}_{\mathcal{W}_\infty}(W_i)\{\{T\}\}$ telle que $(Q,S)=1$. On obtient ainsi, d'apr\`es le Lemme \ref{qs1},
\begin{equation}\label{sections}
\Gamma(U_i',\mathcal{N}_\infty)=\varprojlim_{U_h\supset W_i}\left( (V_h\widehat{\otimes}_{A_h\{\{T\}\}}\mathcal{O}_{\mathcal{W}_\infty}(W_i)\{\{T\}\})/Q(T)\right),
\end{equation}
o\`u $T$ agit sur $V_h$ via $z_h^{-1}$. Or on peut d\'ecomposer $V_h\widehat{\otimes}_{A_h}\mathcal{O}_{\mathcal{W}_\infty}(W_i)$ en une somme directe $N_i\oplus F_i$ o\`u $N_i$ est un $\mathcal{O}_{\mathcal{W}_\infty}(W_i)$-module projectif de rang $\deg Q$ sur lequel $Q^*(z_h)$ s'annule et $Q^*(z_h)$ est inversible sur $F_i$ (suivant \cite[\S3]{Buzzard}, $Q^*(T)=T^{\deg Q}Q(T^{-1})$). Ainsi chaque terme du syst\`eme projectif \eqref{sections} est un $\mathcal{O}_{\mathcal{W}_\infty}(W_i)$-module projectif de rang $\deg Q$ et l'existence des fl\`eches $\alpha_h$ ainsi que le fait que $Q^*(z_h)$ soit inversible sur $F_i$ montrent que les fl\`eches de transition du syst\`eme \eqref{sections} sont des isomorphismes. Ainsi $\Gamma(U_i',\mathcal{N}_\infty)$ est un $\mathcal{O}_{\mathcal{W}_\infty}(W_i)$-module projectif de type fini.
\end{proof}

\begin{prop}\label{recouvrement}
Il existe un recouvrement admissible affino\"ide $(U_i)_{i\in I}$ de $X_p(\rhobar)$ tel que pour tout $i$, il existe $W_i$ ouvert affino\"ide de $\mathcal{W}_\infty$ tel que $\omega_X$ induise, en restriction \`a chaque composante irr\'eductible de $U_i$, un morphisme fini surjectif sur $W_i$ et que $\mathcal{O}_{X_p(\rhobar)}(U_i)$ soit isomorphe \`a une $\mathcal{O}_{\mathcal{W}_\infty}(W_i)$-alg\`ebre d'endomorphismes d'un $\mathcal{O}_{\mathcal{W}_\infty}(W_i)$-module projectif de rang fini.
\end{prop}
\begin{proof}
On fixe $(U_i')_{i\in I}$ un recouvrement affino\"ide admissible de $Z_z(\rhobar)$ comme dans le Lemme \ref{Fredholm}, on pose $U_i=f^{-1}(U_i')$ et $W_i=g(U_i')$. La famille $(U_i)_{i\in I}$ est alors un recouvrement ouvert admissible de $X_p(\rhobar)$. Nous allons prouver que chaque $U_i$ est affino\"ide et que $\mathcal{O}_{X_p(\rhobar)}(U_i)$ est isomorphe \`a une sous-$\mathcal{O}_{\mathcal{W}_\infty}(W_i)$-alg\`ebre de $\End_{\mathcal{O}_{W_\infty}(W_i)}(\Gamma(U_i',\mathcal{N}_\infty))$. Ceci implique en particulier que $\mathcal{O}_{X_p(\rhobar)}(U_i)$ est un $\mathcal{O}_{\mathcal{W}_\infty}(W_i)$-module fini et que l'application $\mathcal{O}_{\mathcal{W}_\infty}(W_i)\rightarrow\mathcal{O}_{X_p(\rhobar)}(U_i)$ est injective. Comme un morphisme fini d'espaces rigides est un morphisme propre, en particulier ferm\'e, on en d\'eduit que $\omega_X$ induit une application finie surjective de $U_i$ sur $W_i$. La surjectivit\'e en restriction \`a chaque composante irr\'eductible est alors une cons\'equence de \cite[Lem.~6.2.10]{CheHecke}.

Posons $M=J_{B_p}(\Pi_\infty^{R_\infty-{\rm an}})'$. D'apr\`es le Lemme \ref{qs2} en appendice, on a un isomorphisme entre $\mathcal{O}_{X_p(\rhobar)}(U_i)$ et $\mathcal{O}(X_p(\rhobar))\widehat{\otimes}_{\mathcal{O}(Z_z(\rhobar))}\mathcal{O}_{Z_z(\rhobar)}(U_i')$. On d\'eduit donc du Lemme \ref{qs1} un isomorphisme de $\mathcal{O}_{Z_z(\rhobar)}(U_i')$-modules
\begin{equation*}
\Gamma(U_i,\mathcal{M}_\infty)\simeq\mathcal{O}_{X_p(\rhobar)}(U_i)\widehat{\otimes}_{\mathcal{O}(X_p(\rhobar))}M\simeq\Gamma(U_i',\mathcal{N}_\infty).
\end{equation*}
Ainsi, $\Gamma(U_i,\mathcal{M}_\infty)$ est en particulier un $\mathcal{O}_{\mathcal{W}_\infty}(W_i)$-module projectif de type fini par le Lemme \ref{Fredholm}. Comme $X_p(\rhobar)$ est le support de $\mathcal{M}_\infty$, l'action de $\mathcal{O}_{X_p(\rhobar)}(U_i)$ sur $\Gamma(U_i,\mathcal{M}_\infty)$ est fid\`ele, en particulier $\mathcal{O}_{X_p(\rhobar)}(U_i)$ est isomorphe \`a une sous-$\mathcal{O}_{\mathcal{W}_\infty}(W_i)$-alg\`ebre de l'alg\`ebre des endomorphismes $\mathcal{O}_{W_\infty}(W_i)$-lin\'eaires de $\Gamma(U_i,\mathcal{M}_\infty)$.
\end{proof}

\begin{coro}\label{equidim}
L'espace $X_p(\rhobar)$ est \'equidimensionnel de dimension \[g+[F^+:\mathbb{Q}]\frac{n(n+1)}{2}+|S|n^2\] et n'a pas de composante immerg\'ee. De plus le morphisme $f:\, X_p(\rhobar)\rightarrow Z_z(\rhobar)$ est fini et l'image d'une composante irr\'eductible de $X_p(\rhobar)$ par $f$ est une composante irr\'eductible de $Z_z(\rhobar)$.
\end{coro}

\begin{proof}
Consid\'erons un recouvrement de $X_p(\rhobar)$ par des ouverts affino\"{\i}des $U$ comme dans la Proposition \ref{recouvrement}. Pour tout $U$, posons $V=\omega_X(U)$, alors $B=\mathcal{O}_{X_p(\rhobar)}(U)$ est une A=$\mathcal{O}_{\mathcal{W}_\infty}(V)$-alg\`ebre agissant fid\`element sur un $A$-module projectif de type fini $M$. Ceci implique que si $y$ est un point de $U$ d'image $x$ dans $V$, l'anneau $\widehat{B}_y$ est isomorphe \`a un sous-$\widehat{A}_x$-module d'un $\widehat{A}_x$-module libre de type fini (on utilise le fait que l'anneau $\widehat{A}_x$ est complet donc hens\'elien). Si $\mathfrak{p}$ est un id\'eal premier associ\'e de $\widehat{B}_y$, il existe par d\'efinition un plongement de $\widehat{B}_y/\mathfrak{p}$ dans $\widehat{B}_y$. Comme $\widehat{A}_x$ est int\`egre, on en d\'eduit que le $\widehat{A}_x$-module $\widehat{B}_y/\mathfrak{p}$ est sans $\widehat{A}_x$-torsion, c'est donc un $\widehat{A}_x$-module de dimension $\dim(\widehat{A}_x)$. La Proposition 16.1.9. de \cite{EGAIV1} montre alors que $\widehat{B}_y/\mathfrak{p}$ est un $\widehat{B}_y$-module de dimension $\dim(\widehat{A}_x)$. En particulier, $\widehat{B}_y$ est \'equidimensionnel et sans composante immerg\'ee, de dimension
\begin{align*}
\begin{aligned}
\dim(\widehat{A}_x)&=\dim(\mathcal{W}_\infty)\\
&=q+n[F^+:\Q]\\
&=g+n^2|S|+[F^+:\Q]\frac{n(n-1)}{2}+n[F^+:\Q]\\
&=g+n^2|S|+[F^+:\Q]\frac{n(n+1)}{2}.
\end{aligned}
\end{align*}
Par ailleurs, l'application $\omega_X$ de $U$ dans $V$ est surjective (voir la preuve de la Proposition \ref{recouvrement}). Comme $f$ est par ailleurs fini, l'image d'une composante irr\'eductible de $X_p(\rhobar)$ par $\omega_X$ est une composante irr\'eductible de $\mathcal{W}_\infty$. On en d\'eduit que  l'image d'une telle composante par $f$ est n\'ecessairement de dimension sup\'erieure ou \'egale \`a $\dim(\mathcal{W}_\infty)=\dim(Z_z(\rhobar))$, il s'agit donc d'une composante irr\'eductible de $Z_z(\rhobar)$.
\end{proof}

\begin{coro}\label{ouvert}
L'image d'une composante irr\'eductible de $X_p(\rhobar)$ par le morphisme $\omega_X$ est un ouvert de Zariski de $\mathcal{W}_\infty$.
\end{coro}

\begin{proof}
Le corollaire \ref{equidim} montre qu'il suffit de prouver que l'image d'une composante irr\'eductible de $Z_z(\rhobar)$ par $g$ est un ouvert Zariski de $\mathcal{W}_\infty$. C'est alors une cons\'equence de la preuve de \cite[Cor. 6.4.4]{CheHecke} puisque $Z_z(\rhobar)$ est une hypersurface de Fredholm par le Lemme \ref{Fredholm}, et que d'apr\`es \cite[Thm. 4.2.2]{ConradIrred}, une composante irr\'eductible d'une hypersurface de Fredholm est une hypersurface de Fredholm.
\end{proof}

\subsection{S\'erie principale localement $\Q_p$-analytique de $G_p$}

Le groupe $G_p$ est le groupe des $\Q_p$-points du groupe $\Q_p$-alg\'ebrique $\prod_{v\in S_p}\Res_{F_v^+/\Q_p}\GL_{n,F_v^+}$, qui est un groupe r\'eductif quasi-d\'eploy\'e. Le sous-groupe $B_p$ est le groupe des $\Q_p$-points du $\Q_p$-sous-groupe de Borel $\prod_{v\in S_p}\Res_{F_v^+/\Q_p}B_{F_v^+}$. Dans la suite, on appelle sous-groupe parabolique standard le groupe des $\Q_p$-points d'un sous-groupe parabolique de $\prod_{v\in S_p}\Res_{F_v^+/\Q_p}\GL_{n,F_v^+}$ contenant $\prod_{v\in S_p}\Res_{F_v^+/\Q_p}B_{F_v^+}$ (un tel groupe \'etant automatiquement d\'efini sur $\Q_p$). Si $P$ est un sous-groupe parabolique standard, on note $\overline{P}$ le groupe des $\Q_p$-points du sous-groupe parabolique oppos\'e.

On note $\mathfrak{g}$, $\mathfrak{b}$, $\overline{\mathfrak{b}}$, $\mathfrak{p}$, $\overline{\mathfrak{p}}$ les $\Q_p$-alg\`ebres de Lie des groupes $G_p$, $B_p$, $\overline{B}_p$, $P_p$ et $\overline{P}_p$, lorsque $P_p$ est un sous-groupe parabolique de $G_p$ contenant $B_p$ et $\overline{P}_p$ le parabolique oppos\'e (relativement \`a $T_p$). Si $P_p$ est un tel sous-groupe parabolique, on note $L_P$ son sous-groupe de Levi contenant $T_p$. Si $\mathfrak{h}$ est une $\Q_p$-alg\`ebre de Lie, on note $\mathfrak{h}_L=\mathfrak{h}\otimes_{\Q_p}L$. Comme $L$ est choisie de sorte que $\dim_L\Hom(F^+\otimes_{\Q}\Q_p,L)=[F^+:\Q]$, on a une d\'ecomposition $\mathfrak{g}_L\simeq\bigoplus_{\tau\in\Hom(F^+,L)}\mathfrak{g}_\tau$, o\`u $\mathfrak{g}_\tau=\mathfrak{g}\otimes_{F^+,\tau}L$. Le groupe de Weyl de $\mathfrak{g}_L$ se d\'ecompose alors canoniquement en un produit $\prod_{\tau\in\Hom(F^+,L)}W_\tau$, o\`u $W_\tau$ est le groupe de Weyl de $\mathfrak{g}_\tau$.

 Les r\'esultats et preuves des \S2 et \S3 de \cite{BreuilAnalytiqueI} restent valables \emph{mutatis mutandis} dans ce contexte. On peut d\'efinir la cat\'egorie $\mathcal{O}_{\rm alg}^{\overline{\mathfrak{p}}}$ de \cite{HumBGG} comme dans \cite[\S2]{BreuilAnalytiqueI} et d\'efinir suivant Orlik et Strauch (\cite{OSJH}), pour toute repr\'esentation lisse de longueur finie $\pi_P$ de $L_P$, et tout objet $M$ de la cat\'egorie $\mathcal{O}_{\rm alg}^{\overline{\mathfrak{p}}}$, une repr\'esentation localement analytique admissible de $G_p$ not\'ee $\mathcal{F}_{\overline{P}_p}^{G_p}(M,\pi_P)$.
 
De m\^eme, si $P$ d\'esigne un sous-groupe parabolique standard de $\prod_{v\in S_p}\Res_{F_v^+/\Q_p}\GL_{n,F_v^+}$ et $M$ un objet de $\mathcal{O}_{\rm alg}^{\overline{\mathfrak{p}}}$, on dit que $P$ est maximal pour $M$, si $M$ n'appartient pas \`a la sous-cat\'egorie pleine $\mathcal{O}_{\rm alg}^{\overline{\mathfrak{q}}}$ pour tout sous-groupe parabolique standard $Q$ contenant strictement $P$.

On a alors le r\'esultat suivant.

\begin{theo}[\cite{OSJH}, \cite{BreuilAnalytiqueI}]\label{OSJH}
(i) La construction $(M,\pi_P)\mapsto\mathcal{F}_{\overline{P}_p}^{G_p}(M,\pi_P)$ est fonctorielle et exacte en les deux arguments.\\
(ii) Si $M$ est un $U(\mathfrak{g}_L)$-module simple, $P$ est maximal pour $M$ et $\pi_P$ est irr\'eductible, la repr\'esentation $\mathcal{F}_{\overline{P}_p}^{G_p}(M,\pi_P)$ est irr\'eductible.
\end{theo}

Rappelons que $(-)^\vee$ d\'esigne le foncteur de dualit\'e sur la cat\'egorie $\mathcal{O}_{\rm alg}^{\overline{\mathfrak{b}}}$ (cf.\cite[\S3.2]{HumBGG} et \cite[\S9.3]{HumBGG})

Si $H\subset G_p$ est un sous-groupe et $\pi$ une repr\'esentation de $P_p$ sur un $L$-espace vectoriel, on note $\Ind_H^{G_p}(\pi)$ l'induite lisse non normalis\'ee de $\pi$, c'est-\`a-dire le $L$-espace vectoriel des fonctions localement constantes $f:\,G_p\rightarrow \pi$ telles que $f(hg)=h\cdot f(g)$ pour tout $h\in H$ et $g\in G_p$. Il est muni de l'action lisse de $G_p$ par translation \`a droite. On note $\cind_H^{G_p}(\pi)$ le sous-espace stable par $G_p$ des fonctions \`a support compact.

\subsection{Densit\'e des points classiques}

Rappelons que, pour $K$ extension finie de $\Q_p$, on a d\'efini \`a la fin du \S\ref{triangulines} un automorphisme $\iota_K$ de l'espace des caract\`eres continus du groupe $T(K)$. On d\'esigne par $\iota$ la version globale de cet automorphisme, c'est-\`a-dire que l'on pose $\iota=(\iota_{F^+_v})_{v\in S_p}$, automorphisme de $\widehat{T}_{p,L}$. 

Soit $x=(y,\delta)\in X_p(\rhobar)(L)\subset(\Spf(R_\infty)^{\rm rig}\times\widehat{T}_{p,L})(L)$. Notons $\mathfrak{p}_y$ l'id\'eal maximal de $R_\infty[\frac{1}{p}]$ correspondant au point $y\in\Spf(R_\infty)^{\rm rig}(L)=\Hom(R_\infty[\frac{1}{p}],L)$ et posons, pour le moment, $\Pi=\Pi_\infty[\mathfrak{p}_y]$. De la caract\'erisation des points de $X_p(\rhobar)$ donn\'ee par la Proposition \ref{points}, on d\'eduit que
\begin{equation*}
\Hom_{T_p}({\rm \delta},J_{B_p}(\Pi^{\rm an})))\neq0.
\end{equation*}
 Supposons que $\delta$ soit un caract\`ere localement alg\'ebrique de $T_p$ de poids dominant $\lambda\in\prod_{v\in S_p}(\mathbb{Z}^n)^{[F_v^+:\Q_p]}$. Par d\'efinition, la repr\'esentation $\Pi^{\rm an}$ est tr\`es fortement admissible (cf.~\cite[Def.~0.12]{EmertonJacquetII}), on peut donc appliquer le Th\'eor\`eme 4.3 de \cite{BreuilAnalytiqueII} qui donne un isomorphisme
\begin{equation}\label{recfrob}
\Hom_{G_p}\left(\mathcal{F}_{\overline{B}_p}^{G_p}((U(\mathfrak{g}_L)\otimes_{U(\bar{\mathfrak{b}}_L)}(-\lambda))^\vee,\delta_{B_p}^{-1}\delta\delta_{\lambda}^{-1}),\Pi^{\rm an}\right)\simeq\Hom_{T_p}(\delta,J_{B_p}(\Pi^{\rm an})).
\end{equation}
L'apparition du caract\`ere module $\delta_{B_P}^{-1}$ par rapport \`a \emph{loc.~cit.} vient du fait que nous avons choisi la param\'etrisation d'Emerton pour le foncteur $J_{B_p}$.

Notons $L(\lambda)$ le $U(\mathfrak{g}_L)$-module simple de plus haut poids $\lambda$ relativement \`a $\mathfrak{b}_L$. Comme $\lambda$ est dominant, il s'agit d'un $L$-espace vectoriel de dimension finie. De plus, l'action de $\mathfrak{g}$ sur $L(\lambda)$ s'\'etend de fa\c{c}on unique en une repr\'esentation alg\'ebrique du groupe $G_p$. Par ailleurs, il existe une injection de $U(\mathfrak{g}_L)$-modules $L(\lambda)'\hookrightarrow(U(\mathfrak{g}_L)\otimes_{U(\bar{\mathfrak{b}}_L)}(-\lambda))^\vee$. Par le (i) du Th\'eor\`eme \ref{OSJH}, cette injection induit une surjection continue entre repr\'esentations localement analytiques admissibles de $G_p$
\begin{align*}
\begin{aligned}
\mathcal{F}_{\overline{B}_p}^{G_p}((U(\mathfrak{g}_L)\otimes_{U(\bar{\mathfrak{b}}_L)}(-\lambda))^\vee,\delta_{B_p}^{-1}\delta\delta_{\lambda}^{-1})&\twoheadrightarrow\mathcal{F}_{G_p}^{G_p}(L(\lambda)',\Ind_{\overline{B}_p}^{G_p}(\delta_{B_p}^{-1}\delta\delta_{\lambda}^{-1})).
\end{aligned}
\end{align*}

Lorsque la repr\'esentation $\Ind_{\overline{B}_p}^{G_p}(\delta_{B_p}^{-1}\delta\delta_{\lambda}^{-1})$ est irr\'eductible, elle est isomorphe \`a $\Ind_{B_p}^{G_p}(\delta\delta_{\lambda}^{-1})$ en utilisant les isomorphismes \[\Ind_{\overline{B}_p}^{G_p}(\chi)\simeq\Ind_{B_p}^{G_p}(\chi^{w_0})\ \text{et}\ \Ind_{B_p}^{G_p}(\chi)\simeq\Ind_{B_p}^{G_p}(\chi^{w_0}\delta_{B_p})\] entre induites paraboliques lisses, $w_0$ d\'esignant l'\'el\'ement du groupe de Weyl de $G_p$ de longueur maximale (relativement \`a l'ensemble de racines simples d\'efini par le choix de $B_p$).

Au final on obtient une injection
\begin{equation}\label{classicite}
\Hom_{G_p}\left(L(\lambda)\otimes_L\Ind_{\overline{B}_p}^{G_p}(\delta_{B_p}^{-1}\delta\delta_{\lambda}^{-1}),\Pi\right)\hookrightarrow\Hom_{T_p}\left(\delta,J_{B_p}(\Pi^{\rm an})\right).
\end{equation}

\begin{defi}\label{defclassique}
Un point $x=(y,\delta)\in(\Spf(R_\infty)^{\rm rig}\times\widehat{T}_{p,L})(L)$ est dit \emph{classique} de poids $\lambda$ si le membre de gauche de \eqref{classicite} est non nul.
\end{defi}

Si $x=(y,\delta)\in(\Spf(R_\infty)^{\rm rig}\times\widehat{T}_{p,L})(L)$, et $v\in S_p$, on note $y_v$ la composante en $v$ de l'image de $y$ dans $\mathfrak{X}^\square(\rhobar_p)\simeq\prod_{v\in S_p}\mathfrak{X}^\square(\rhobar_{\tilde{v}})$, on note alors $r_{y,v}$ la classe d'isomorphisme de repr\'esentation de $\mathcal{G}_{F_{{\tilde{v}}}}$ correspondant \`a $y_v$. De m\^eme on note $\delta_v$ la composante en $v$ de $\delta$ via la d\'ecomposition $\widehat{T}_p\simeq\prod_{v\in S_p}\widehat{T}_v$.

\begin{prop}\label{tresclassique}
Soit $(y,\iota(\delta))\in X_p(\rhobar)(L)$ un point classique de poids $\lambda$. Supposons que pour tout $v\in S_p$, la repr\'esentation $\Ind_{B_v}^{G_v}(\delta_v\delta_{\lambda_v}^{-1})$ soit absolument irr\'eductible. Alors la repr\'esentation $r_{y,v}$ est cristalline de poids de Hodge-Tate \[(k_{\tau,i}=\lambda_{\tau,i}-(i-1))_{\tau\in\Hom(F^+,L),1\leq i\leq n}.\]

De plus, l'ensemble des valeurs propres du Frobenius cristallin de $r_{y,v}$, est
\begin{equation*}
\big\{\varphi_{v,i}=\delta_{v,i}(\varpi_v)\cdot\prod_{\tau\in\Hom(F_v ^+,L)}\delta_{k_{\tau,i}}(\varpi_v)^{-1}\big|\, 1\leq i\leq n\big\}.
\end{equation*}
Si de plus le raffinement $(\varphi_{v,i})_{1\leq i\leq n}$ est non critique, alors $(y_v,\delta_v)\in X_{\rm tri}^\square(\rhobar_{\tilde{v}})$.
\end{prop}

\begin{proof}
Si $x=(y,\iota(\delta))$ est un point classique de poids dominant $\lambda$, le caract\`ere $\chi=\iota(\delta)\delta_\lambda^{-1}$ est lisse et non ramifi\'e. De plus, $\chi$ se d\'ecompose naturellement en un produit tensoriel $\bigotimes_{v\in S_p}\chi_v$, o\`u $\chi_v$ est un caract\`ere lisse et non ramifi\'e de $T_v$. Comme la repr\'esentation $\Ind_{B_v}^{G_v}(\chi_v)$ est absolument irr\'eductible, on a $\chi_{v,i}\chi_{v,j}^{-1}\notin\{|\cdot|_{F_v}^{j-i+1}\}$ pour $i\neq j$ (voir par exemple \cite{KudlaLocLan}). Soit $\mathcal{H}(\lambda)$ l'alg\`ebre de Hecke $\End_{G_p}(\cind_{K_p}^{G_p}L(\lambda))$. Il existe un unique caract\`ere $\theta_\chi:\mathcal{H}(\lambda)\rightarrow L$ tel que
\begin{equation}\label{Satake}
L(\lambda)\otimes_L\Ind_{B_p}^{G_p}(\chi)\simeq\cind_{K_p}^{G_p}(L(\lambda))\otimes_{\mathcal{H}(\lambda)}L(\theta_\chi).
\end{equation}
L'existence d'un tel caract\`ere r\'esulte par exemple de l'isomorphisme du bas de la page 670 de \cite{STBanach} et de la remarque suivant la Conjecture page 673 de \emph{loc.~cit.}.

Par r\'eciprocit\'e de Frobenius, on a un isomorphisme
\begin{equation*}
\Hom_{K_p}(L(\lambda),\Pi_\infty)\simeq\Hom_{G_p}(\cind_{K_p}^{G_p}L(\lambda),\Pi_\infty),
\end{equation*}
qui fournit ainsi une action de l'alg\`ebre commutative $\mathcal{H}(\lambda)$ sur l'espace $\Hom_{K_p}(L(\lambda),\Pi_\infty)$, que l'on note $\Pi_{\infty}(\lambda)$. Comme le point $x$ est classique, il existe un morphisme non nul de $L(\lambda)\otimes_L\Ind_{P_p}^{G_p}(\chi)$ dans $\Pi_\infty[\mathfrak{p}_y]$, l'isomorphisme \eqref{Satake} implique alors qu'il existe un sous-espace non nul de $\Pi_\infty(\lambda)[\mathfrak{p}_y]$ sur lequel $\mathcal{H}(\lambda)$ agit via $\theta_\chi$. Comme le point $x$ est classique de poids $\lambda$, on a $\Pi_\infty(\lambda)[\mathfrak{p}_y]\neq0$. Le Lemme 4.16 de \cite{CEGGPS}, appliqu\'e \`a a repr\'esentation $\sigma=L(\lambda)$, implique alors que pour $v\in S_p$, la repr\'esentation $r_{y,v}$ de $\mathcal{G}_{F_{\tilde{v}}}$ est cristalline de poids ${\bf k}_v=(k_{\tau,i})_{\tau\in\Hom(F^+_v,L),1\leq i\leq n}$ o\`u $k_{\tau,i}=\lambda_{\tau,i}-(i-1)$ et que le Frobenius cristallin lin\'earis\'e de $r_{y,v}$ a pour valeurs propres $(\varphi_{i,v})$ o\`u $\varphi_{i,v}=\chi_i(\varpi_v)|\varpi_v|_v^{i-n}$. Plus pr\'ecis\'ement les d\'ecompositions $G_p=\prod_{v\in S_p}G_v$ et $K_p=\prod_{v\in S_p}K_v$ permettent de d\'ecomposer $\mathcal{H}(\lambda)$ en produit tensoriel $\bigotimes_{v\in S_p}\mathcal{H}(\lambda_v)$ o\`u $\mathcal{H}(\lambda_v)$ d\'esigne l'alg\`ebre des endomorphismes de $\cind_{K_v}^{G_v}L(\lambda_v)$. Le Th\'eor\`eme 4.1 de \cite{CEGGPS} permet alors de construire un unique morphisme de $L$-alg\`ebres $\eta_v:\,\mathcal{H}(\lambda_v)\rightarrow R_{\rhobar_{\tilde{v}}}^{\square,k_v-{\rm cr}}[\frac{1}{p}]$ qui interpole la correspondance de Langlands comme normalis\'ee dans le \S\ref{triangulines}. On peut alors, en raisonnant exactement comme dans la preuve de \cite[Lem.~4.16]{CEGGPS}, prouver que l'action de $R^\square_{\rhobar_{\tilde{v}}}$ sur $\Pi_\infty(\lambda)$ se factorise par $R_{\rhobar_{\tilde{v}}}^{\square,k_v-{\rm cr}}$ et que l'action de $\mathcal{H}(\lambda_v)$ se d\'eduit alors de cette action de $R_{\rhobar_{\tilde{v}}}^{\square,k_v-{\rm cr}}[\frac{1}{p}]$ via $\eta_v$.

Tout ceci implique au final que $r_{y,v}$ est trianguline et a un param\`etre \'egal \`a
\begin{equation*}
\delta_{\lambda_v}\chi_v\cdot(1,\varepsilon^{-1}\circ\rec_{F_v^+},\dots,\varepsilon^{1-n}\circ\rec_{F_v^+})\cdot\delta_{B_v}^{-1}=\delta_v.
\end{equation*}
La derni\`ere assertion est imm\'ediate car, si le raffinement est non critique, le caract\`ere $\delta_v$ est un param\`etre de $r_{y,v}$ d'apr\`es \cite[Prop.~2.4.7]{BelChe}.
\end{proof}

Il est commode d'introduire la notion suivante.

\begin{defi}
Un point $x=(y,\delta)\in X_p(\rhobar)(L)\subset(\Spf(R_\infty)^{\rm rig}\times\widehat{T}_{p,L})(L)$ est dit \emph{tr\`es classique} de poids $\lambda$ si la restriction de $\delta$ \`a $T_p^0$ est un caract\`ere alg\'ebrique de poids entier dominant $\lambda$ et si \[L(\lambda)\otimes_L\Ind_{\overline{B}_p}^{G_p}(\delta\delta_{B_p}^{-1}\delta_\lambda^{-1})\] est le seul sous-quotient irr\'eductible de $\Ind_{\overline{B}_p}^{G_p}(\delta\delta_{B_p}^{-1})^{\rm an}$ sur lequel existe une norme $G_p$-invariante.
\end{defi}

\begin{rema}
La repr\'esentation $\Ind_{\overline{B}_p}^{G_p}(\delta\delta_{B_p}^{-1})^{\rm an}$ a les m\^emes sous-quotients irr\'eductibles que \[\mathcal{F}_{\overline{B}_p}^{G_p}((U(\mathfrak{g}_L)\otimes_{U(\bar{\mathfrak{b}}_p)_L}(-\lambda))^\vee,\delta\delta_\lambda^{-1}\delta_{B_p}^{-1}).\] Il r\'esulte de cette remarque et de \eqref{recfrob} qu'un point tr\`es classique est en particulier classique. L'assertion inverse \'etant fausse !

\end{rema}

\begin{theo}\label{densite}
Les points classiques forment une partie Zariski-dense et d'accumulation de $X_p(\rhobar)$. Plus pr\'ecis\'ement, pour tout point $x$ de $X_p(\rhobar)$ de poids alg\'ebrique et $X$ une composante irr\'eductible de $X_p(\rhobar)$ contenant $x$, il existe un voisinage affino\"ide $U$ de $x$ dans $X$ et une partie $W$ de $\widehat{T}_{p,L}^0$ tels que $\omega_X^{-1}(\Spf(S_\infty)^{\rm rig}\times W)\cap U$ soit Zariski-dense dans $U$ et constitu\'e de points tr\`es classiques, \`a param\`etre r\'egulier non critique.
\end{theo}

\begin{proof}
Soit $x\in X_p(\rhobar)$ de poids alg\'ebrique. La Proposition \ref{recouvrement} nous permet de consid\'erer $U$ un voisinage affino\"ide et connexe de $x$ dans la composante $X$ fix\'ee, tel que $\omega_X(U)$ est un ouvert affino\"ide et la restriction de $\omega_X$ \`a $U$ est une application finie de $U$ sur $\omega_X(U)$. Comme $U$ est affino\"ide, les $n|S_p|$ fonctions analytiques rigides $(y,\delta)\mapsto (\delta\delta_{B_p}^{-1})(\gamma_{\tilde{v},i})$ sont born\'ees sur $U$ (rappelons que l'on a d\'efini $\gamma_{\tilde{v},i}$ comme l'image d'une uniformisante $\varpi_v$ par le cocaract\`ere $\beta_i$ de $\mathbb{G}_{m,F_v^+}$ dans $\GL_{n,F_{\tilde{v}}}\simeq G_{F_v^+}$). Fixons donc $C>0$ tel que
\begin{equation}\label{condition}
-C\leq v_{F_v^+}((\delta_v\delta_{B_v}^{-1})(\gamma_{\tilde{v},i}))\leq C
\end{equation}
pour tout $v\in S_p$, $1\leq i\leq n$ et $(y,\delta)\in U$. Soit $Z$ l'ensemble des caract\`eres alg\'ebriques dominants ${\bf \delta}\in\widehat{T}_{p,L}^0$ de poids $\lambda$ tels que pour tout $\tau:\, F^+\hookrightarrow \Qbar_p$,
\begin{align}\label{genericite}
\begin{aligned}
&\lambda_{1,\tau}-\lambda_{2,\tau}>2(C+1),\\
&\lambda_{i,\tau}-\lambda_{i+1,\tau}> \lambda_{i-1,\tau}-\lambda_{i,\tau}+(C+1),\, {\rm pour}\,i\geq2.
\end{aligned}
\end{align}
Nous allons prouver que tous les points de $\omega_X^{-1}(\Spf(S_\infty)^{\rm rig}\times Z)$ sont classiques et non critiques.

En effet, l'ensemble $Z$ s'accumule en tout caract\`ere alg\'ebrique de $\widehat{T}_{p,L}^0$, on en conclut que si $V$ est un ouvert irr\'eductible de $\widehat{T}_{p,L}^0$ contenant un poids alg\'ebrique, alors $V\cap Z$ est Zariski-dense dans $V$, ce qui implique que $\omega_X(U)\cap(\Spf(S_\infty)^{\rm rig}\times Z)$ est Zariski-dense dans $\omega_X(U)$. Par finitude et surjectivit\'e de $\omega_X|_U$, on en d\'eduit que $\omega_X^{-1}(\Spf(S_\infty)^{\rm rig}\times Z)\cap U$ est Zariski-dense dans l'ouvert irr\'eductible $U$. En effet, si ce n'\'etait pas le cas, par connexit\'e de $U$, le plus petit ferm\'e analytique contenant $\omega_X^{-1}(\Spf(S_\infty)^{\rm rig}\times Z)\cap U$ serait de dimension $<\dim(U)$, donc son image par $\omega_X$ \'egalement. Or, par finitude de $\omega_X:\,U\rightarrow\omega_X(U)$, cette image est un ferm\'e analytique contenant $(\Spf(S_\infty)^{\rm rig}\times Z)\cap\omega_X(U)$, ce qui contredit la Zariski-densit\'e de cet ensemble dans $\omega_X(U)$ et le fait que $\dim(U)=\dim(\omega_X(U))$.

Soit donc $x\in\omega_X^{-1}(\Spf(S_\infty)^{\rm rig}\times Z)\cap U$. Reprenons les notations de la discussion suivant la D\'efinition \ref{defclassique} et de la preuve de la Proposition \ref{tresclassique}.

Nous affirmons que pour tout sous-groupe de Levi standard $L_I$ de $\GL_n$, et pour tout $v\in S_p$, l'induite parabolique lisse $\Ind_{\overline{B}_v\cap L_{I,v}}^{L_{I,v}}(\chi_v\delta_{B_v}^{-1})$ est irr\'eductible. En effet, d'apr\`es \cite[Thm.~4.2]{ZelevinskyInduites}, il suffit pour cela de v\'erifier que pour tout $i < j$ on a \[v_{F^+_v}(\chi_{v,j}(\varpi_v)\chi_{v,i}(\varpi_v)^{-1}|\varpi_v|_v^{j-i})>1,\] c'est-\`a-dire \[v_{F^+_v}(\delta_{v,j}(\varpi_v)\delta_{v,i}(\varpi_v)^{-1} |\varpi_v|_v^{j-i})>1+\sum_{\tau\in\Hom(F^+_v,L)}(\lambda_{\tau,j}-\lambda_{\tau,i}),\] ce qui est une cons\'equence de la condition \eqref{genericite}. Ainsi, pour tout sous-groupe parabolique standard $P$ de $G_p$, si $L_P$ d\'esigne l'unique sous-groupe de Levi de $P$ contenant $T_p$, la repr\'esentation $\Ind_{L_P\cap \overline{B}_p}^{L_P}(\chi\delta_{B_p}^{-1})$, isomorphe \`a $\bigotimes_{v\in S_p}\Ind_{L_{P_v}\cap\overline{B}_v}^{\GL_n(F_v^+)}(\chi_v\delta_{B_v}^{-1})$, est irr\'eductible.

Pour prouver que $x$ est tr\`es classique, il suffit de prouver que $L(\lambda)\otimes_L\Ind_{B_p}^{G_p}(\chi)$ est le seul sous-quotient irr\'eductible de $\Ind_{\overline{B}_p}^{G_p}(\delta\delta_{B_p}^{-1})^{\rm an}$ admettant une norme invariante.

Il r\'esulte du Th\'eor\`eme \ref{OSJH} que les sous-quotients irr\'eductibles de la s\'erie principale localement analytique $\Ind_{\overline{B}_p}^{G_p}(\delta\delta_{B_p}^{-1})^{\rm an}$ sont exactement les $\mathcal{F}_{\overline{P}}^{G_p}(M,\pi)$, o\`u $M$ est un sous-quotient irr\'eductible de $U(\mathfrak{g}_L)\otimes_{U(\bar{\mathfrak{b}}_L)}(-\lambda)$, $\overline{P}$ le sous-groupe parabolique maximal pour $M$ et $\pi$ la repr\'esentation lisse irr\'eductible $\Ind_{\overline{B}_p\cap L_P}^{L_P}(\chi\delta_{\overline{B}_p}^{-1})$. Si $M$ est un tel sous-quotient, alors il existe $w$ dans le groupe de Weyl $W$ de $\mathfrak{g}_L$ tel que $M$ est isomorphe \`a $M_{-w\cdot\lambda}$, le $U(\mathfrak{g}_L)$-module simple de plus haut poids $-w\cdot\lambda$ (relativement au sous-groupe de Borel $\overline{\mathfrak{b}}_L$). Nous allons prouver que si $\mathcal{F}_{\overline{P}}^{G_p}(M_{-w\cdot\lambda},\pi)$ est muni d'une norme stable par $G_p$, alors $w=1$. Si $P$ est un sous-groupe parabolique de $G_p$ contenant $B_p$, notons $L_P$ le facteur de Levi de $P$ qui contient $T$ et $N_P$ le radical unipotent de $P$. De plus on note $Z_P$ le centralisateur de $L_P$ dans $G_p$. On note $Z_P^+$, comme en \cite[(12)]{BreuilAnalytiqueI}, le sous-mono\"ide de $Z_P$ constitu\'e des $g\in Z_P$ tels que $g(N_P\cap K_p)g^{-1}\subset N_P\cap K_p$. Alors, d'apr\`es le corollaire 3.5 de \cite{BreuilAnalytiqueI}, si $\mathcal{F}_{\overline{P}_p}^{G_p}(M_{-w\cdot\lambda},\pi_P)$ admet une norme $G_p$-invariante, on doit avoir
\begin{equation}\label{conditionNorme}
\delta_{w\cdot\lambda}(z)\chi(z)\delta_{B_p}^{-1}(z)\in\mathcal{O}_L
\end{equation}
pour tout $z\in Z_{P}^+$. D\'ecomposons $P_p$ en un produit $P=\prod_{v\in S_p}P_v$ de sous-groupes paraboliques selon la d\'ecomposition $G_p=\prod_{v\in S_p}\GL_n(F_v^+)$. Pour tout $v\in S_p$, le groupe $P_v$ est le groupe des $F_v^+$-points d'un sous-groupe parabolique standard $P_{I_v}$ de $\GL_n$, pour $I_v\subset\Delta$. Si, pour tout $v\in S_p$, on choisit $1\leq i_v\leq n-1$ tel que $i_v\notin I_v$, et que l'on pose $\gamma_{v,i_v}=\beta_{i_v}(\varpi_v)$, l'\'el\'ement $(\gamma_{v,i_v})_{v\in S_p}$ de $G_p$ est en fait dans $Z_P^+$.

D'apr\`es \cite[pages 186-187]{HumBGG}, pour $v\in S_p$, l'ensemble $I_v$ est la plus grande partie de $\Delta$ telle que $w_\tau\cdot\lambda_\tau$ est un poids dominant du groupe alg\'ebrique $L_{P_v}\times_{F_v^+}L$ (relativement au Borel $(B_v\cap L_{P_v})\times_{F_v^+} L$) pour $\tau\in\Hom(F_v^+,L)$. Si $w_v\neq1$, on a $I_v\neq\{1,\dots,n-1\}$ et, d'apr\`es la proposition \ref{racines} appliqu\'ee \`a $\lambda+\rho$, il est possible de trouver une racine simple $\alpha_{v,i_v}\notin I_v$ telle que
\begin{equation*}
\langle w\cdot\lambda-\lambda,\beta_{v,i_v}\rangle\leq-\min_{\tau,i}(\lambda_{\tau,i}-\lambda_{\tau,i+1}).
\end{equation*}

Mais alors, puisque \[\delta_\lambda(\gamma_{v,i_v})=\varpi_v^{\big\langle \sum_{\tau\in\Hom(F_v^+,L)}\lambda_\tau,\beta_{i_v}\big\rangle},\] on a avec \eqref{condition} et \eqref{genericite} :
\begin{align*}
|\delta_{w\cdot \lambda}(\gamma_{v,i_v})\chi(\gamma_{v,i_v})\delta_{B_v}(\gamma_{v,i_v})^{-1}|_L&=|\delta_{w\cdot \lambda-\lambda}(\gamma_{v,i_v})\delta_v(\gamma_{v,i_v})\delta_{B_v}(\gamma_{v,i_v})^{-1}|_L\\
&\geq |\varpi_v|_L^{C-\min_{i,\tau}(\lambda_{\tau,i}-\lambda_{\tau,i+1})}>1.
\end{align*}
Ainsi, en posant $z=(\gamma_{v,i_v})_{v\in S_p}$, on a bien $z\in Z_P^+$, tel que
\begin{equation*}
|\delta_{w\cdot\lambda-\lambda}(z)\chi(z)\delta_{B_p}^{-1}(z)|_L>1,
\end{equation*}
ce qui implique par \eqref{conditionNorme} que $\mathcal{F}_{\overline{P}_p}^{G_p}(M_{-w\cdot\lambda},\pi_P)$ ne peut avoir de norme $G_p$-invariante \`a moins que $w=1$. Ceci prouve que le point $x$ est tr\`es classique. Par ailleurs, les in\'egalit\'es \eqref{genericite} impliquent \'egalement que le raffinement de $r_{y,v}$ d\'eduit de $\delta_v$ est non critique (on peut par exemple comparer ces in\'egalit\'es aux conditions de \cite[Lem.~2.9]{HellmSchrDensity}). Enfin, le param\`etre $\iota^{-1}(\delta)$ est r\'egulier, ou de fa\c{c}on \'equivalente on a $\chi_{i,v}\neq\chi_{j,v}|\cdot|_v^{j-i\pm1}$ pour tous $i \neq j$ et $v\in S_p$, ce qui est une cons\'equence du fait que les repr\'esentations $\Ind_{L_P\cap\overline{B}_p}^{G_p}(\chi)$ sont toutes irr\'eductibles.

Nous avons prouv\'e que l'ensemble des points classiques non critiques est un ensemble d'accumulation dans $X_p(\rhobar)$. Pour conclure qu'il est dense au sens de Zariski dans $X_p(\rhobar)$, il suffit de montrer que toute composante irr\'eductible de $X_p(\rhobar)$ contient un point de poids alg\'ebrique. Or c'est une cons\'equence du Corollaire \ref{ouvert} puisque les poids alg\'ebriques forment une partie dense au sens de Zariski dans $\mathcal{W}_\infty$.
\end{proof}

\begin{coro}\label{reduit}
L'espace analytique $X_p(\rhobar)$ est r\'eduit.
\end{coro}
\begin{proof}
Il suffit de prouver que chaque composante irr\'eductible de $X_p(\rhobar)$ contient un point $x$ dont l'anneau local est r\'eduit. En effet, soit $\mathcal{I}$ le nilradical de $X_p(\rhobar)$ (cf.~\cite[\S9.5.1]{BGR}). Il s'agit d'un faisceau coh\'erent d'id\'eaux sur $X_p(\rhobar)$. Comme d'apr\`es le Corollaire \ref{equidim}, l'espace $X_p(\rhobar)$ est sans composante immerg\'ee, le support du faisceau $\mathcal{I}$ est une union de composantes irr\'eductibles. Si chaque composante irr\'eductible contient un point r\'eduit, c'est-\`a-dire un point dont l'anneau local compl\'et\'e est r\'eduit, alors $\mathcal{I}=0$ et $X_p(\rhobar)$ est r\'eduit. D'apr\`es le Th\'eor\`eme \ref{densite}, toute composante irr\'eductible de $X_p(\rhobar)$ contient un point tr\`es classique, il suffit donc de prouver que $X_p(\rhobar)$ est r\'eduit en tout point tr\`es classique.

On raisonne alors comme dans \cite[Prop. 3.9]{CheJL}, l'espace des poids \'etant ici remplac\'e par $\mathcal{W}_\infty$. Soient $x\in X_p(\rhobar)$ un point tr\`es classique, disons de poids $\lambda$, et $z$ son image dans $\mathcal{W}_\infty$. 

Soit $U=\Sp(B)$ un voisinage affino\"ide de $x$ comme dans la Proposition \ref{recouvrement}, et $V=\Sp(A)$ un ouvert affino\"ide de $\mathcal{W}_\infty$ tel que $\omega_X$ induise une application finie surjective de $U$ sur $V$ et que $U$ soit une composante connexe de $\omega_X^{-1}(V)$. On note $M=\Gamma(U,\mathcal{M}_\infty)$. Il s'agit d'un $A$-module projectif de type fini et l'action de $B$ sur $M$ identifie $B$ \`a une sous-$A$-alg\`ebre de $\End_A(M)$.

On d\'eduit du Th\'eor\`eme \ref{densite}, quitte \`a r\'etr\'ecir $U$, l'existence d'un ensemble $Z$ de poids alg\'ebriques dominants de $\widehat{T}_{p,L}^0$ tels que $(\Spf(S_\infty)^{\rm rig}\times Z)\cap V$ soit dense au sens de Zariski dans $V$ et que tous les points de $\omega_X^{-1}(\Spf(S_\infty)^{\rm rig}\times Z)\cap U$ soient tr\`es classiques et de param\`etre r\'egulier. On peut \'egalement, quitte \`a r\'etr\'ecir $U$ et $V$, supposer que $V$ est isomorphe \`a un produit $V_1\times V_2$, o\`u $V_1$ et $V_2$ sont des ouverts affino\"{\i}des connexes de $\Spf(S_\infty)^{\rm rig}$ et de $\widehat{T}_{p,L}^0$. Nous allons prouver que pour tout $\lambda\in Z$, il existe un ouvert de Zariski $V_\lambda$, dense au sens de Zariski, de $V\cap(\Spf(S_\infty)^{\rm rig}\times\{\delta_\lambda\})=V_1\times\{\delta_\lambda\}$ tel que si $z\in V_\lambda$, le $B$-module $M\otimes_A k(z)$ est semi-simple. Comme $\bigcup_{\lambda\in Z} V_\lambda$ est alors une partie dense au sens de Zariski de $V$ (ce qui se d\'eduit de la forme de $V$), on conclut comme dans la preuve de la Proposition $4.9$ de \cite{CheJL}.

Prouvons donc l'existence d'un tel $V_\lambda$. Comme $B$ est la sous-$A$-alg\`ebre de $\End_A(M)$ engendr\'ee par $A$, l'image de $R_\infty$ et la sous-alg\`ebre engendr\'ee par $T_p^+$, il suffit de construire $V_\lambda$ tel que pour $z\in V_\lambda$, le $k(z)$-espace vectoriel $M\otimes_A k(z)$ est semi-simple \`a la fois comme $R_\infty$-module et comme $k(z)[T_p^+]$-module.

Si $z\in \Spf(S_\infty)^{\rm rig}\times\{\delta_\lambda\}$, notons $\mathfrak{p}_z\subset S_\infty[\frac{1}{p}]$ l'id\'eal maximal correspondant et $k(z)$ son corps r\'esiduel. Notons $\Sigma=\Hom_{k(z)}(M\otimes_A k(z),k(z))$. Comme $M\otimes_A k(z)$ est l'ensemble des sections de $\mathcal{M}_\infty$ sur l'ensemble fini $U\cap\omega_X^{-1}(z)$, la repr\'esentation $\Sigma$ est isomorphe \`a un facteur direct de $J_{B_p}(\Pi_\infty^{R_\infty-{\rm an}}[\mathfrak{p}_z])^{T^0_p=\delta_\lambda}$. On peut donc \'ecrire $\Sigma\simeq\delta_\lambda\otimes_L\Sigma_\infty$, o\`u $\Sigma_\infty$ est une repr\'esentation lisse non ramifi\'ee de dimension finie de $T_p$.

Le Th\'eor\`eme 4.3 de \cite{BreuilAnalytiqueII} permet alors d'associer \`a tout morphisme $\Sigma\rightarrow J_{B_p}(\Pi_\infty^{R_\infty-{\rm an}}[\mathfrak{p}_z])$ un morphisme de repr\'esentations $G_p$-analytiques
\begin{equation}\label{adjoint}
\mathcal{F}_{\overline{B}_p}^{G_p}((U(\mathfrak{g}_L)\otimes_{U(\overline{\mathfrak{b}}_L)}(-\lambda))^\vee, \Sigma_\infty(\delta_{B_p}^{-1}))\rightarrow\Pi_\infty[\mathfrak{p}_z].
\end{equation}
Comme deux caract\`eres lisses de $T_p$ ont une extension lisse non triviale si et seulement si ils sont \'egaux, on en conclut que les facteurs de Jordan-H\"older de $\Sigma_\infty$ co\"{\i}ncident avec les constituants du socle de $\Sigma_\infty$. Remarquons de plus qu'une extension lisse entre deux caract\`eres non ramifi\'es de $T_p$ est non ramifi\'ee, donc $\Sigma_\infty$ est une repr\'esentation de dimension finie du groupe $T_p/T_p^0$. Or par construction de $\lambda$, $U$ et $V$, si $\chi$ est un caract\`ere lisse du socle de $\Sigma_\infty$, le seul sous-quotient de \[\mathcal{F}_{\overline{B}_p}^{G_p}((U(\mathfrak{g}_L)\otimes_{U(\overline{\mathfrak{b}}_L)}(-\lambda))^\vee,\chi\delta_{B_p}^{-1})\] admettant une norme $G_p$-invariante est son quotient localement alg\'ebrique. On en d\'eduit donc que les seuls sous-quotients de \[\mathcal{F}_{\overline{B}_p}^{G_p}((U(\mathfrak{g}_L)\otimes_{U(\overline{\mathfrak{b}}_L)}(-\lambda))^\vee, \Sigma_\infty(\delta_{B_p}^{-1}))\] admettant une norme $G_p$-invariante se trouvent dans le quotient localement alg\'ebrique $L(\lambda)\otimes_L\Ind_{\overline{B}_p}^{G_p}(\Sigma_\infty(\delta_{B_p}^{-1}))$. Ainsi le morphisme \eqref{adjoint} se factorise par $L(\lambda)\otimes_L\Ind_{\overline{B}_p}^{G_p}(\Sigma_\infty(\delta_{B_p}^{-1}))$. On obtient au final un isomorphisme
\begin{equation*}
\Hom_{T_p}(\Sigma,J_{B_p}(\Pi_\infty^{R_\infty-{\rm an}}[\mathfrak{p}_z]))\simeq\Hom_{G_p}(L(\lambda)\otimes_L\Ind_{\overline{B}_p}^{G_p}(\Sigma_\infty(\delta_{B_p}^{-1})),\Pi_\infty[\mathfrak{p}_z]).
\end{equation*}
Comme $\Ind_{\overline{B}_p}^{G_p}(\Sigma_\infty(\delta_{B_p}^{-1}))$ est une extension de longueur finie de repr\'esentations non ramifi\'ees irr\'eductibles, il existe un $\mathcal{H}:=L[K_p\backslash G_p/K_p]$-module de dimension finie $\mathcal{M}^{\rm lc}$ et un isomorphisme $G_p$-\'equivariant entre $\Ind_{\overline{B}_p}^{G_p}(\Sigma_\infty(\delta_{B_p}^{-1}))$ et $\cind_{K_p}^{G_p}(1)\otimes_{\mathcal{H}}\mathcal{M}^{\rm lc}$. En utilisant par ailleurs \cite[Lem.~1.4]{STBanach}, on a donc au final un isomorphisme
\begin{equation*}
\Hom_{T_p}(\Sigma,J_{B_p}(\Pi_\infty^{R_\infty-{\rm an}}[\mathfrak{p}_z]))\simeq\Hom_{\mathcal{H}(\lambda)}(\mathcal{M}^{\rm lc},\Pi_\infty(\lambda)[\mathfrak{p}_z]).
\end{equation*}
Or d'apr\`es \cite[Lemma 4.16]{CEGGPS} dont la preuve s'\'etend essentiellement verbatim \`a notre situation, l'image de $R_\infty$ dans l'anneau des endomorphismes de $\Pi_\infty(\lambda)$ est un anneau r\'eduit. Rappelons que l'action de $R_\infty$ sur $\Pi_\infty(\lambda)$ se factorise \`a travers \[R_\infty(\lambda):=R_{\rhobar_p}^{\square,{\bf k}-{\rm cr}}\dbl x_1,\dots,x_g\dbr\] o\`u $\lambda=(\lambda_{\tau,i})\in\prod_{v\in S_p}(\mathbb{Z}^n)^{[F_v^+:\Q_p]}$ est li\'e \`a ${\bf k}\in \prod_{v\in S_p}(\mathbb{Z}^n)^{[F_v^+:\Q_p]}$ par la relation usuelle $\lambda_{\tau,i}=k_{\tau,i}+(i-1)$. Comme $\Pi_\infty(\lambda)'$ est un $S_\infty[\frac{1}{p}]$-module projectif de type fini par \emph{loc.~cit.}, on d\'eduit de \cite[Lem.~3.10.(b)]{CheJL} qu'il existe un ouvert $W$ de $\Spec(S_\infty[\frac{1}{p}])$ tel que pour tout point ferm\'e $z\in W$, le module $\Pi_\infty(\lambda)'\otimes k(z)$ soit un $R_\infty(\lambda)$-module semi-simple. Posons alors $V_\lambda=(W\times\{\lambda\})\cap V$. Si $z\in W$, la repr\'esentation $\Pi_\infty(\lambda)[\mathfrak{p}_z]$ est donc un $R_\infty$-module semi-simple d'apr\`es ce qui pr\'ec\`ede. De plus, l'action de $\mathcal{H}(\lambda)$ sur $\Pi_\infty(\lambda)$ se factorise par l'action de $R_\infty(\lambda)$, via le morphisme $\mathcal{H}(\lambda)\rightarrow R_{\rhobar_p}^{\square,{\bf k}-{\rm cr}}$ d\'ej\`a consid\'er\'e dans la preuve de la Proposition \ref{tresclassique}. En particulier, comme $\Pi_\infty(\lambda)[\mathfrak{p}_z]$ est un $R_\infty(\lambda)$-module semi-simple de dimension finie sur $L$, c'est \'egalement un $\mathcal{H}(\lambda)$-module semi-simple. On en d\'eduit donc que tout morphisme de $\mathcal{M}^{\rm lc}$ dans $\Pi_\infty(\lambda)[\mathfrak{p}_z]$ se factorise \`a travers le plus grand quotient semi-simple de $\mathcal{M}^{\rm lc}$. Notons $\mathcal{M}^{\rm lc}_{\rm ss}$ ce plus grand quotient. Alors $\cind_{K_p}^{G_p}(1)\otimes_{\mathcal{H}(1)}\mathcal{M}^{\rm lc}_{\rm ss}$ est le plus grand quotient semi-simple de $\Ind_{\overline{B}_p}^{G_p}(\Sigma_\infty(\delta_{B_p}^{-1}))$, il est donc de la forme $\Ind_{\overline{B}_p}^{G_p}(\Sigma_{\rm ss}(\delta_{B_p}^{-1}))$, o\`u $\Sigma_{\rm ss}$ est le plus grand quotient semi-simple de $\Sigma_\infty$ (utiliser l'exactitude du foncteur de Jacquet sur les repr\'esentations lisses). On obtient donc au final
\[\Hom_{T_p}(\Sigma,J_{B_p}(\Pi_\infty^{R_\infty-{\rm an}}[\mathfrak{p}_z]))\simeq\Hom_{T_p}(\delta_\lambda\otimes_L\Sigma_{\rm ss},J_{B_p}(\Pi_\infty^{R_\infty-{\rm an}}[\mathfrak{p}_z])).\]
Comme $\Sigma$ est une sous-repr\'esentation de $J_{B_p}(\Pi_\infty^{R_\infty-{\rm an}}[\mathfrak{p}_z])$, on en conclut que $\Sigma$ est en fait une repr\'esentation semi-simple de $T_p^+$.
\end{proof}

\subsection{Composantes irr\'eductibles de $X_p(\rhobar)$}\label{sectioncomposantes}

Notons $\iota$ l'automorphisme \[\Id_{\mathfrak{X}^\square_{\rhobar_p}}\times (\iota_{F^+_v})_{v\in S_p}:\mathfrak{X}^\square_{\rhobar_p}\times\widehat{T}_{p,L}\longrightarrow \mathfrak{X}^\square_{\rhobar_p}\times\widehat{T}_{p,L}.\]
\begin{theo}\label{egalitecomposante}
L'immersion \[X_p(\rhobar)\longrightarrow \Spf(R_\infty)^{\rm rig}\times\widehat{T}_{p,L}\simeq\mathfrak{X}^\square_{\rhobar_p}\times\mathfrak{X}^\square_{\rhobar^p}\times\mathbb{U}^g\times\widehat{T}_{p,L}\] induit un isomorphisme d'espaces analytiques rigides entre $X_p(\rhobar)$ et une union de composantes irr\'eductibles du sous-espace analytique ferm\'e $\iota(X_{\rm tri}^\square(\rhobar_p))\times\mathfrak{X}^\square_{\rhobar^p}\times\mathbb{U}^g$ muni de sa structure de sous-espace analytique ferm\'e r\'eduit.
\end{theo}
\begin{proof}
Nous allons prouver que le sous-espace analytique ferm\'e de l'espace rigide $\mathfrak{X}^\square_{\rhobar_p}\times\mathfrak{X}^\square_{\rhobar^p}\times\mathbb{U}^g\times\widehat{T}_{p,L}$ sous-jacent \`a $X_p(\rhobar)$ est en r\'ealit\'e contenu dans $\iota(X_{\rm tri}^\square(\rhobar_p))\times\mathfrak{X}_{\rhobar^p}^\square\times\mathbb{U}^g$. Comme ces deux espaces sont \'equidimensionnels de m\^eme dimension et que $X_p(\rhobar)$ est r\'eduit, on en conclut que $X_p(\rhobar)$ s'identifie bien \emph{en tant qu'espace analytique rigide} \`a une union de composantes irr\'eductibles de $\iota(X_{\rm tri}^\square(\rhobar_p))\times\mathfrak{X}^\square_{\rhobar^p}\times\mathbb{U}^g$, cette union \'etant munie de sa structure r\'eduite de sous-espace analytique ferm\'e.

Par d\'efinition de $\iota$ et des vari\'et\'es $X_p(\rhobar)$ et $X_{\rm tri}^\square(\rhobar_p)$, les points tr\`es classiques de $X_p(\rhobar)$ appartiennent aussi \`a $\iota(X_{\rm tri}^\square(\rhobar_p))\times\mathfrak{X}^\square_{\rhobar^p}\times\mathbb{U}^g$. D'apr\`es le Th\'eor\`eme \ref{densite}, ils forment une partie dense au sens de Zariski de $X_p(\rhobar)$. Comme $\iota(X_{\rm tri}^\square(\rhobar_p))\times\mathfrak{X}^\square_{\rhobar^p}\times\mathbb{U}^g$ est une partie analytique ferm\'ee de $\mathfrak{X}^\square_{\rhobar_p}\times\mathfrak{X}^\square_{\rhobar^p}\times\mathbb{U}^g\times\widehat{T}_{p,L}$, on obtient bien l'inclusion recherch\'ee.
\end{proof}

Une composante irr\'eductible de l'espace $\iota(X_{\rm tri}^\square(\rhobar_p))\times\mathfrak{X}^\square_{\rhobar^p}\times\mathbb{U}^g$ est un produit de la forme $\iota(X)\times\mathfrak{X}^p\times\mathbb{U}^g$, o\`u $X$ est une composante irr\'eductible de $X_{\rm tri}^\square(\rhobar_p)$ et $\mathfrak{X}^p$ une composante irr\'eductible de $\mathfrak{X}^\square_{\rhobar^p}$. Il nous para\^it raisonnable de conjecturer qu'\`a $\mathfrak{X}^p$ fix\'ee, les composantes irr\'eductibles $X$ de $X_p(\rhobar)$ de la forme $\iota(X)\times\mathfrak{X}^p\times\mathbb{U}^g$ ne d\'ependent pas des choix globaux que nous avons faits, c'est-\`a-dire de $G$, de $U^p$ et du syst\`eme de Taylor-Wiles intervenant dans la construction de $\Pi_\infty$.

\begin{defi}\label{compautomorphe}
Soit $\mathfrak{X}^p$ une composante irr\'eductible de $\mathfrak{X}^\square_{\rhobar^p}$. On dit qu'une composante irr\'eductible $X$ de $X_{\rm tri}^\square(\rhobar_p)$ est $\mathfrak{X}^p$-automorphe si $\iota(X)\times\mathfrak{X}^p\times\mathbb{U}^g$ est une composante irr\'eductible de $X_p(\rhobar)$.
\end{defi}

\begin{conj}\label{conjcomp}
Une composante irr\'eductible $X$ de $X_{\rm tri}^\square(\rhobar_p)$ est $\mathfrak{X}^p$-automorphe si et seulement si $X\cap U_{\rm tri}^\square(\rhobar_p)^{\rm reg}$ contient un point cristallin.
\end{conj}

\begin{rema}\label{rema}
(i) Si $S\backslash S_p=\{v_1\}$, o\`u $v_1$ est une place de $F^+$ comme en \cite[\S\ 2.3]{CEGGPS}, et $U^p$ est choisi comme dans \emph{loc.~cit.}, alors $\mathfrak{X}^\square_{\rhobar^p}$ est lisse et irr\'eductible.\\
(ii) Il r\'esulte du Th\'eor\`eme \ref{densite} (cf.~la fin de la preuve) qu'une composante irr\'eductible $\mathfrak{X}^p$-automorphe de $X_{\rm tri}^\square(\rhobar_p)$ contient toujours un tel point cristallin. La Conjecture \ref{conjcomp} peut donc se reformuler comme suit : les composantes irr\'eductibles de $X_p(\rhobar)$ sont exactement les composantes $\iota(X)\times\mathfrak{X}^p\times\mathbb{U}^g$ o\`u $\mathfrak{X}^p$ est une composante irr\'eductible de $\mathfrak{X}^\square_{\rhobar^p}$ et $X$ une composante irr\'eductible de $X_{\rm tri}^\square(\rhobar_p)$ telle que $X\cap U_{\rm tri}^\square(\rhobar_p)^{\rm reg}$ contient un point cristallin de poids de Hodge-Tate deux-\`a-deux distincts.
\end{rema}

Dans ce qui suit, si $\lambda\in(\Z^n)^{\Hom(F_v^+,L)}$, on note $\sigma(\lambda)$ la repr\'esentation alg\'ebrique de $\GL_{n,F_v^+}$ de plus haut poids $\lambda$, et on note de m\^eme sa restriction \`a $\GL_n(\mathcal{O}_{F_v^+})$. Fixons un poids ${\bf k}\in\prod_{v\in S_p}(\Z^n)^{\Hom(F_v^+,L)}$ strictement dominant. On d\'efinit alors une repr\'esentation $\sigma_{\bf k}$ de $K_p=\prod_{v\in S_p}G(F_v^+)$ en posant \[\sigma_{\bf k}=\bigotimes_{v\in S_p}\sigma((k_{\tau,i}+i-1)_{\tau\in\Hom(F_v^+,L),1\leq i \leq n}).\]

La conjecture suivante appara\^it essentiellement dans l'article \cite{EmertonGee}.

\begin{conj}\label{modularite}
Soit $\mathfrak{X}^p$ une composante irr\'eductible de $\mathfrak{X}_{\rhobar^p}^\square$. Soit ${\bf k}\in\prod_{v\in S_p}(\Z^n)^{[F_v^+:\Q_p]}$ strictement dominant, l'espace $\Spf(R_{\rhobar_p}^{\square,{\bf k}-{\rm cr}})^{\rm rig}\times\mathfrak{X}^p\times\mathbb{U}^g$ est contenu dans le support du $R_\infty$-module $\Hom_{K_p}(\sigma_{\bf k},\Pi_\infty)'$.
\end{conj}

\begin{rema}
En fait lorsque $S=S_p\sqcup\{v_1\}$ o\`u $v_1$ est une place de $F^+$ v\'erifiant les conditions de \cite[\S5.3]{EmertonGee}, d'apr\`es \cite[Thm.~5.5.2]{EmertonGee}, la Conjecture \ref{modularite} est \'equivalente \`a la Conjecture de Breuil-M\'ezard raffin\'ee dans les cas cristallins \'enonc\'ee  en \cite[Conj.~4.1.6]{EmertonGee}.
\end{rema}

\begin{prop}\label{implicationconj}
La conjecture \ref{modularite} implique la conjecture \ref{conjcomp}.
\end{prop}
\begin{proof}
Soit $\mathfrak{U}^p$ l'ensemble des points lisses de $\mathfrak{X}^p$. Comme $\mathfrak{X}^p$ est r\'eduit, il s'agit d'un ouvert de Zariski, Zariski-dense dans $\mathfrak{X}^p$ par la Proposition \ref{ouvertlisse}. Soit $X$ une composante irr\'eductible de $X_{\rm tri}^\square(\rhobar_p)$. Supposons que $X\cap U_{\rm tri}^\square(\rhobar_p)^{\rm reg}$ contienne un point $x$ cristallin. Son image $\omega(x)\in\widehat{T}_{p,L}^0$ est donc un caract\`ere alg\'ebrique. D'apr\`es le Th\'eor\`eme \ref{ouvertsature}, l'application $\omega$ est lisse au point $x$. Il existe donc un voisinage affino\"ide connexe $U\subset U_{\rm tri}^\square(\rhobar_p)^{\rm reg}\cap X$ de $x$ dont l'image par $\omega$ est un ouvert affino\"ide de $\widehat{T}_{p,L}^0$. En effet, d'apr\`es \cite[Prop.~6.4.22]{Abbes}, l'application $\omega$ est plate en $x$. Quitte \`a se restreindre \`a des voisinages affino\"{\i}des de $x$ et $\omega(x)$, on peut voir $\omega$ comme l'analytifi\'e d'un morphisme plat entre sch\'emas formels $\mathcal{X}\rightarrow\mathcal{Y}$. D'apr\`es \cite[Prop.~5.10.14]{Abbes}, un tel morphisme se factorise en $\mathcal{X}\xrightarrow{g}\mathcal{U}\xrightarrow{j}\mathcal{Y}$ o\`u $j$ est une immersion ouverte et $g$ est fid\`element plate, on conclut alors en observant qu'une application fid\`element plate est surjective (on peut aussi appliquer \cite[Prop.~1.7.8]{Huber}).

Comme de plus l'ensemble des caract\`eres alg\'ebriques strictement dominants dans $\widehat{T}_{p,L}^0$ s'accumule en l'ensemble des caract\`eres alg\'ebriques, il existe ${\bf k}$ strictement dominant tel que $\delta_{\bf k}\in\omega(U)$. Fixons un tel ${\bf k}$. D'apr\`es \cite[Cor.~2.7]{HellmSchrDensity}, on peut choisir le point $y=(y_1,y_2)\in U\times\mathfrak{U}^p$ tel que $\omega(y_1)=\delta_{\bf k}$ et $y_1\in\mathfrak{X}_{\rhobar_p}^{{\bf k}-{\rm cr}}$. Soit $\mathfrak{p}$ un ideal maximal de $R_\infty[\frac{1}{p}]$ dont l'intersection avec $R^{\rm loc}[\frac{1}{p}]$ soit l'id\'eal du point $y$. 

La Conjecture \ref{modularite} implique alors que $\Hom_{K_p}(\sigma_{\bf k},\Pi_\infty[\mathfrak{p}])\neq0$. On d\'eduit alors de \cite[Lemma 4.16]{CEGGPS}, dont la preuve s'\'etend verbatim \`a notre situation, que l'on a un isomorphisme
\[\Hom_{K_p}(\sigma_{\bf k},\Pi_\infty[\mathfrak{p}])\simeq\Hom_{G_p}(\sigma_{\bf k}\otimes_L\pi(\mathfrak{p}),\Pi_\infty[\mathfrak{p}])\]
o\`u $\pi(\mathfrak{p})$ d\'esigne la repr\'esentation lisse irr\'eductible de $G_p$, de la s\'erie principale, telle que $\pi(\mathfrak{p})\otimes|\det|^{\frac{1-n}{2}}$ corresponde au point $\mathfrak{p}\cap R_{\rhobar_p}^{\square,{\bf k}-{\rm cr}}$ par la correspondance de Langlands locale (i.e.~par la correspondance de Langlands locale appliqu\'ee en chaque place $v$ divisant $p$). On d\'eduit de ceci et du calcul du module de Jacquet de $\sigma_{\bf k}\otimes_L\pi(\mathfrak{p})$ que tous les points de $\{(\iota(y_1),y_2)\}\times\mathbb{U}^g$ sont dans le support de $\mathcal{M}_\infty$. Comme $\{y\}\times\mathbb{U}^g$ est un ensemble irr\'eductible de points lisses de $X_{\rm tri}^\square(\rhobar_p)\times\mathfrak{X}_{\rhobar^p}^\square\times\mathbb{U}^g$, il est inclus dans une unique composante irr\'eductible de cet espace. Comme d'apr\`es le Th\'eor\`eme \ref{egalitecomposante} le support de $\mathcal{M}_\infty$ est une union de composantes irr\'eductibles, on en d\'eduit que l'image par $\iota$ de cette composante irr\'eductible est incluse dans $X_p(\rhobar)$. Remarquons \`a pr\'esent que l'ouvert $U$ est connexe, donc cette composante irr\'eductible contient \'egalement le point $x$ de depart. Comme $x$ est \'egalement un point lisse de $X_{\rm tri}^\square(\rhobar_p)\times\mathfrak{X}_{\rhobar^p}^\square\times\mathbb{U}^g$, l'ensemble $X\times\mathfrak{X}^p\times\mathbb{U}^g$ est l'unique composante irr\'eductible contenant ce point, on a donc au final $\iota(X)\times\mathfrak{X}^p\times\mathbb{U}^g\subset X_p(\rhobar)$.
\end{proof}

\begin{prop}\label{reciproqueconj}
La conjecture \ref{conjcomp} implique la conjecture \ref{modularite}.
\end{prop}

\begin{proof}
Soit $\mathfrak{X}^p$ une composante irr\'eductible de $\mathfrak{X}^\square_{\rhobar^p}$ et $\delta_{\bf k}$ un caract\`ere alg\'ebrique strictement dominant de $T_p$. Soit $\mathfrak{X}_p$ une composante irr\'eductible de $\mathfrak{X}_{\rhobar_p}^{\square,{\bf k}-{\rm cr}}$. D'apr\`es \cite[Lem.~4.4]{Cheneviertriangulines}, on peut choisir \[x_p=(x_v)_{v\in S_p}\in\mathfrak{X}_p\subset\mathfrak{X}_{\rhobar_p}^{\square,{\bf k}-{\rm cr}}\simeq\prod_{v\in S_p}\mathfrak{X}_{\rhobar_{\tilde{v}}}^{\square,k_v-{\rm cr}}\] tel que la repr\'esentation cristalline de $\mathcal{G}_{F_{\tilde{v}}}$ associ\'ee \`a $x_v$ soit g\'en\'erique pour tout $v\in S_p$ au sens de la D\'efinition \ref{generique}. Soit $\delta$ le caract\`ere $\delta_{\bf k}\eta$ de $T_p$ o\`u $\eta$ est le caract\`ere lisse non ramifi\'e de la forme $\prod_{v\in S_p}\eta_v$ avec $\eta_v$ l'unique caract\`ere lisse non ramifi\'e de $T_v$ envoyant $(x_1,\dots,x_n)$ sur $\prod_{i=1}^n\varphi_i^{v(x_i)}$ o\`u $(\varphi_1,\dots,\varphi_n)$ est une suite de valeurs propres du Frobenius cristallin de $x_v$. Le Lemme \ref{parametregenerique} implique en particulier que le point $(x_p,\delta)$ appartient \`a $U_{\rm tri}^\square(\rbar)^{\rm reg}$. Sous la Conjecture \ref{conjcomp}, on a donc $\{(x_p,\delta)\}\times\mathfrak{X}^p\times\mathbb{U}^g\subset X_p(\rhobar)$. Notons temporairement \[\mathcal{F}_{\overline{B}_p}^{G_p}(\delta):=\mathcal{F}_{\overline{B}_p}^{G_p}(U(\mathfrak{g}_L)\otimes_{U(\overline{\mathfrak{b}}_L)}(-\lambda),\eta\delta_{B_p}^{-1}),\] o\`u $\lambda=(k_{\tau,i}+i-1)_{\tau\in\Hom(F_v^+,L),1\leq i\leq n}$. Fixons $x\in\{(x_p,\delta)\}\times\mathfrak{X}^p\times\mathbb{U}^g$. Comme $x\in X_p(\rhobar)$, il existe un morphisme non trivial de $\mathcal{F}_{\overline{B}_p}^{G_p}(\delta)$ vers $\Pi_\infty[\mathfrak{p}_x]^{\rm an}$. Pour prouver que le point $y$, image de $x$ dans $\Spf(R_\infty)^{\rm rig}$ est dans le support de $\Hom_{K_p}(\sigma_\delta,\Pi_\infty)'$, il suffit de prouver que ce morphisme se factorise par le quotient localement alg\'ebrique de $\mathcal{F}_{\overline{B}_p}^{G_p}(\delta)$. Or si ce n'\'etait pas le cas, la description des facteurs de Jordan-H\"older de $\mathcal{F}_{\overline{B}_p}^{G_p}(\delta)$ (cf.~\cite[Cor.~4.6]{BreuilAnalytiqueII}) implique qu'il existerait un caract\`ere localement alg\'ebrique non dominant dans le socle de la $T_p$-repr\'esentation $J_{B_p}(\Pi_\infty[\mathfrak{p}_x]^{\rm an})$ et donc un point de $X_p(\rhobar)$ dont la projection dans $X_{\rm tri}^\square(\rhobar_p)$ est de la forme $(x_p,\delta')$ avec $\delta'$ un caract\`ere localement alg\'ebrique de poids non dominant, ce qui contredit le Lemme \ref{parametregenerique}.
\end{proof}

\section{Application aux vari\'et\'es de Hecke}\label{Global}

\subsection{L'espace des repr\'esentations galoisiennes de pente finie}

Rappelons que l'on a d\'efini au \S\ref{automorphe} le Banach $p$-adiques $\widehat{S}(U^p,L)$ de niveau mod\'er\'e $U^p$.

Soit $S$, $\rhobar$ et $\mathcal{S}$ comme aux \S\ref{automorphe} et \S\ref{HTW}, et $\mathfrak{X}_{\rhobar,\mathcal{S}}=\Spf(R_{\rhobar,\mathcal{S}})^{\rm rig}$. On d\'esigne par $Y(U^p,\rhobar)$ la sous-vari\'et\'e analytique rigide de $\mathfrak{X}_{\rhobar,\mathcal{S}}\times\widehat{T}_{p,L}$ support du faisceau coh\'erent $\mathcal{M}_{U^p}$ d\'efini par $J_{B_p}(\widehat{S}(U^p,L)_{\mathfrak{m}}^{\rm an})$. Il s'agit de la d\'efinition de la vari\'et\'e de Hecke (\cite{CheHecke}) propos\'ee dans \cite{Emint}. Les m\'ethodes de \cite[\S3.8]{CheJL} permettent de montrer qu'il s'agit d'un espace rigide analytique r\'eduit.

Notons $R_{\rhobar_S}^{\square}$ l'alg\`ebre $\widehat\bigotimes_{v\in S}R^{\square}_{\rhobar_{\tilde{v}}}$ et $R_{\rhobar_S}^{\bar\square}$ l'alg\`ebre $\widehat\bigotimes_{v\in S}R^{\bar\square}_{\rhobar_{\tilde{v}}}$ et $\mathfrak{X}^\square_{\rhobar_S}=\Spf(R_{\rhobar_S}^{\bar \square})^{\rm rig}$. On a alors une d\'ecomposition en produit $\mathfrak{X}^\square_{\rhobar_S}\simeq\mathfrak{X}_{\rhobar_p}^\square\times\mathfrak{X}_{\rhobar^p}^\square$ et on d\'efinit \[X_{\rm tri}^\square(\rhobar_S)=X_{\rm tri}^\square(\rhobar_p)\times\mathfrak{X}_{\rhobar^p}^\square\subset \mathfrak{X}^\square_{\rhobar_S}\times\widehat{T}_{p,L}.\] On d\'esigne par $X_{\rm tri}^\square(\rhobar_S)^{\rm aut}$ le sous-espace analytique rigide ferm\'e de $X_{\rm tri}^\square(\rhobar_p)\times\mathfrak{X}_{\rhobar^p}^\square$ union des composantes irr\'eductibles $X\times\mathfrak{X}^p$, o\`u $X$ parcourt l'ensemble des composantes irr\'eductibles $\mathfrak{X}^p$-automorphes (cf.~D\'ef.~\ref{compautomorphe}) et $\mathfrak{X}^p$ parcourt l'ensemble des composantes irr\'eductibles de $\mathfrak{X}_{\rhobar^p}^\square$. L'espace $X_{\rm tri}^\square(\rhobar_S)^{\rm aut}$ d\'epend donc des choix faits dans le Th\'eor\`eme \ref{TaylorWiles}.

Ce m\^eme choix donne lieu \`a un morphisme local de $\mathcal{O}_L$-alg\`ebres locales $R_{\rhobar_S}^{\bar\square}\rightarrow R_{\rhobar,\mathcal{S}}$ obtenu par composition de l'inclusion naturelle $R_{\rhobar_S}^{\bar\square}\hookrightarrow R_\infty$ avec le morphisme surjectif $R_\infty\twoheadrightarrow R_{\rhobar,\mathcal{S}}$ du Th\'eor\`eme \ref{TaylorWiles}. Passant aux espaces analytiques rigides associ\'es, on obtient une application $\mathfrak{X}_{\rhobar,\mathcal{S}}\rightarrow\mathfrak{X}_{\rhobar_S}^{\square}$.

Le but de cette partie est de comparer les espaces $Y(U^p,\rhobar)$ et
\[X(\rhobar)_{\rm fs}:=\iota(X^\square_{\rm tri}(\rhobar_S))\times_{\mathfrak{X}^\square_{\rhobar_S}}\mathfrak{X}_{\rhobar,\mathcal{S}}.\] Ce dernier espace est introduit dans \cite{HellmannFS} (avec fs pour \og finite slope\fg).

\begin{rema}
On a $R_{\rhobar,\mathcal{S}}=R_{\rhobar,S}\widehat\otimes_{R^\square_{\rhobar_S}}R^{\bar\square}_{\rhobar_S}$. En particulier on a \[\mathfrak{X}_{\rhobar,\mathcal{S}}=\Spf(R_{\rhobar,S})^{\rm rig}\times_{\Spf(R^\square_{\rhobar_S})^{\rm rig}}\Spf(R^{\bar\square}_{\rhobar_S})^{\rm rig}.\] Chaque $R^{\bar\square}_{\rhobar_{\tilde{v}}}$ est le plus grand quotient r\'eduit et sans $p$-torsion de $R^\square_{\rhobar_{\tilde{v}}}$. On d\'eduit alors de la preuve de \cite[Lem. 3.4.12]{KisinModularity} que $R^{\bar\square}_{\rhobar_S}$ est le plus grand quotient r\'eduit et sans $p$-torsion de $R^\square_{\rhobar_S}$. Ainsi $\Spf(R^\square_{\rhobar_S})^{\rm rig}$ et $\Spf(R^{\bar\square}_{\rhobar_S})^{\rm rig}$ ont m\^eme espace sous-jacent, et donc les espaces analytiques rigides $\Spf(R_{\rhobar,S})^{\rm rig}$ et $\mathfrak{X}_{\rhobar,\mathcal{S}}$ ont m\^eme espace sous-jacent.
\end{rema}

Nous aurons \'egalement besoin des notations suivantes. L'espace analytique rigide $\mathfrak{X}_{\rhobar^p}^\square$, \'etant un produit d'espaces analytiques rigides r\'eduits, est r\'eduit. L'ensemble de ses points lisses forme donc un ouvert au sens de Zariski qui est non vide et Zariski-dense. Notons $\mathfrak{X}_{\rhobar^p}^{\square,{\rm reg}}$ cet ouvert.
Soit \[X(\rhobar)_{\rm fs}^{\rm reg}=\mathfrak{X}_{\rhobar,\mathcal{S}}\times_{\mathfrak{X}_{\rhobar_p}^\square\times\mathfrak{X}_{\rhobar^p}^\square}(\iota(U_{\rm tri}^\square(\rhobar_p)^{\rm reg})\times\mathfrak{X}_{\rhobar^p}^{\square,{\rm reg}}).\] Il s'agit d'un ouvert au sens de Zariski de $X(\rhobar)_{\rm fs}$. On d\'efinit alors $Y(U_p,\rhobar)^{\rm reg}$ le sous-espace analytique ouvert (au sens de Zariski) de $Y(U^p,\rhobar)$ dont l'espace sous-jacent est donn\'e par $Y(U^p,\rhobar)\cap X(\rhobar)_{\rm fs}^{\rm reg}$ et muni de la structure d'espace analytique rigide induite par celle de $Y(U^p,\rhobar)$.

Remarquons que comme $X_{\rm tri}^{\square}(\rhobar_S)$ est un sous-espace analytique ferm\'e de $\mathfrak{X}^{\square}_{\rhobar_S}\times\widehat{T}_{p,L}$, l'espace $X(\rhobar)_{\rm fs}$ est un sous-espace analytique ferm\'e de $\mathfrak{X}_{\rhobar,\mathcal{S}}\times\widehat{T}_{p,L}$.

\begin{theo}\label{compagnons}
Dans $\mathfrak{X}_{\rhobar,\mathcal{S}}\times\widehat{T}_{p,L}$, on a une \'egalit\'e \emph{d'ensembles ferm\'es analytiques} $Y(U^p,\rhobar)=\mathfrak{X}_{\rhobar,\mathcal{S}}\times_{\mathfrak{X}_{\rhobar_S}^\square}\iota(X_{\rm tri}^\square(\rhobar_S)^{\rm aut})$.
\end{theo}

\begin{proof}
L'isomorphisme $R_{\rhobar,\mathcal{S}}\simeq R_\infty/\mathfrak{a}$ se traduit, en termes d'espaces analytiques rigides par un isomorphisme $\mathfrak{X}_{\rhobar,\mathcal{S}}\simeq\Sp(L)\times_{\Spf(S_\infty)^{\rm rig}}(\mathfrak{X}^\square_{\rhobar_S}\times\mathbb{U}^g)$, l'application $\Sp(L)\rightarrow\Spf(S_\infty)^{\rm rig}$ provenant de la sp\'ecialisation $y_i\mapsto0$, ce qui nous permet d'identifier $\mathfrak{X}_{\rhobar,\mathcal{S}}$ \`a un sous-espace analytique rigide ferm\'e de $\mathfrak{X}_\infty\times\widehat{T}_{p,L}$. Par exactitude \`a gauche du foncteur $J_{B_p}$ ainsi que du foncteur de passage aux vecteurs localement analytiques, on a $J_{B_p}(\Pi_\infty^{R_\infty-{\rm an}})[\mathfrak{a}]\simeq J_{B_p}(\Pi_\infty[\mathfrak{a}]^{\rm an})$. L'isomorphisme \eqref{isopatching} montre alors que le faisceau coh\'erent $\mathcal{M}_{U^p}$ sur $\mathfrak{X}_{\rhobar,\mathcal{S}}\times\widehat{T}_{p,L}^0$ est l'image inverse du faisceau coh\'erent $\mathcal{M}_\infty$ sur $\mathfrak{X}_\infty\times\widehat{T}_{p,L}^0$. On en conclut que le support de $\mathcal{M}_{U^p}$ \emph{en tant que partie ferm\'ee analytique} co\"incide avec l'image inverse du support de $\mathcal{M}_\infty$, c'est-\`a-dire est le ferm\'e analytique sous-jacent \`a $\Sp(L)\times_{\Spf(S_\infty)^{\rm rig}}X_p(\rhobar)$. Il suffit \`a pr\'esent de remarquer que $X_p(\rhobar)=\iota(X_{\rm tri}^\square(\rhobar_S)^{\rm aut})\times\mathbb{U}^g$ par le Th\'eor\`eme \ref{egalitecomposante} et que
\begin{align*}
\Sp(L)\times_{\Spf(S_\infty)^{\rm rig}}(\iota(X_{\rm tri}^\square(\rhobar_S)^{\rm aut})\times\mathbb{U}^g)&\simeq\Sp(L)\times_{\Spf(S_\infty)^{\rm rig}}\mathfrak{X}_{\rhobar_S}^\square\times_{\mathfrak{X}_{\rhobar_S}^\square}(\iota(X_{\rm tri}^\square(\rhobar_S)^{\rm aut})\times\mathbb{U}^g)\\
&\simeq(\Sp(L)\times_{\Spf(S_\infty)^{\rm rig}}\mathfrak{X}_\infty)\times_{\mathfrak{X}_{\rhobar_S}^\square}\iota(X_{\rm tri}^\square(\rhobar_S)^{\rm aut})\\
&\simeq\mathfrak{X}_{\rhobar,\mathcal{S}}\times_{\mathfrak{X}_{\rhobar_S}^\square}\iota(X_{\rm tri}^\square(\rhobar_S)^{\rm aut}).
\end{align*}
\end{proof}

\begin{coro}
La conjecture 4.16 de \cite{HellmannFS} est vraie lorsque $\rhobar$ v\'erifie les hypoth\`eses des \S\ref{automorphe} et \ref{HTW}. 
\end{coro}

Nous allons d\'esormais montrer que la Conjecture \ref{conjcomp} implique une partie des conjectures de \cite{BreuilAnalytiqueII}. Pour cela, pla\c{c}ons dans le contexte de \emph{loc.~cit.} et supposons donc que $F_v^+=\Q_p$ pour tout $v\in S_p$. Remarquons qu'il n'y a pas de difficult\'e non typographique \`a \'etendre \cite{BreuilAnalytiqueII}, ainsi que le Corollaire \ref{remBreuilAnalytiqueII} ci-dessous, au cas g\'en\'eral.

Soit $\pi$ une repr\'esentation automorphe de $G$ qui v\'erifie les conditions suivantes :\\
(i) $\pi^{U^p}\neq0$, en particulier $\pi$ est non ramifi\'ee hors de $S$ ;\\
(ii) pour $v\in S_p$, $\rho_{\pi,\tilde{v}}:=\rho_\pi|_{\mathcal{G}_{\tilde{v}}}$ est cristalline g\'en\'erique au sens de \cite[D\'ef.~5.2]{BreuilAnalytiqueII}\footnote{Attention, ce n'est pas la D\'efinition \ref{generique} du pr\'esent article.};\\
(iii) $\rhobar_\pi$ est isomorphe \`a $\rhobar$.

On note $\mathfrak{p}_\pi$ l'id\'eal maximal de $R_{\rhobar,S}[1/p]$ correspondant \`a $\rho_\pi$.

\begin{coro}\label{remBreuilAnalytiqueII}
Supposons que $p\geq2n+2$ et $\rhobar|_{\mathcal{G}_{F(\zeta_p)}}$ est absolument irr\'eductible. Alors la partie \og existence\fg\, de la Conjecture 6.5 de \cite{BreuilAnalytiqueII} est vraie dans le cas cristallin. Plus pr\'ecis\'ement, pour tout caract\`ere $\eta=\prod_{v\in S_p}\eta(w_{\tilde{v}}^{\rm alg},w_{\tilde{v}})$ avec $w_{\tilde{v}}^{\rm alg}\leq w_{\tilde{v}}^{\rm alg}(w_{\tilde{v}})$ (les notations sont celles de \cite{BreuilAnalytiqueII}), on a
\[\Hom_{T_p}(\eta,J_{B_p}(\widehat{S}(U_p,L)^{\rm an}[\mathfrak{p}_\pi]))\neq0.\] 
\end{coro}

\begin{rema}\label{temperee}
La condition (ii) peut \^etre remplac\'ee par la condition plus faible suivante :\\
(ii)' pour tout $v\in S_p$, $\rho_{\pi,\tilde{v}}$ est cristalline, pour tout plongement de $F_v^+$ dans $L$, les poids de Hodge-Tate de $\rho_{\pi,\tilde{v}}$ sont deux-\`a-deux distincts, et si $\delta_v$ est un param\`etre de $\rho_{\pi,\tilde{v}}$, on a $\delta_{v,i}\neq\delta_{v,j}$ pour $i\neq j$. En effet, par compatibilit\'e local-global (cf.~\cite[Thm.~1.1]{caraianil=p}) on sait que $\pi_v$ est isomorphe \`a une repr\'esentation lisse de la forme $\Ind_{B_v}^{\GL_n(F_v^+)}(\eta_v)$ avec $\eta_{v,i}\eta_{v,j}^{-1}\neq|\cdot|_{F_v^+}^{i-j+1}$. C'est une cons\'equence du fait que $\pi_v$ est une repr\'esentation temp\'er\'ee de $\GL_n(F_v^+)$ (cf.~\cite[Thm.~1.2]{Caraiani}) ainsi que de la classification des repr\'esentations temp\'er\'ees non ramifi\'ees (\cite[\S2.1 et 2.2]{KudlaLocLan}).
\end{rema}

\begin{proof}
Soit $(\rho_\pi,\iota(\delta))\in Y(U^p,\rhobar)$ avec $\rho_\pi$ automorphe v\'erifiant les hypoth\`eses (i) \`a (iv) de \cite[\S6]{BreuilAnalytiqueII}, en particulier $\rho_\pi$ est cristalline aux places de $S_p$ et $\delta$ localement alg\'ebrique strictement dominant, i.e.~de la forme $\delta=\prod_{v\in S_p}(\delta_{{\bf k}_v}\eta_v)$ avec $\delta_{k_v}$ strictement dominant (au sens du d\'ebut du \S\ref{triangulines} avec $K=F^+_v$) et $\eta_v$ un caract\`ere lisse de $T_v$. Rappelons que, si $k_v=(k_{\tau,i})_{1\leq i\leq n, \tau:F_v^+\hookrightarrow L}\in(\mathbb{Z}^n)^{\mathrm{Hom}(F_v^+,L)}$, alors il existe $w_v=(w_\tau)_{\tau:F_v^+\hookrightarrow L}\in S_n^{\mathrm{Hom}(F_v^+,L)}$ tel que $\delta_{w_v(k_v)}\eta_v$ est un param\`etre de $\rho_\pi|_{G_{F_{\tilde{v}}}}$ o\`u $w_v(k_v)=(k_{\tau,w_\tau^{-1}(i)})_{1\leq i\leq n;\tau:F_v^+\hookrightarrow L}$ (cf.~par exemple \cite[Lem.~6.4]{BreuilAnalytiqueII} o\`u l'hypoth\`ese $F_v^+=\mathbb{Q}_p$ ne joue aucun r\^ole). En particulier, le point $(\rho_\pi,(\rho_\pi|_{G_{F_{\tilde{v}}}},\iota_{F_v^+}(\delta_{w_v(k_v)}\eta_v))_{v\in S_p}, (\rho_\pi|_{G_{F_{\tilde{v}}}})_{v\in S_p})$ est dans
 \[\mathfrak{X}_{\rhobar,\mathcal{S}}\times_{\mathfrak{X}^\square_{\rhobar_S}}(\iota(U_{\rm tri}^\square(\rhobar_p)^{\rm reg})\times\mathfrak{X}_{\rhobar^p}^\square)\subset \mathfrak{X}_{\rhobar,\mathcal{S}}\times_{\mathfrak{X}^\square_{\rhobar_S}}X_{\rm tri}^\square(\rhobar_S).\] 
 Si l'on suppose la Conjecture \ref{conjcomp}, il est donc dans $\mathfrak{X}_{\rhobar,\mathcal{S}}\times_{\mathfrak{X}^\square_{\rhobar_S}}X_{\rm tri}^\square(\rhobar_S)^{\rm aut}$. Par le Th\'eor\`eme \ref{compagnons}, on en d\'eduit que le point $(\rho_\pi, \iota(\prod_{v\in S_p}\delta_{w_v(k_v)}\eta_v))$ est dans $Y(U^p,\rhobar)$. Par \cite[Prop.~8.1]{BreuilAnalytiqueII} appliqu\'e avec $\Sigma_p=S_p=\{v|p\}$ (loc.~cit. suppose $F_v^+=\mathbb{Q}_p$, mais ce point est inutile dans la preuve ici aussi), on en d\'eduit que $Y(U^p,\rhobar)$ contient tous les points $(\rho_\pi, \iota(\prod_{v\in S_p}\delta_{w_v'(k_v)}\eta_v))$ pour tout $w_v'=(w_\tau')_{\tau:F_v^+\hookrightarrow L}$ dans $S_n^{\mathrm{Hom}(F_v^+,L)}$ tel que $w_\tau'\leq w_\tau$ pour tous $\tau$ et $v$ (pour l'ordre de Bruhat sur $S_n$). En particulier, lorsque $F_v^+=\mathbb{Q}_p$ pour tout $v\in S_p$, on voit que la Conjecture \ref{conjcomp} implique que tous les points compagnons de $(\rho_\pi,\iota(\delta))$ pr\'evus par \cite[Conj.~6.5]{BreuilAnalytiqueII} pour $\Sigma_p=S_p$ (aux changements de notations pr\`es) sont bien dans $Y(U^p,\rhobar)$.
\end{proof}

\subsection{Un r\'esultat de modularit\'e infinit\'esimale}

Soit ${\bf k}\in\prod_{v\in S_p}(\Z^n)^{[F_v^+:\Q_p]}$ un poids strictement dominant. Fixons $\pi$ une repr\'esentation automorphe de $G$ v\'erifiant les conditions suivantes :\\
(i) $\pi^{U^p}\neq0$ ;\\
(ii) pour $v\in S_p$, $\rho_{\pi,\tilde{v}}:=\rho_\pi|_{\mathcal{G}_{\tilde{v}}}$ est cristalline de poids ${\bf k}_v$ et admet un raffinement non critique;\\
(iii) pour $v\in S_p$, les valeurs propres du Frobenius cristallin de $\rho_{\pi,\tilde{v}}$ sont deux-\`a-deux distinctes ;\\
(iv) $\rhobar_\pi$ est isomorphe \`a $\rhobar$.
 
Rappelons que l'on note $\overline{R}_\infty^{\bf k}$ l'image de $R_\infty$ dans l'anneau des endomorphismes $S_\infty$-lin\'eaires de $\Hom_{K_p}(\sigma_{\bf k},\Pi_\infty)$. D'apr\`es \cite[Lem.~4.17]{CEGGPS}, le $R_\infty$-module $\Hom_{K_p}(\sigma_{\bf k},\Pi_\infty)'$ est un $\overline{R}_\infty^{\bf k}[1/p]$-module projectif de type fini. Comme par ailleurs $\Hom_{K_p}(\sigma_{\bf k},\Pi_\infty)'$ est un $S_\infty[1/p]$-module projectif de type fini, l'anneau $\overline{R}_\infty^{\bf k}[1/p]$ est une $S_\infty[1/p]$-alg\`ebre finie et plate. 

La repr\'esentation $\rho_\pi$ correspond alors \`a un point de $\mathfrak{X}_{\rhobar,\mathcal{S}}$. Notons $y$ son image dans $\Spf(R_\infty)^{\rm rig}$ ainsi que $(y_p,y^p,y_\infty)$ ses coordonn\'ees dans la d\'ecomposition \[\Spf(R_\infty)^{\rm rig}\simeq\mathfrak{X}_{\rhobar}^{\square,{\bf k}}\times\mathfrak{X}^\square_{\rhobar^p}\times\mathbb{U}^g.\] D'apr\`es \cite[Thm.~1.2]{Caraiani}, la repr\'esentation $\pi$ est temp\'er\'ee en toute place de $F^+$. En particulier, on d\'eduit de \cite[Lem.~1.3.2.(1)]{BLGGT} que le point $y^p$ est un point lisse de $\mathfrak{X}^\square_{\rhobar^p}$, il appartient donc \`a une unique composante irr\'eductible de cet espace. Fixons $\chi=\bigotimes_{v\in S_p}\chi_v$ un caract\`ere lisse non ramifi\'e de $T_p$ tel que $\delta_v:=\delta_{{\bf k}_v}\chi_v$ est un param\`etre non critique de $\rho_{\pi,\tilde{v}}$ pour tout $v\in S_p$, et donc, en posant $\delta=\bigotimes_{v\in S_p}\delta_v$, que $(\rho_\pi,\delta)\in Y(U^p,\rhobar)^{\rm reg}$. Si $F_v^+=\Q_p$, on d\'eduit de \cite[Prop.~6.4.8]{BelChe} et de la remarque \ref{temperee} qu'il existe toujours un tel caract\`ere et la partie non critique de l'hypoth\`ese (ii) ci-dessus est superflue. En particulier, pour tout $v\in S_p$, on a $(\rho_{\pi,\tilde{v}},\delta_v)\in U_{\rm tri}^\square(\rhobar_{\pi,\tilde{v}})^{\rm reg}$.

\begin{lemm}\label{lemetale}
Le sch\'ema $\Spec(\bar{R}_\infty^{\bf k}[1/p])$ est fini \'etale au-dessus d'un ouvert de Zariski de $\Spec(S_\infty[1/p])$. De plus, le point $y$ provenant de $\pi$ appartient \`a cet ouvert.
\end{lemm}

\begin{proof}
Comme $\overline{R}_\infty^{\bf k}[1/p]$ est une $S_\infty[1/p]$-alg\`ebre finie et plate et que $S_\infty[1/p]$ est int\`egre, il suffit de montrer qu'il existe un point ferm\'e de $\Spec(\overline{R}_\infty^{\bf k}[1/p])$ o\`u le morphisme $\Spec(\overline{R}_\infty^{\bf k}[1/p])\rightarrow\Spec(S_\infty[1/p])$ est \'etale. Pour prouver les deux assertions simultan\'ement, il suffit de prouver que tous les points ferm\'es de $\Spec(\overline{R}_\infty^{\bf k}[1/p])$ au-dessus de l'id\'eal d'augmentation $\mathfrak{a}$ de $S_\infty[1/p]$ sont \'etales, autrement dit que l'anneau $\overline{R}_\infty^{\bf k}[1/p]/\mathfrak{a}$ est r\'eduit. Or, $\Hom_{K_p}(\sigma_{\bf k},\Pi_\infty)'/\mathfrak{a}\simeq\Hom_{K_p}(\sigma_{\bf k},\Pi_\infty[\mathfrak{a}])$ est l'espace de formes automorphes $S(U^pK_p,\sigma)_{\rhobar}$. Comme $\Hom_{K_p}(\sigma_{\bf k},\Pi_\infty)'$ est un $\overline{R}_\infty^{\bf k}[1/p]$-module localement libre de type fini et fid\`ele, $\overline{R}_\infty^{\bf k}[1/p]/\mathfrak{a}$ s'identifie \`a la sous-$L$-alg\`ebre des endomorphismes $L$-lin\'eaires de $S(U^pK_p,\sigma)_{\rhobar}$ engendr\'ee par les op\'erateurs de Hecke hors de $S$ qui est bien semi-simple.
\end{proof}

On note $X_{\rm tri}^\square(\rhobar_p)_{\bf k}$ la fibre $\omega^{-1}(\delta_{\bf k})$ et $U_{\rm tri}^\square(\rhobar_p)^{\rm reg}_{\bf k}$ son intersection avec $U_{\rm tri}^\square(\rhobar_p)^{\rm reg}$. De m\^eme, on note $X_p(\rhobar)^{\rm reg}$ et $X_p(\rhobar)^{\rm reg}_{\bf k}$ les intersections respectives de $X_p(\rhobar)$ avec l'espace $\iota(U_{\rm tri}^\square(\rhobar)^{\rm reg})\times\mathfrak{X}^{\square,{\rm reg}}_{\rhobar^p}\times\mathbb{U}^g$ et $\iota(U_{\rm tri}^\square(\rhobar)^{\rm reg}_{\bf k})\times\mathfrak{X}^{\square,{\rm reg}}_{\rhobar^p}\times\mathbb{U}^g$.

\begin{prop}\label{etale}
(i) L'ensemble $U_{\bf k}^{\rm \'et}\subset X_p(\rhobar)^{\rm reg}_{\bf k}$ des points de $X_p(\rhobar)_{\bf k}^{\rm reg}$ o\`u l'application $\omega_X$ de $X_p(\rhobar)$ dans $\mathcal{W}_\infty$ est \'etale forme un ouvert dense de Zariski de $X_p(\rhobar)^{\rm reg}_{\bf k}$.\\
(ii) Si $z\in Y(U^p,\rhobar)$ est le point associ\'e \`a $(\pi,\iota(\delta))$, alors l'image de $z$ dans $X_p(\rhobar)$ appartient \`a $U_{\bf k}^{\rm \'et}$.
\end{prop}

\begin{proof}
Par changement de base au dessus de l'inclusion canonique d'une fibre $\{\iota(\delta_{\bf k})\}\times\Spf(S_\infty)^{\rm rig}\hookrightarrow \mathcal{W}_\infty$, on sait qu'en un point de $X_p(\rhobar)^{\rm reg}_{\bf k}$ o\`u l'application $\omega_X$ est \'etale, l'application $X_p(\rhobar)^{\rm reg}_{\bf k}\rightarrow\Spf(S_\infty)^{\rm rig}$ (d\'eduite par changement de base) est a fortiori \'etale. Comme l'action de $S_\infty$ sur $\Pi_\infty$ se factorise par $R_\infty$, on a un diagramme commutatif
\begin{equation}\label{fibres}
\begin{aligned}
\xymatrix{X_p(\rhobar)^{\rm reg}_{\bf k}\ar[r]\ar[dr]&\Spf(\overline{R}_\infty^{{\bf k}-{\rm cr}})^{\rm rig}\ar[d]\\
&\Spf(S_\infty)^{\rm rig}.}
\end{aligned}\end{equation}
En effet, si $v\in S_p$, d'apr\`es \cite[Prop.~2.6]{HellmSchrDensity}, le rel\`evement $\rho_{\tilde{v}}$ de $\rhobar_{\tilde{v}}$ correspondant \`a un point de $X_p(\rhobar)^{\rm reg}_{\bf k}$ est potentiellement semi-stable. De plus comme $\rho_{\tilde{v}}$ est trianguline de param\`etre $\delta_v\in\mathcal{T}^n_{\rm reg}$, le quotient de deux valeurs propres distinctes de son Frobenius cristallin n'appartient pas \`a $\{1,q_v\}$, donc $\rho_{\tilde{v}}$ est n\'ecessairement de monodromie nulle, c'est-\`a-dire cristalline.
 
En utilisant le Th\'eor\`eme \ref{egalitecomposante} ainsi que (le raisonnement de) \cite[Prop.~4.8]{HellmSchrDensity}, on montre que la fl\`eche horizontale sup\'erieure est \'etale. En particulier, l'ensemble recherch\'e est contenu dans l'ensemble des points de $\Spf(\overline{R}_\infty^{{\bf k}-{\rm cr}})^{\rm rig}$ o\`u la fl\`eche verticale en \eqref{fibres} est \'etale. Remarquons que ce dernier ensemble co\"incide avec l'ensemble des points ferm\'es de $\Spec(\overline{R}_\infty^{{\bf k}-{\rm cr}}[1/p])$ o\`u $\Spec(\overline{R}_\infty^{{\bf k}-{\rm cr}}[1/p])$ est \'etale au-dessus de $\Spec(S_\infty[1/p])$. D\'efinissons alors $U_{\bf k}^{\rm \'et}\subset X_p(\rhobar)^{\rm reg}_{\bf k}$ comme l'image r\'eciproque de l'ensemble des points ferm\'es de $\Spec(\overline{R}_\infty^{{\bf k}-{\rm cr}}[1/p])$ pour lesquels l'application $\Spec(\overline{R}_\infty^{{\bf k}-{\rm cr}}[1/p])\rightarrow\Spec(S_\infty[1/p])$ est \'etale. Le Lemme \ref{lemetale} (avec l'int\'egrit\'e de $S_\infty[1/p]$) montre alors que pour prouver (i), il suffit de prouver qu'aux points de $U_{\bf k}^{\rm \'et}$, l'application $\omega_X$ de $X_p(\rhobar)$ dans $\mathcal{W}_\infty$ est \'etale.

Soit $x\in U_{{\bf k}}^{\rm \'et}$. Consid\'erons le diagramme suivant
\[\xymatrix{X_p(\rhobar)^{\rm reg}\ar^{\omega_X}[rr]\ar^{\omega}[rd]& &\mathcal{W}_\infty\ar[dl]\\
&\widehat{T}_{p,L}^0&}.\]
Les deux fl\`eches diagonales sont lisses. Celle de droite car il s'agit d'une projection de fibre lisse. Celle de gauche car il s'agit d'une projection de fibre lisse, compos\'ee avec l'application poids $X_{\rm tri}^\square(\rhobar_p)\rightarrow\widehat{T}_{p,L}^0$ qui est lisse au point $y_p$ d'apr\`es le (iii) du Th\'eor\`eme \ref{ouvertsature} et le fait que $X_p(\rhobar)^{\rm reg}$ est un ouvert de $\iota(U_{\rm tri}^\square(\rhobar_p)^{\rm reg})\times\mathfrak{X}^\square_{\rhobar^p}\times\mathbb{U}^g$ (ce qui se d\'eduit des Th\'eor\`emes \ref{ouvertsature} (ii) et \ref{egalitecomposante}). On d\'eduit alors du Lemme \ref{lissite} de l'appendice que l'application horizontale est lisse si et seulement si l'application obtenue par changement de base entre les fibres au-dessus de $\omega(y)$ est lisse. Or, par changement de base, $\omega_X$ devient la fl\`eche diagonale du diagramme \eqref{fibres} et par d\'efinition de $U_{\bf k}^{\rm \'et}$ l'application horizontale est lisse apr\`es changement de base, ce qui nous permet de conclure.

Le point (ii) est alors une application imm\'ediate de (i) et du fait que le point classique $z\in Y(U^p,\rhobar)\subset\iota(X^\square_{\rm tri}(\rhobar_p))\times\mathfrak{X}^\square_{\rhobar^p}\times\mathbb{U}^g$ est en fait dans $\iota(U^\square_{\rm tri}(\rhobar_p)^{\rm reg})\times\mathfrak{X}^{\square,{\rm reg}}_{\rhobar^p}\times\mathbb{U}^g$ (cf.~la discussion pr\'ec\'edant le Lemme \ref{lemetale}).
\end{proof}

On peut \'egalement s'int\'eresser \`a la structure \og sch\'ematique\fg\, de l'espace \\$\mathfrak{X}_{\rhobar,\mathcal{S}}\times_{\mathfrak{X}^\square_{\rhobar_S}}X_{\rm tri}^\square(\rhobar_S)$. Dans cette direction, on a le r\'esultat suivant.

\begin{theo}\label{structreduite}
Les espaces analytiques rigides $\mathfrak{X}_{\rhobar,\mathcal{S}}\times_{\mathfrak{X}^\square_{\rhobar_S}}\iota(X_{\rm tri}^\square(\rhobar_S)^{\rm aut, reg})$ et $Y(U^p,\rhobar)^{\rm reg}$  sont isomorphes.
\end{theo}

\begin{proof}
Comme ces deux espaces rigides analytiques ont le m\^eme espaces sous-jacent, il suffit de prouver qu'ils sont tous les deux r\'eduits. L'espace $Y(U^p,\rhobar)$ est le sous-espace ferm\'e de $\Spf(R_\infty)^{\rm rig}\times\widehat{T}_{p,L}$ d\'efini par l'id\'eal annulateur $\mathcal{I}_{U^p}$ du faisceau coh\'erent $\mathcal{M}_{U^p}$ d\'efini par $J_{B_p}(\widehat{S}(U^p,L)_{\rhobar}^{\rm an})$. Comme ce faisceau est annul\'e par l'id\'eal $\mathfrak{a}$, le faisceau d'id\'eaux $\mathcal{I}_{U^p}$ contient le sous-faisceau d'id\'eaux engendr\'e par $\mathfrak{a}$ et par l'id\'eal annulateur de $\mathcal{M}_\infty$. Ceci implique que l'inclusion $Y(U^p,\rhobar)\subset X(\rhobar)_{\rm fs}$ est en fait un morphisme d'espaces analytiques rigides. D'apr\`es le Th\'eor\`eme \ref{compagnons} les espaces $Y(U^p,\rhobar)^{\rm reg}$ et $\mathfrak{X}_{\rhobar,\mathcal{S}}\times_{\mathfrak{X}_{\rhobar_S}^\square}\iota(X^\square_{\rm tri}(\rhobar_S)^{\rm aut, reg})$ ont m\^eme espace sous-jacent. Il suffit donc de prouver que l'espace 
\[\mathfrak{X}_{\rhobar,\mathcal{S}}\times_{\mathfrak{X}_{\rhobar_S}^\square}\iota(X^\square_{\rm tri}(\rhobar_S)^{\rm aut, reg})=\Sp(L)\times_{\Spf(S_\infty)^{\rm rig}}X_p(\rhobar)^{\rm reg}\]
 est r\'eduit. Il r\'esulte du Corollaire \ref{equidim} et de \cite[\S3.2]{Emint} que les espaces $X_p(\rhobar)$ et $Y(U^p,\rhobar)$ sont tous deux \'equidimensionnels de dimensions respectives $g+[F^+:\Q]\frac{n(n+1)}{2}+|S|n^2$ et $[F^+:\Q]n$. Ainsi $\dim(X_p(\rhobar))-\dim(Y(U^p,\rhobar))$ est \'egal \`a $q$ o\`u $q$ est aussi le cardinal d'une suite de g\'en\'erateurs $y_1,\dots,y_q$ de $\mathfrak{a}$. Par ailleurs, comme $X_p(\rhobar)$ est r\'egulier aux points de $X_p(\rhobar)^{\rm reg}$, il est en particulier de Cohen-Macaulay en ces points, on d\'eduit donc de \cite[Cor.~16.5.6]{EGAIV1} et de l'\'egalit\'e ci-dessus que la suite $y_1,\dots,y_q$ est une suite r\'eguli\`ere sur $X_p(\rhobar_p)^{\rm reg}$ et que $\Sp(L)\times_{\Spf(S_\infty)^{\rm rig}}X_p(\rhobar)^{\rm reg}$ est de Cohen-Macaulay. Ceci implique en particulier que cet espace analytique rigide est sans composante immerg\'ee (cf.~\cite[Prop.~16.5.4]{EGAIV1}). Pour conclure qu'il est r\'eduit, il suffit donc de prouver que chaque composante irr\'eductible contient un point r\'eduit. C'est une cons\'equence de la Proposition \ref{etale} et du fait que pour toute composante irr\'eductible $X$ de $X_p(\rhobar)$, $X\cap X_p(\rhobar)^{\rm reg}$ contient un point de poids alg\'ebrique strictement dominant d'apr\`es le Corollaire \ref{ouvert}.
\end{proof}

\begin{rema}\label{remalissite}
Le point (ii) de la Proposition \ref{etale} ainsi que le Th\'eor\`eme \ref{structreduite} donnent une autre preuve du Th\'eor\`eme 4.10 de \cite{CheFougere}.
\end{rema}

Si $v\in S_p$, d\'esignons par $\mathcal{R}_v$ le raffinement de $\rho_{\pi,\tilde{v}}$ associ\'e au caract\`ere $\chi_v=\chi|_{T_v}$ de $T_v$ et $\mathcal{R}=(\mathcal{R}_v)_{v\in S_p}$. Soit $X_{\rho_{\pi,\tilde{v}}}$ le foncteur des d\'eformations de $\rho_{\pi,\tilde{v}}$ et $X_{\rho_{\pi,\tilde{v}},\mathcal{R}_v}$ le foncteur des d\'eformations de la paire $(\rho_{\pi,\tilde{v}},\mathcal{R}_v)$ tel que d\'efini dans \cite[\S2.5]{BelChe}. D'apr\`es \cite[Prop.~2.5.8]{BelChe}, la transformation naturelle $X_{\rho_{\pi,\tilde{v}},\mathcal{R}_v}\rightarrow X_{\rho_{\pi,\tilde{v}}}$ est relativement repr\'esentable, et on note $R^\square_{\rho_{\pi,\tilde{v}},\mathcal{R}_v}$ l'alg\`ebre locale compl\`ete pro-repr\'esentant le foncteur $X_{\rho_{\pi,\tilde{v}},\mathcal{R}_v}\times_{X_{\rho_{\pi,\tilde{v}}}}\Spf(R^{\bar\square}_{\rho_{\pi,\tilde{v}}})$.

Notons alors $R_{\rho_\pi,S,\mathcal{R}}$ l'anneau pro-repr\'esentant le foncteur des d\'eformations du couple $(\rho_\pi,\mathcal{R})$ d\'efini comme en \cite[\S2.6]{BelChe}. Plus pr\'ecis\'ement il s'agit de l'anneau $R_{\rho_\pi,S}\widehat{\otimes}_{R^\square_{\rho_{\pi,p}}}R^\square_{\rho_{\pi,p},\mathcal{R}}$ avec $R^\square_{\rho_{\pi,p},\mathcal{R}}=\widehat\bigotimes_{v\in S_p}R^\square_{\rho_{\pi,\tilde{v}},\mathcal{R}_v}$, o\`u $\rho_{\pi,S}$ d\'esigne le s\'epar\'e compl\'et\'e du localis\'e de $R_{\rhobar_\pi}[1/p]$ en l'id\'eal maximal associ\'e \`a $\rho_\pi$.

\begin{lemm}\label{globalreduit}
L'anneau $R_{\rho_\pi,S}$ est aussi le s\'epar\'e compl\'et\'e du localis\'e de $R_{\rhobar_\pi,\mathcal{S}}[1/p]$ en l'id\'eal associ\'e \`a $\rho_\pi$.
\end{lemm}

\begin{proof}
Comme $R_{\rhobar_\pi,\mathcal{S}}=R_{\rhobar_\pi,S}\otimes_{R_{\rhobar_{\pi,S}}^\square}R_{\rhobar_{\pi,S}}^{\bar\square}$, il suffit de v\'erifier que pour $v\in S$, le s\'epar\'e compl\'et\'e du localis\'e de $R^\square_{\rhobar_{\pi,\tilde{v}}}$ au point correspondant \`a $\rho_{\pi,\tilde{v}}$ est r\'eduit. Pour $v\in S\backslash S_p$, c'est une cons\'equence de la discussion pr\'ec\'edant le Lemme \ref{etale}. Quant au cas $v\in S_p$, il suffit de v\'erifier que le probl\`eme de d\'eformations de $\rho_{\pi,\tilde{v}}$ est non obstru\'e, autrement dit que $H_{\rm cont}^2(\mathcal{G}_{F_{\tilde{v}}},{\rm ad}(\rho_{\pi,\tilde{v}}))=0$. Par dualit\'e de Tate, il suffit de prouver que $H^0(\mathcal{G}_{F_{\tilde{v}}},{\rm ad}(\rho_{\pi,\tilde{v}}(1)))=0$. Or il suffit de v\'erifier ceci au niveau des $(\varphi,\Gamma)$-modules sur $\mathcal{R}_{L,F_{\tilde{v}}}$, c'est alors une cons\'equence du fait que $\rho_{\pi,\tilde{v}}$ est trianguline de param\`etre dans $\mathcal{T}^n_{\rm reg}$.
\end{proof}

\begin{coro}\label{R=T}
Si $p>2n+1$, la Conjecture 7.6.12 (R=T) de \cite{BelChe} est vraie au point $y=(\rho_\pi,\iota(\delta))$ de $Y(U^p,\rhobar)$. Plus pr\'ecis\'ement, on a un isomorphisme d'anneaux locaux complets
\begin{equation*}
R_{\rho_\pi,S,\mathcal{R}}\xrightarrow{\sim}\widehat{\mathcal{O}}_{Y(U^p,\rhobar),y}.
\end{equation*}
\end{coro}

\begin{proof}
La condition sur les valeurs de $\chi_i(\varpi_v)$ et le fait que $\chi$ corresponde \`a un raffinement non critique de $\pi$, implique que l'image de $z$ dans $X_{\rm tri}^\square(\rhobar_p)$ appartient \`a l'ouvert lisse $U_{\rm tri}^\square(\rhobar_p)^{\rm reg}$ et se situe donc sur une unique composante irr\'eductible de $X_{\rm tri}^\square(\rhobar_p)$. Ceci implique que l'anneau local compl\'et\'e de $Y(U^p,\rhobar)$ en $z$ est isomorphe \`a l'anneau local compl\'et\'e de $X_{\rm fs}(\rhobar)$ en $z$. Or cet anneau local compl\'et\'e est exactement $R_{S,\rho_\pi}\widehat{\otimes}_{R^\square_{\rho_{\pi,p}}}R^\square_{\rho_{\pi,p},\mathcal{R}}$ qui est canoniquement isomorphe \`a $R_{S,\rho_{\pi,p},\mathcal{R}}$.
\end{proof}

\begin{coro}
Si $A$ est une $L$-alg\`ebre locale de dimension finie, toute d\'eformation de $\rho_\pi$ sur $A$ v\'erifiant $\rhobar_\pi^\vee\circ c\simeq\rho_\pi(n-1)$ et cristalline en toute place divisant $p$ est triviale.
\end{coro}
\begin{proof}
D'apr\`es le th\'eor\`eme $4.10$ de \cite{CheFougere} (ou la Remarque \ref{remalissite}), l'application poids est \'etale au point $z=(\pi,\chi)$ de $Y(U^p,\rhobar)$. Le r\'esultat est alors une cons\'equence du corollaire 7.6.14 de \cite{BelChe} et du corollaire \ref{R=T}.
\end{proof}

Le m\^eme raisonnement que \cite[Thm. 8.2]{Kisinoverconvergent} montre alors le r\'esultat suivant sur la g\'eom\'etrie de l'espace $\Spec(R_S(\rhobar_\pi)[1/p])$ au point $\rho_\pi$.

\begin{coro}
L'espace $\Spec(R_S(\rhobar_\pi)[1/p])$ est lisse et de dimension $\frac{n(n+1)}{2}[F^+:\Q]$ au point $\rho_\pi$.
\end{coro}

\begin{proof}
D'apr\`es le Lemme \ref{globalreduit}, on a un isomorphisme \[R_{\rho_\pi,S}\otimes_{R_{\rho_\pi,p}^\square}R_{\rho_\pi,p}^{\square,{\bf k}-{\rm cr}}\simeq R_{\rho_\pi,\mathcal{S}}\otimes_{R_{\rho_\pi,p}^\square}R_{\rho_\pi,p}^{\square,{\bf k}-{\rm cr}}.\] En raisonnant comme dans \cite[Prop.~7.6.4]{BelChe}, on peut alors d\'eduire du Corollaire \ref{R=T} et de \cite[Thm.~4.10]{CheFougere} (ou de la Remarque \ref{remalissite}) que
\[ {\rm Ker}(\Hom_{L-{\rm alg}}(R_{\rho_\pi,S},L[X]/(X^2))\rightarrow \Hom_{L-{\rm alg}}(R_{\rho_{\pi,p}}^{\square,{\bf k}-{\rm cris}},L[X]/(X^2)))=\{0\}.\]
En raisonnant exactement comme dans la preuve de \cite[Thm. 8.2]{Kisinoverconvergent}, on en d\'eduit que $\Hom_{L{\rm -alg}}(R_{\rho_\pi,S},L[X]/(X^2))$, autrement dit l'espace tangent de $\Spec(R_{\rhobar_\pi,S})$ au point $\rho_\pi$ est de dimension $\leq \frac{n(n+1)}{2}[F^+:\Q]$. Comme par ailleurs, d'apr\`es \cite[Cor.~2.2.12]{CHT}, on a $\dim(R_{\rhobar_\pi,S})\geq 1+\frac{n(n+1)}{2}[F^+:\Q]$, on en d\'eduit le r\'esultat.
\end{proof}

\begin{rema}
Sous des hypoth\`eses plus g\'en\'erales, essentiellement en ne supposant pas $\pi$ non ramifi\'ee en $p$, les deux corollaires pr\'ec\'edents ont \'et\'e prouv\'es de fa\c{c}on ind\'ependante par Allen dans \cite{Allen}.
\end{rema}

\section{Appendice}
\subsection{Un peu d'analyse fonctionnelle non archim\'edienne}

Soit $H$ un pro-$p$-groupe uniforme. Fixons $A$ une $\mathcal{O}_L$-alg\`ebre quotient sans $p$-torsion de l'alg\`ebre de groupe compl\'et\'ee $\mathcal{O}_L\dbl H\dbr$ (l'alg\`ebre $A$ n'est donc pas n\'ecessairement commutative). Si $\mathfrak{m}_A$ d\'esigne l'id\'eal maximal de $A$, on note, pour $n\geq1$, $A_n^\circ$ le compl\'et\'e $\mathfrak{m}_A$-adique de la $\mathcal{O}_L$-alg\`ebre $A[\frac{\mathfrak{m}_A^n}{\varpi}]\subset A[\frac{1}{p}]$ et on pose $A_n=A_n^\circ[\frac{1}{p}]$. On d\'efinit alors, comme dans le cas commutatif $A^{\rm rig}=\varprojlim_n A_n$. Notons $D(H)$ l'alg\`ebre des distributions de $H$ et $I$ le noyau de $\mathcal{O}_L\dbl H\dbr\rightarrow A$.

\begin{lemm}\label{quotient}
La $\mathcal{O}_L$-alg\`ebre $A^{\rm rig}$ est une alg\`ebre de Fr\'echet-Stein, qui est isomorphe \`a $D(H)/ID(H)$.
\end{lemm}
\begin{proof}
D'apr\`es le th\'eor\`eme $5.4$ de \cite{STdist}, l'alg\`ebre $D(H)$ est de Fr\'echet-Stein, munie d'une famille de normes $(q_n)_{n\geq1}$ d\'efinie comme suit. On choisit $(h_1,\cdots,h_d)$ une base ordonn\'ee de $H$, et pour tout $\alpha\in\mathbb{N}^d$, on pose $b^{\alpha}=b_1^{\alpha_1}\cdots b_d^{\alpha_d}\in \mathcal{O}_L[H]$ o\`u $b_i=h_i-1$. On pose alors, pour $x=\sum_{\alpha}x_{\alpha}b^{\alpha}$, $q_n(x)=\sup_{\alpha}(|a_\alpha|p^{-|\alpha|/n})$. La boule unit\'e de $\mathcal{O}_L\dbl H\dbr$ pour la norme $q_n$ est donc exactement $\mathcal{O}_L\dbl H\dbr_n^\circ$. Ceci implique que $A_n^\circ$ est la boule unit\'e de $A$ muni de la semi-norme quotient induite par $q_n$. La proposition $3.7$ de \cite{STdist} implique alors que $A^{\rm rig}$ est une alg\`ebre de Fr\'echet-Stein isomorphe \`a $\mathcal{O}_L\dbl H\dbr^{\rm rig}/I\mathcal{O}_L\dbl H \dbr^{\rm rig}=D(H)/D(H)I$.
\end{proof}

\begin{lemm}\label{annexe}
Supposons que $\varphi:\,A\rightarrow B$ soit un morphisme local fini entre deux alg\`ebres locales noeth\'eriennes compl\`ete sans $p$-torsion. Supposons que $B\varphi(\mathfrak{m}_A)=\varphi(\mathfrak{m}_A)B$. Alors le morphisme $A^{\rm rig}\otimes_A B\rightarrow B^{\rm rig}$ est un isomorphisme.
\end{lemm}
\begin{proof}
Remarquons que $A^{\rm rig}\otimes_AB$ est un $A^{\rm rig}$-module coadmissible, on a donc d'apr\`es \cite[Cor. 3.1]{STdist}, un isomorphisme $A^{\rm rig}\otimes_A B\simeq\varprojlim_n A_n\otimes_A B$. En raisonnant comme dans la preuve du lemme 7.2.2 de \cite{DeJong}, on obtient l'isomorphisme
\begin{equation*}
B^{\rm rig}\simeq\varprojlim_n(A_n\widehat{\otimes}_A B).
\end{equation*}
\end{proof}

\subsection{Vecteurs localement analytiques}

Les r\'esultats pr\'esent\'es ici sont essentiellement issus de \cite[\S4]{EmertonJacquetI}. Nous en donnons ici une formulation un peu plus adapt\'ee \`a nos besoins.

Soit $K$ une extension finie de $\Q_p$ et $G$ le groupe des $K$-points d'un groupe r\'eductif quasi-d\'eploy\'e d\'efini sur $K$. Fixons $B$ (le groupe des $K$-points d') un sous-groupe de Borel de $G$. Soit $N$ le radical unipotent de $B$ et $M$ un sous-groupe de Levi de $B$ de sorte que $B=MN$. On note $\overline{B}$ le sous-groupe de Borel oppos\'e \`a $B$ relativement \`a $M$ et $\overline{N}$ son radical unipotent. Soit $M^0$ le sous-groupe compact maximal de $M$, il s'agit d'un tore.

Fixons $q\geq0$ et posons $\widetilde{G}=G\times\mathbb{Z}_p^q$. Si $H$ est un sous-groupe de $G$, on pose $\widetilde{H}=H\times\mathbb{Z}_p^q$. On note $\mathcal{W}$ l'espace des caract\`eres continus de $M^0\times\Z_p^q$. On choisit $H$ un sous-pro-$p$-groupe ouvert et uniforme de $G$, tel que la multiplication \[(\overline{N}\cap H)\times(M\cap H)\times(N\cap H)\rightarrow H\] soit un hom\'eomorphisme. On pose $N_0=N\cap H$, ainsi $N_0$ est stable sous l'action de $M^0$ par conjugaison. Soit $\Pi$ une repr\'esentation unitaire admissible de $\widetilde{G}$ sur un $L$-espace de Banach. On suppose par ailleurs que $\Pi|_{\widetilde{H}}$ est isomorphe \`a $\mathcal{C}(\widetilde{H},L)^m$ pour un certain $m\geq0$. On d\'efinit une action, dite de Hecke, du mono\"ide $M^+=\{z\in M|\, zN_0z^{-1}\subset N_0\}$ sur $\Pi^{N_0}$ par la formule
\begin{equation*}
z\cdot v=\sum_{n\in N_0/zN_0z^{-1}}nzv.
\end{equation*}

De plus on fixe un \'el\'ement $z\in M^+$ tel que $\bigcap_{n\geq 0} z^n N_0z^{-n}=\{1\}$.

\begin{prop}\label{existencefactorisation}
Il existe un recouvrement admissible de $\mathcal{W}$ par des ouverts affino\"{\i}des $U_1\subset U_2\subset\cdots\subset U_h\subset\cdots$ et, pour tout $h\geq1$,
\begin{itemize}
\item un $A_h=\mathcal{O}_{\mathcal{W}}(U_h)$-module de Banach $V_h$ v\'erifiant la condition $(Pr)$ de \cite{Buzzard} ;
\item un endomorphisme $A_h$-compact, not\'e $z_h$, du $A_h$-module $V_h$ (la notion d'endomorphisme $A_h$-compact d'un espace de type $(Pr)$ est d\'efinie dans \cite[\S2]{Buzzard}) ;
\item des applications $A_h$-lin\'eaires continues \[\alpha_h:\,V_h\longrightarrow V_{h+1}\widehat{\otimes}_{A_{h+1}}A_h\ \text{et}\ \beta_h:V_{h+1}\widehat{\otimes}_{A_{h+1}}A_h\rightarrow V_h\] telles que $\beta_h\circ\alpha_h=z_h$ et $\alpha_h\circ\beta_h=z_{h+1}\otimes1_{A_h}$;
\item un isomorphisme topologique $\mathcal{O}(\mathcal{W})$-lin\'eaire $((\Pi^{\rm an})^{N_0})'\simeq\varprojlim_h V_h$ commutant \`a l'action de $z$ \`a gauche et de $(z_h)_{h\geq1}$ \`a droite.
\end{itemize}
On peut r\'esumer toutes ces conditions dans le diagramme commutatif suivant
\begin{equation}
\begin{aligned}
\xymatrix{((\Pi^{\rm an})^{N_0})'\ar[r]\ar_{z}[d]&\cdots\ar[r]&V_{h+1}\ar_{z_{h+1}}[d]\ar[r]&V_{h+1}\widehat{\otimes}_{A_{h+1}}A_h\ar_{z_{h+1}\otimes1_{A_h}}[d]\ar^(.7){\beta_h}[r]&V_h\ar[r]\ar^{z_h}[d]\ar_{\alpha_h}[ld]&\cdots\\
((\Pi^{\rm an})^{N_0})'\ar[r]&\cdots\ar[r]&V_{h+1}\ar[r]&V_{h+1}\widehat{\otimes}_{A_{h+1}}A_h\ar[r]^(.7){\beta_h}&V_h\ar[r]&\cdots}
\end{aligned}
\end{equation}
\end{prop}

\begin{proof}
Pour $\frac{1}{p}<r<1$, on note $D_r(H,L)$ et $D_{<r}(H,L)$ les alg\`ebres de Banach d\'efinies dans \cite[\S3]{STdist}. Si $h\geq1$, posons en suivant \cite[\S0.3]{CD}, $r_h=\frac{1}{p^{h-1}(p-1)}$ et \[B_h=D_{p^{r_h}}(M\cap H,L).\] On note $\Pi_{\widetilde{H}}^{(h)}\subset\Pi$ le sous-espace des vecteurs $r_h$-analytiques de $\Pi$ d\'efini dans \emph{loc.~cit.} Il s'agit d'un $L$-espace de Banach, isomorphe \`a $\mathcal{C}^{(h)}(\widetilde{H},L)^m$ comme repr\'esentation de $\widetilde{H}$, avec les notations de \cite{CD}. On v\'erifie facilement que l'action de $z$ sur $\Pi$ envoie $\Pi_{\widetilde{H}}^{(h)}$ dans $\Pi_{z\widetilde{H}z^{-1}}^{(h)}$. Par ailleurs, l'action de Hecke de $z$ pr\'eserve $\Pi^{N_0}$, on a en fait $z\cdot(\Pi_{\widetilde{H}}^{(h)})\subset\Pi_{z\widetilde{H}z^{-1}N_0}^{(h)}$. En identifiant 
\[(\Pi_{\widetilde{H}}^{(h)})^{N_0}=\mathcal{C}^{(h)}(\overline{N}\cap H,L)^m\widehat{\otimes}_L\mathcal{C}^{(h)}(\widetilde{M}\cap\widetilde{H},L),\] 
ceci prouve que si $v=\sum_{n\geq0}a_n\otimes b_n\in\mathcal{C}^{(h)}(\overline{N}\cap H,L)^m\widehat{\otimes}_L\mathcal{C}^{(h)}(\widetilde{M}\cap\widetilde{H},L)$, alors on peut \'ecrire $z(v)=\sum_{n\geq0}a_n'\otimes b_n'$, o\`u $b'_n\in\mathcal{C}^{(h)}(\widetilde{M}\cap\widetilde{H},L)$ et $a'_n$ est la restriction \`a $\overline{N}\cap H$ d'un \'el\'ement de $\mathcal{C}^{(h)}(z(\overline{N}\cap H)z^{-1},L)^m$. D'apr\`es \cite[Thm. 0.5.(ii)]{CD} et la condition $\bigcap_{n\geq0}z^nN_0z^{-n}=\{1\}$, l'application de restriction $\mathcal{C}^{(h)}(z(\overline{N}\cap H)z^{-1},L)\rightarrow\mathcal{C}^{(h)}(\overline{N}\cap H,L)$ se factorise \`a travers $\mathcal{C}^{(h-1)}(\overline{N}\cap H,L)$, ce qui prouve au final que
\begin{equation*}
z((\mathcal{C}^{(h)}(\widetilde{H},L)^m)^{N_0})\subset\mathcal{C}^{(h-1)}(\overline{N}\cap H,L)^m\widehat{\otimes}_L\mathcal{C}^{(h)}(\widetilde{M}\cap\widetilde{H},L).
\end{equation*}
Comme le dual de $\mathcal{C}^{(h)}(\widetilde{M}\cap\widetilde{H},L)$ est l'alg\`ebre $D_{<p^{r_h}}(\widetilde{H}\cap\widetilde{M},L)$ et que pour tout $h$ la fl\`eche d'inclusion $B_{h+1}\rightarrow B_h$ se factorise par $D_{<p^{r_{h+1}}}(\widetilde{H},L)$, on peut poser $V_h=W_h\widehat{\otimes}_LB_h$, o\`u $W_h=(\mathcal{C}^{(h)}(\overline{N}\cap H,L)^m)'$ de telle sorte que l'on ait un isomorphisme topologique de $B_h$-modules
\begin{equation*}
V_h\simeq((\mathcal{C}^{(h)}(\overline{N}\cap H,L)^m)\widehat{\otimes}_L\mathcal{C}^{(h+1)}(\widetilde{M}\cap\widetilde{H},L))'\widehat{\otimes}_{D_{<p^{r_{h+1}}}(\widetilde{H},L)}B_h.
\end{equation*}
Ainsi $V_h$ est un $B_h$-module orthonormalisable. On d\'efinit alors $\alpha_h$ comme la fl\`eche $\alpha_h'\otimes 1_{B_h}$ o\`u $\alpha_h'$ est l'application duale de l'application 
\begin{equation*}
\mathcal{C}^{(h+1)}(\overline{N}\cap H,L)^m\widehat{\otimes}_L\mathcal{C}^{(h+1)}(\widetilde{M}\cap\widetilde{H},L)\hookrightarrow\mathcal{C}^{(h)}(\overline{N}\cap H,L)^m\widehat{\otimes}_L\mathcal{C}^{(h+1)}(\widetilde{M}\cap\widetilde{H},L)
\end{equation*}
induite par l'action de Hecke de $z$, et $\beta_h$ comme $\beta_h'\otimes1_{B_h}$ o\`u $\beta_h'$ est l'application duale de l'application
\begin{equation*}
\mathcal{C}^{(h)}(\overline{N}\cap H,L)^m\widehat{\otimes}_L\mathcal{C}^{(h+1)}(\widetilde{M}\cap\widetilde{H},L)\rightarrow\mathcal{C}^{(h+1)}(\overline{N}\cap H,L)^m\widehat{\otimes}_L\mathcal{C}^{(h+1)}(\widetilde{M}\cap\widetilde{H},M)
\end{equation*}
induite de l'inclusion de $\mathcal{C}^{(h)}(\overline{N}\cap H,L)^m$ dans $\mathcal{C}^{(h+1)}(\overline{N}\cap L)^m$. Cette derni\`ere \'etant une application compacte, l'application $\beta_h$ est $B_h$-compacte. On pose alors $z_h=\beta_h\circ\alpha_h$, qui est $B_h$-compacte puisque $\beta_h$ l'est. Il est alors facile de v\'erifier la formule pour $\alpha_h\circ\beta_h$.

Remarquons \`a pr\'esent que l'action du groupe $M^0$ par conjugaison pr\'eserve le groupe $H$, ainsi que tous les espaces $\Pi^{(h)}_{\widetilde{H}}$. Ainsi chaque $V_h$ est muni d'une action $B_h$-lin\'eaire de $M^0$. En particulier, chaque espace $V_h$ est naturellement un $L[M^0]$-module. Comme le groupe $M^0/(M\cap H)$ est fini, chaque $V_h$ est un $L[M^0]\otimes_{L[M\cap H]}B_h$-module de Banach. Comme $L[M^0]\otimes_{L[M\cap H]}B_h$ est une $B_h$-alg\`ebre finie \'etale, le $L[M^0]\otimes_{L[M\cap H]}B_h$-module de Banach $V_h$ s'identifie donc \`a un facteur direct du $L[M^0]\otimes_{L[M\cap H]}B_h$-module orthonormalisable $L[M]\otimes_{L[M\cap H]} V_h$, et v\'erifie donc la propri\'et\'e $(Pr)$. Comme $M\cap H$ est un pro-$p$-groupe uniforme commutatif, il est isomorphe \`a une certaine puissance du groupe $\mathbb{Z}_p$, la transform\'ee d'Amice (cf.~\cite[Thm.~2.2]{STFour}) fournit alors un isomorphisme de $L$-alg\`ebre de Banach $B_h\cong \Gamma(U'_h,\mathcal{O}_Y)$ o\`u \[Y=\widehat{\widetilde{M}\cap\widetilde{H}}\]
est l'espace de caract\`eres continus de $(M\cap H)\times\Z_p^q$ et $U'_h$ est un certain ouvert affino\"ide de $Y$, les $U'_h$ formant un recouvrement admissible de $Y$. L'alg\`ebre $L[M]\otimes_{L[M\cap H]}B_h$ est alors isomorphe \`a l'anneau structural de l'ouvert affino\"ide \[U_h\subset \mathcal{W}=\widehat{M^0\times\Z_p^q}\] image r\'eciproque de $U'_h$, ce qui ach\`eve la preuve.
\end{proof}

\subsection{Un lemme sur les syst\`emes de racines de type $A$}

Soit $K$ une extension finie de $\Q_p$ et $L$ et une extension finie de $\Q_p$ telle qu'il existe $[K:\Q_p]$ plongements de $K$ dans $L$. 

Le foncteur $H\mapsto L\times_{\Q_p}\Res_{K/\Q_p}(H)$ de la cat\'egorie des groupes alg\'ebriques sur $K$ vers la cat\'egorie des groupes alg\'ebriques sur $L$ transforme un groupe d\'eploy\'e sur $K$ en un groupe d\'eploy\'e sur $L$.

\begin{prop}\label{racines}
Soit $\lambda$ un poids strictement dominant de $L\times_{\Q_p}\Res_{K/\Q_p}(\GL_{n,K})$. On fixe $P_I\subset\GL_{n,K}$ un sous-groupe parabolique standard, de sous-groupe de Levi $L_I$, et $w$ un \'el\'ement du groupe de Weyl de $L\times_{\Q_p}\Res_{K/\Q_p}(\GL_{n,K})$. Supposons que $w\lambda$ soit un poids dominant de \[L\times_{\Q_p}\Res_{K/\Q_p}(L_I)\cong\prod_{\tau:K\hookrightarrow L}L_{I,\tau}.\] Alors il existe un \'el\'ement $i\notin I$ tel que 
\[\langle w\lambda-\lambda,\tilde{\beta}_i\rangle\leq-\min_{\tau,i}(\lambda_{\tau,i}-\lambda_{\tau,i+1})\]
o\`u $\tilde\beta_i$ est la compos\'ee
\[\mathbb{G}_{m,L}\xrightarrow{\rm diag} L\times_{\Q_p}\Res_{K/\Q_p}\mathbb{G}_{m,K}\xrightarrow{{\rm id}\times\Res(\beta_i)} L\times_{\Q_p}\Res_{K/\Q_p} \GL_{n,K}.\]
\end{prop}
\begin{proof}
Notons $I=I_1\sqcup\cdots\sqcup I_r$ la d\'ecomposition de $I$ en parties orthogonales. Fixons $\tau$ un plongement de $K$ dans $L$. Comme $w_\tau\lambda_\tau$ est un poids strictement dominant de $L_{I,\tau}$, on a $\lambda_{w_\tau^{-1}(i_1)}>\cdots>\lambda_{w_\tau^{-1}(i_s)}$ pour $1\leq j\leq r$ et $I_j=\{i_1<\cdots <i_s\}$. Notons $j_\tau$ le plus petit $j$ tel que $w_\tau^{-1}|_{I_j}$ ne soit pas l'identit\'e. On a alors n\'ecessairement $\lambda_{\tau,w_\tau^{-1}(i)}\leq \lambda_{\tau,i}$ pour $i\in I_{j_\tau}$ et il existe $i_j\in I_{j_\tau}$ tel que $\lambda_{\tau,w_\tau^{-1}(i_j)}<\lambda_{\tau,i_j}$, ce qui implique alors $\lambda_{\tau,w_\tau^{-1}(i)}<\lambda_{\tau,i}$ (et donc $\lambda_{\tau,w_\tau^{-1}(i)}\leq\lambda_{\tau,i+1}$) pour $i\in I_{j_\tau}$ et $i\geq i_j$. On a alors
\begin{equation*}
\langle w_\tau\lambda_\tau-\lambda_\tau,\beta_{|I_1|+\cdots+|I_{j_\tau}|}\rangle=\sum_{i=1}^{|I_{j_\tau}|}(\lambda_{w_\tau^{-1}(i)}-\lambda_{\tau,i})\leq-\min_i(\lambda_{\tau,i}-\lambda_{\tau,i+1}).\end{equation*}
De plus, on a
\begin{equation*}
\langle w_\tau\lambda_\tau-\lambda_\tau,\beta_{|I_1|+\cdots+|I_{j'}|}\rangle=0
\end{equation*}
pour $j'<j_\tau$. En posant $j=\min_{\tau\in\Hom(K,L)}j_\tau$, on v\'erifie facilement que
\begin{equation*}
\langle w\lambda-\lambda,\tilde{\beta}_j\rangle=\sum_{\tau\in\Hom(K,L)}\langle w_\tau\lambda_\tau-\lambda_\tau,\beta_j\rangle\leq-\min_{\tau,i}(\lambda_{\tau,i}-\lambda_{\tau,i+1}).
\end{equation*}
\end{proof}

\subsection{Quelques r\'esultats de g\'eom\'etrie rigide}

Un espace analytique rigide $X$ est dit quasi-Stein (cf. \cite[\S2]{KiehlAB}) si il existe un recouvrement admissible de $X$ par des ouverts affines $U_1\subset U_2\subset\cdots\cdots U_n\subset\cdots$ tel que pour tout $i$, l'image de $\mathcal{O}_X(U_{i+1})$ dans $\mathcal{O}_X(U_i)$ soit dense. On v\'erifie imm\'ediatement qu'un sous-espace analytique ferm\'e d'un espace quasi-Stein est quasi-Stein.

\begin{lemm}\label{qs1}
Soit $\mathcal{M}$ est un faisceau coh\'erent sur $X$ espace quasi-Stein, muni d'un recouvrement admissible comme ci-dessus. On note $M$ l'espace de ses sections globales que l'on munit de la topologie de la limite projective donn\'ee par la formule $M=\varprojlim_i\Gamma(U_i,\mathcal{M})$. Alors pour tout ouvert admissible $U$ de $X$, l'application \[\mathcal{O}_X(U)\widehat{\otimes}_{\mathcal{O}(X)}M\rightarrow\Gamma(U,\mathcal{M})\] est un isomorphisme topologique.
\end{lemm}

\begin{proof}
L'alg\`ebre $\mathcal{O}_X(X)$ est une alg\`ebre de Fr\'echet-Stein, ainsi \cite[Cor. 3.1]{STdist} montre que, pour tout $i$, on a $\Gamma(U_i,\mathcal{M})\simeq\mathcal{O}_X(U_i)\otimes_{\mathcal{O}(X)}M$. Soit $V$ un ouvert affino\"ide de $X$. Par admissibilit\'e du recouvrement, il existe $i$ tel que $V\subset U_i$, ainsi on en d\'eduit que 
\[\Gamma(V,\mathcal{M})\simeq\mathcal{O}_X(V)\otimes_{\mathcal{O}_X(U_i)}\Gamma(U_i,\mathcal{M})\simeq\mathcal{O}_X(V)\otimes_{\mathcal{O}(X)}M.\] 
En utilisant, \cite[Prop. 9.1.4.2.(i)]{BGR}, $U$ poss\`ede un recouvrement admissible par des ouverts affino\"{\i}des. On conclut en utilisant la formule $\Gamma(U,\mathcal{M})\simeq\varprojlim_{V\subset U}\Gamma(V,\mathcal{M})$ la limite projective \'etant prise sur tous les ouverts affino\"{\i}des contenus dans $U$.
\end{proof}

\begin{lemm}\label{qs2}
Soit $f:\,X\rightarrow Y$ un morphisme d'espaces analytiques quasi-Stein. Si $V$ est un ouvert admissible de $Y$, alors \[\mathcal{O}_X(f^{-1}(V))\simeq\mathcal{O}(X)\widehat{\otimes}_{\mathcal{O}(Y)}\mathcal{O}_Y(V).\]
\end{lemm}

\begin{proof}
On fixe $(U_i)_{i\geq1}$ et $(V_i)_{i\geq1}$ des recouvrements admissibles de $X$ et $Y$ v\'erifiant la propri\'et\'e ci-dessus. Par admissibilit\'e du recouvrement $(V_i)_{i\geq1}$, on peut, quitte \`a en extraire une sous-suite, supposer que pour tout $i\geq1$, on a $f(U_i)\subset V_i$. Soit $f_i:U_i\rightarrow V_i$ le morphisme induit. La propri\'et\'e de faisceau nous donne un isomorphisme $\mathcal{O}_X(f^{-1}(V))\simeq\varprojlim_{i}\mathcal{O}_X(f_i^{-1}(V\cap V_i))$. On peut donc supposer $X$ et $Y$ affino\"{\i}des.
On utilise alors la propri\'{e}t\'{e} de faisceau pour \'ecrire $\mathcal{O}_X(f^{-1}(V))\simeq\varprojlim_{U\subset V}\mathcal{O}_X(f^{-1}(U))$, o\`u la limite projective est prise sur les ouverts affino\"ide contenus dans $V$. Lorsque $X$, $Y$ et $U$ sont affino\"{\i}des, on a alors $\mathcal{O}_X(f^{-1}(U))\simeq\mathcal{O}(X)\widehat{\otimes}_{\mathcal{O}(Y)}\mathcal{O}_Y(U)$.
\end{proof}

\begin{lemm}\label{lissite}
Soit $X$, $Y$ et $Z$ trois espaces rigides analytiques lisses, ainsi que $f:\,X\rightarrow Y$, $g:\,Y\rightarrow Z$ et $h=g\circ f$ des applications rigides analytiques. Si $x\in X$, que $h$ est lisse au point $x$ et $g$ lisse au point $f(x)$. Alors l'application $f$ est lisse au point $x$ si et seulement si l'application $f_{h(x)}$ de $X\otimes k(h(x))$ dans $Y\otimes k(h(x))$ est lisse au point $x$.
\end{lemm}

\begin{proof}
Il suffit de prouver que si $f_{h(x)}$ est lisse en $x$ alors il en est de m\^eme de $x$, l'implication inverse r\'esultat de la pr\'eservation de la lissit\'e par changement de base. Supposons donc $f_{h(x)}$ lisse. Il suffit de prouver que l'application $df_x$ sur les plans tangents $T_xX\rightarrow T_{f(x)}Y$ est surjective. Dans le diagramme commutatif ci-dessous, les lignes sont exactes, ce qui r\'esulte de la lissit\'e de $g$ et $h$.
\[\xymatrix{0\ar[r]&T_x(X\otimes k(h(x)))\ar[r]\ar^{df_{h(x)}}[d]&T_xX\ar[r]\ar^{df}[d]&T_{h(x)}Z\ar^{=}[d]\ar[r]&0\\
0\ar[r]&T_{f(x)}(Y\otimes k(h(x)))\ar[r]&T_{f(x)}Y\ar[r]&T_{h(x)}Z\ar[r]&0}\]
La lissit\'e de $df_{h(x)}$ implique alors la surjectivit\'e de la fl\`eche verticale de gauche. Comme la fl\`eche verticale de droite est l'identit\'e, le lemme des cinq implique que la fl\`eche verticale du milieu est surjective.
\end{proof}

\subsection{Descente de la lissit\'e en g\'eom\'etrie rigide}

\begin{lemm}\label{lissite2}
Soit $X$, $Y$, $Z$ trois espaces rigides analytiques sur $\Q_p$, ainsi que $f:\,X\rightarrow Y$ et $g:\,Y\rightarrow Z$ des applications rigides analytiques. Posons $h=g\circ f$. Soit $x\in X$, $y=f(x)$ et $z=g(y)$. Si $f$ est plat et $h$ est lisse en $x$, alors $g$ est lisse en $y$.
\end{lemm}

\begin{proof}
D'apr\`es \cite[Prop.~5.10.3.(ii)]{Abbes}, le morphisme $g$ est plat en $y$. D'apr\`es \cite[Prop.~6.4.22]{Abbes}, il suffit de prouver que $Y\otimes_Z k(z)$ est lisse en $y$. On est donc ramen\'e \`a consid\'erer la situation suivante : $Z=\Sp(k)$ pour $k$ une extension finie de $\Q_p$, $X$ et $Y$ sont deux espaces rigides sur $k$ et $f:\, X\rightarrow Y$ est un morphisme plat en $x\in X$. Il suffit de prouver que si $X$ est $k$-lisse en $x$, $Y$ est $k$-lisse en $y=f(x)$. Soit $\mathcal{O}$ l'anneau des entiers de $k$. L'assertion \'etant locale en $x$, on peut supposer que $X=\Spf(B)^{\rm rig}$ et $Y=\Sp(A)^{\rm rig}$ avec $A$ et $B$ deux $\mathcal{O}$-alg\`ebres topologiquement de type fini sur $\mathcal{O}$. On note $\mathfrak{p}_y\subset A$ et $\mathfrak{p}_x\subset B$ les id\'eaux premiers correspondants aux points $y$ et $x$. D'apr\`es \cite[Prop.~6.4.12]{Abbes}, $X$ est lisse en $x$ si et seulement si l'anneau local $B_{\mathfrak{p}_x}$ est une $k$-alg\`ebre formellement lisse (\cite[Prop.~1.14.5]{Abbes}). Comme $k$ est de caract\'eristique $0$, l'anneau local $B_{\mathfrak{p}_x}$ est une $k$-alg\`ebre formellement lisse si et seulement si c'est un anneau r\'egulier (\cite[Prop.~Thm.~19.6.4]{EGAIV1}). Comme de plus $f$ est plat en $x$, on d\'eduit de \cite[Prop.~5.5.1]{Abbes} que l'anneau $B_{\mathfrak{p}_x}$ est une $A_{\mathfrak{p}_y}$-alg\`ebre plate. Ainsi si $X$ est $k$-lisse en $x$, l'anneau $B_{\mathfrak{p}_x}$ est r\'egulier et la Proposition \cite[17.3.3.(i)]{EGAIV1} permet alors de conclure que $A_{\mathfrak{p}_y}$ est un anneau local r\'egulier. On d\'eduit alors comme pr\'ec\'edemment que $Y$ est lisse sur $k$ en $f(x)$.
\end{proof}

\def\cprime{$'$} \def\cprime{$'$} \def\cprime{$'$} \def\cprime{$'$}
\providecommand{\bysame}{\leavevmode ---\ }
\providecommand{\og}{``}
\providecommand{\fg}{''}
\providecommand{\smfandname}{\&}
\providecommand{\smfedsname}{\'eds.}
\providecommand{\smfedname}{\'ed.}
\providecommand{\smfmastersthesisname}{M\'emoire}
\providecommand{\smfphdthesisname}{Th\`ese}

\end{document}